\documentclass[11pt]{amsart}
\usepackage{amssymb, amsmath, amsfonts, amsthm, graphics,xcolor}
\usepackage{MnSymbol}
\usepackage[hmargin=1 in, vmargin = 1 in]{geometry}
\usepackage{youngtab} 
\usepackage[all]{xy}

\usepackage{caption}
\usepackage{subcaption}

\newcommand{\newword}[1]{\textbf{\emph{#1}}}

\usepackage{tikz}
\usetikzlibrary{arrows,decorations.pathmorphing,decorations.markings,backgrounds,positioning,fit,patterns,calc}
\tikzstyle{mutable}=[inner sep=0.5mm,circle,draw,minimum size=2mm]
\tikzstyle{frozen}=[inner sep=.9mm,rectangle,draw]
\tikzstyle{dot} = [fill=black!25,inner sep=0.5mm,circle,draw,minimum size=1mm]
\tikzstyle{blue dot} = [draw=blue,fill=blue!25,inner sep=0.5mm,circle,draw,minimum size=1mm]
\tikzstyle{marked}=[inner sep=0.5mm,circle,draw,blue!75!black,fill=blue!50]
\tikzstyle{outline}=[thick,line width=1.5mm,draw=black!10]
\tikzstyle{oriented}=[draw=red,thick,decoration={markings,mark=at position 0.52 with {\arrow{>}}},postaction={decorate}]
\tikzstyle{antioriented}=[draw=red,thick,decoration={markings,mark=at position 0.52 with {\arrow{<}}},postaction={decorate}]
\tikzstyle{faded oriented}=[draw=black!25,thick,decoration={markings,mark=at position 0.52 with {\arrow{>}}},postaction={decorate}]\tikzstyle{invisible}=[inner sep=-.3, minimum size=-.3]
\tikzstyle{matching}=[line width=1.5pt,blue]

\makeatletter
\newcommand{\xleftrightarrow}[2][]{\ext@arrow 3359\leftrightarrowfill@{#1}{#2}}
\newcommand{\xdashrightarrow}[2][]{\ext@arrow 0359\rightarrowfill@@{#1}{#2}}
\newcommand{\xdashleftarrow}[2][]{\ext@arrow 3095\leftarrowfill@@{#1}{#2}}
\newcommand{\xdashleftrightarrow}[2][]{\ext@arrow 3359\leftrightarrowfill@@{#1}{#2}}
\def\rightarrowfill@@{\arrowfill@@\relax\relbar\rightarrow}
\def\leftarrowfill@@{\arrowfill@@\leftarrow\relbar\relax}
\def\leftrightarrowfill@@{\arrowfill@@\leftarrow\relbar\rightarrow}
\def\arrowfill@@#1#2#3#4{%
  $\m@th\thickmuskip0mu\medmuskip\thickmuskip\thinmuskip\thickmuskip
   \relax#4#1
   \xleaders\hbox{$#4#2$}\hfill
   #3$%
}
\makeatother

\newcommand{\CC}{\mathbb{C}}
\newcommand{\DD}{\mathbb{D}}

\newcommand{\FF}{\mathbb{F}}
\newcommand{\GG}{\mathbb{G}}

\newcommand{\MM}{\mathbb{M}}

\newcommand{\RR}{\mathbb{R}}

\newcommand{\ZZ}{\mathbb{Z}}

\newcommand{\cI}{\mathcal{I}}

\newcommand{\cM}{\mathcal{M}}

\newcommand{\tPi}{\widetilde{\Pi}}
\newcommand{\tPio}{\tPi^{\circ}}
\newcommand{\Pio}{\Pi^{\circ}}
\newcommand{\tDD}{\widetilde{\DD}}
\newcommand{\Mat}{\mathrm{Mat}}
\newcommand{\Mato}{\Mat^{\circ}}

\newcommand{\veccI}{\vec{\cI}}
\newcommand{\vecI}{\vec{I}}
\newcommand{\cevcI}{\cev{\cI}}
\newcommand{\cevI}{\cev{I}}
\newcommand{\sI}{\overset{\bullet \rightarrow}{I}}
\newcommand{\tI}{\overset{\bullet \leftarrow}{I}}
\newcommand{\sF}{\overset{\bullet \rightarrow}{\FF}}
\newcommand{\tF}{\overset{\bullet \leftarrow}{\FF}}
\newcommand{\vecM}{\overrightarrow{M}}
\newcommand{\cevM}{\overleftarrow{M}}
\newcommand{\rM}{\overrightarrow{\MM}}
\newcommand{\lM}{\overleftarrow{\MM}}
\newcommand{\rpartial}{\overrightarrow{\partial}} 
\newcommand{\lpartial}{\overleftarrow{\partial}} 
\newcommand{\Monom}{X} 
\newcommand{\rt}{\vec{\tau}}
\newcommand{\lt}{\cev{\tau}}
\newcommand{\impto}{\Rightarrow}
\newcommand{\impfrom}{\Leftarrow}

\usepackage{graphicx}
\newcommand{\gmat}[2][ccccccccccccccccccccccccccccccccc]{\left(\begin{array}{#1} #2\\ \end{array}\right)}

\newcommand{\garrow}[1]{\stackrel{#1}{\longrightarrow}}

\newcommand{\cev}[1]{\reflectbox{\ensuremath{\vec{\reflectbox{\ensuremath{#1}}}}}}

\theoremstyle{plain}
\newtheorem{theorem}{Theorem}

\newtheorem{thm}[theorem]{Theorem}
\newtheorem{Theorem}[theorem]{Theorem}
\newtheorem{Proposition}[theorem]{Proposition}

\newtheorem{prop}[theorem]{Proposition}

\newtheorem{corollary}[theorem]{Corollary}
\newtheorem{cor}[theorem]{Corollary}

\newtheorem{lemma}[theorem]{Lemma}
\newtheorem{lem}[theorem]{Lemma}
\newtheorem*{theoremMT}{Theorem~\ref{main theorem}}
\newtheorem*{TheoremPluck}{Theorem~\ref{theorem obey Plucker relations}}
\newtheorem*{PropInverseTorii}{Corollary~\ref{prop: isotori}}
\newtheorem*{PropInCellOpen}{Propositions~\ref{in open cell} and~\ref{inversion}}
\newtheorem*{TheoremInverseTwists}{Theorem~\ref{thm: inversetwists}}
\newtheorem*{CorollaryTwistQuotient}{Corollary~\ref{twist quotient}}

\newcommand{\hexcoor}[3]{($ #1*(0,1) + #2*(.866,-.5) + #3*(-.866,-.5)$)}

\theoremstyle{remark}

\newtheorem{Corollary/Definition}[theorem]{Corollary/Definition}

\newenvironment{myexample}{\refstepcounter{theorem}\begin{proof}[Example \emph{\thetheorem}]}{\end{proof}}
\newenvironment{remark}{\refstepcounter{theorem}\begin{proof}[Remark \emph{\thetheorem}]}{\end{proof}}

\numberwithin{theorem}{section}

\title{The twist for positroid varieties}
\author{Greg Muller and David E Speyer}

\begin{document}
%
%

\begin{abstract}
The purpose of this note is to connect two maps related to certain graphs embedded in the disc. The first is Postnikov's \emph{boundary measurement map}, which combines partition functions of matchings in the graph into a map from an algebraic torus to an \emph{open positroid variety} in a Grassmannian. The second is a rational map from the open positroid variety to an algebraic torus, given by certain Pl\"ucker coordinates which are expected to be a cluster in a \emph{cluster structure}.

This paper clarifies the relationship between these two maps, which has been ambiguous since they were introduced by Postnikov in 2001. The missing ingredient supplied by this paper is a \emph{twist} automorphism of the open positroid variety, which takes the target of the boundary measurement map to the domain of the (conjectural) cluster. Among other applications, this provides an inverse to the boundary measurement map, as well as Laurent formulas for twists of Pl\"ucker coordinates.

\end{abstract}


\maketitle


\section{Introduction and survey of results}~\label{Introduction}


In Section~\ref{sec informal}, we will provide an overview of our results. In Sections~\ref{sec positroid background} through~\ref{sec twist intro}, we will state the necessary definitions as rapidly as possible to give a full statement of our main results in Section~\ref{sec main theorem intro}. These definitions will reappear later with more detail, motivation and context.
In Section~\ref{sec outline}, the reader can find an outline of the rest of the paper.

\subsection{Informal summary} \label{sec informal}

The Grassmannian of $k$-planes in $\mathbb{C}^n$ admits a decomposition into \emph{open positroid varieties} $\Pio(\cM)$, analogous to the decomposition of a semisimple Lie group into double Bruhat cells~\cite{FZ99}.  Postnikov~\cite{Pos} showed that an appropriate choice of \emph{reduced graph} $G$ defines a \emph{boundary measurement map}
\[ (\CC^\times)^{\text{Edges}(G)}/\text{Gauge} \longrightarrow \Pio(\cM)\]
Among other properties, this map can be used to parametrize the `totally positive part' of $\Pio(\cM)$.

Scott~\cite{Sco06} gave a combinatorial recipe which assigns, to each face of the reduced graph, a homogenous coordinate on $\Pio(\cM)$. 
Scott works only with the largest positroid, so that $\Pio(\cM)$ is a dense open subset of $Gr(k,n)$, but her recipe makes sense for any positroid.
 These homogeneous coordinates collectively define a rational coordinate chart, the \emph{face Pl\"ucker map}:
\[ \Pio(\cM) \xdashrightarrow{\qquad} \CC^{\text{Faces(G)}}/\text{Scaling} \]
Despite the fact that these two maps are both defined by the same combinatorial input (a choice of reduced graph), the relation between them has been elusive.

Moreover, the results of Postnikov and Scott are weaker than we have stated in the two proceeding paragraphs. Postnikov only shows that the boundary measurement map exists as a rational map, which is well defined on $(\RR_{>0})^{\text{Edges}(G)}/\text{Gauge}$. Scott only studies the case of the largest positroid; when one turns to other positroids, it is not clear that the coordinates of the face Pl\"ucker map generate the function field of $\Pio(\cM)$. 
In fairness, at the time Postnikov and Scott were working, the algebraic structure on $\Pio(\cM)$ had not been defined, so these questions would have been difficult to formulate.\footnote{Postnikov's manuscript~\cite{Pos} was in private circulation since at least 2001, was placed on the arXiv in 2006, and is yet unpublished. Scott's result was first presented in her dissertation in 2001~\cite{ScoThesis}, and was placed on the arXiv as a separate paper in 2003~\cite{Sco06} (publication date 2006). At the time, positroid cells were defined only as real semi-algebraic sets. Knutson, Lam and Speyer identified the corresponding complex varieties in work that appeared on the arXiv in 2009~\cite{KLS09} and in improved form in 2011~\cite{KLS13} (publication date 2013). The varieties in question had been studied earlier by Lusztig~\cite{Lus98}, Rietsch~\cite{Rie06} and others, but the connection to Postnikov's theory was not made in that earlier work. 
}
However, now that we have such algebraic structures, these omissions form a major gap in our understanding.

In this paper, we relate the two maps by introducing a \emph{twist automorphism} $\rt$ of each open positroid variety.  The main theorem of this paper then states that the composition
\[ (\mathbb{C}^\times)^{\text{Edges}(G)}/\text{Gauge} \longrightarrow \Pio(\cM) \stackrel{\rt}{\longrightarrow} \Pio(\cM) \xdashrightarrow{\qquad} \mathbb{C}^{\text{Faces}(G)}/\text{Scaling} \]
is an isomorphism of algebraic tori. Each coordinate is given by a monomial which is defined by a distinguished \emph{matching} on $G$.  

As a consequence, we deduce that the boundary measurement map is a well defined inclusion from $(\mathbb{C}^\times)^{\text{Edges}(G)}/\text{Gauge}$ to $\Pio(\cM)$. We also learn that the face Pl\"ucker map is well defined on an open torus, and gives rational coordinates on $\Pio(\cM)$.
Thus, we show that the statements of the first two paragraphs are correct after all.
Furthermore, we obtain explicit birational inverses to these maps. 

\subsection{Earlier work}
The most important precedent for our work is that of Marsh and Scott~\cite{MS16}. 
They construct a twist map\footnote{Their twist map differs from ours by a rescaling of the columns; see Remark \ref{rem: MStwist}.} for the largest positroid variety in a Grassmannian, although they only give explicit formulas for the composite map above when $G$ is a certain standard reduced graph known as a Le diagram. 

Talaska~\cite{Tal11} provided a birational inverse to the boundary measurement map for any positroid when $G$ is a Le-diagram; her inverse was not formulated in terms of a twist map and seems unlikely to generalize to other reduced graphs.

A double wiring diagram for a type $A$ double Bruhat cell can be converted to a reduced graph for a corresponding positroid variety. In this setting, the twist map was defined by Berenstein, Fomin and Zelevinsky \cite{BFZ96}, and it was proved that an analogous composite map is an isomorphism of tori  (see Appendix \ref{app: DBC}).

Our result combines and generalizes the above results, to a setting that works for all positroid varieties and all reduced graphs.
We also hope that the unified presentation in this paper clarifies the nature of the previous results.

The authors have had many productive conversations with all the above named mathematicians, and are very grateful to them for their generous assistance.

\subsection{Notations}
We use the following standard notations for combinatorial sets:
\[ [n] := \{1,2,...,n\} \]
\[ \binom{[n]}{k} := \{ I \subset [n] \mid |I| =k\}, \text{ the set of $k$-element subsets of $[n]$.} \]

We will write $\mathbb{G}_m$ for the nonzero complex numbers, considered as an abelian group.  
For any finite set $X$, we write $\CC^X$ for the $\CC$-vector space with basis labeled by $X$, and write $\RR^X$ and $\GG_m^X$ similarly.
We write $Gr(k,n)$ for the Grassmannian of $k$-planes in $\CC^n$.

For a $k\times n$ matrix $A$ and $a\in [n]$, define
\[ A_a := \text{the $a$th column of $A$} \]
Given a $k$-element set $I \subset [n]$, write it as $I=\{i_1<i_2<\cdots < i_k\}$ and define the \newword{$I$th maximal minor} of $A$ by
\[ \Delta_I(A) := \det(A_{i_1},A_{i_2},...,A_{i_k})\]
that is, the determinant of the matrix with columns $A_{i_1},A_{i_2},...,A_{i_k}$.

\subsection{Positroids and positroid varieties} \label{sec positroid background}

The definitions in this section can all be found in Knutson, Lam and Speyer~\cite{KLS13}, and are due either to those authors or to Postnikov~\cite{Pos}.
See Section~\ref{sec positroid} for many alternative formulations of these definitions.

Given a $k$-dimensional subspace $V\subset \CC^n$, the corresponding \emph{matroid} is the collection of $k$-element subsets\footnote{Throughout, a matroid is a collection of  `bases', rather than `independent sets' or other conventions.}
\[ \cM = \{ I\subset [n] \mid \text{the projection $\CC^n \rightarrow \CC^I$ restricts to an isomorphism $V\stackrel{\sim}{\longrightarrow} \CC^I$} \} \]
The Grassmannian $Gr(k,n)$ can then be decomposed into pieces, each parametrizing those subspaces with a fixed matroid.  Unfortunately, this decomposition is incredibly poorly-behaved; its many transgressions are explored elsewhere~\cite{Mnev88},  \cite{Stu87}, \cite{GGMS87}.
We focus on a related decomposition of $Gr(k,n)$ which is much nicer.  

\newword{Positroids} are a special class of matroid with many equivalent characterizations.  The shortest definition~\cite{Pos} is that a positroid is a matroid $\cM$ with a `totally non-negative' representation.  
That is, it is the matroid of the columns of a real matrix whose maximal minors are non-negative real numbers.
Every matroid $\cM$ has a \newword{positroid envelope}; the unique smallest positroid containing $\cM$~\cite[Section 3]{KLS13}. 

Given a positroid $\cM$, the \newword{(open) positroid variety} $\Pio(\cM)$ is the subvariety of $Gr(k,n)$ parametrizing those subspaces whose matroid has positroid envelope $\cM$. 
We obtain a stratification
\[ Gr(k,n) = \bigsqcup_{\stackrel{\text{positroids $\cM$}}{\text{of rank $k$ on $[n]$}}} \Pio(\cM) \]
which groups together matroid strata with the same positroid envelope.
This decomposition of $Gr(k,n)$ arises naturally from several different perspectives and the positroid varieties avoid many of the pathologies exhibited by the matroid strata.

While the Grassmannian and its decomposition are the intrinsically interesting objects, the results of this paper will be most easily stated on the affine cone $\widetilde{Gr(k,n)}$ over the Pl\"ucker embedding of the Grassmannian. 
Denote by $\tPio(\cM)$ the lift of a positroid variety $\Pio(\cM)$ to $\widetilde{Gr(k,n)} \setminus \{ 0 \}$.

We write $\tPi(\cM)$ (respectively $\Pi(\cM)$) for the closure of $\tPio(\cM)$ in $\CC^{\binom{[n]}{k}}$ (respectively, the closure of $\Pio(\cM)$ in $Gr(k,n)$).\footnote{We systematically use the following notational conventions: For some sort of algebraic object $X$, a notation like $X^{\circ}$ will always denote an open dense subvariety of $X$ and $\widetilde{X}$ will always denote something like a cone over $X$. So $\tPio(\cM)$ is a torus bundle over $\Pio(\cM)$, and is open and dense in $\tPi(\cM)$. Similarly, $\Pio(\cM)$ is open and dense in $\Pi(\cM)$, while $\tPi(\cM)$ is the affine cone over the Pl\"ucker embedding of $\Pi(\cM)$.}
The origin of $\CC^{\binom{[n]}{k}}$ is in every $\tPi(\cM)$ and in no $\tPio(\cM)$.

\subsection{The boundary measurement map}

Let $G$ be a graph embedded in a disc, with a 2-coloring of its internal vertices as either black or white (e.g. Figure \ref{fig: intrograph}).  For this introduction, we assume that each boundary vertex is adjacent to one white vertex and no other vertices.  Let $n$ denote the number of boundary vertices, and index the boundary vertices by $1,2,...,n$ in a clockwise order.  

\begin{figure}[h!t]
\centering
\begin{subfigure}[b]{.45\textwidth}
\centering
\begin{tikzpicture}[baseline=(current bounding box.center)]
\begin{scope}[scale=.55]
		\path[use as  bounding box] (-4.5,-4.5) rectangle (4.5,4.5);
		\draw[fill=black!10] (0,0) circle (4);
		\node[invisible] (1) at (180:4) {};
		\node[invisible] (2) at (120:4) {};
		\node[invisible] (3) at (60:4) {};
		\node[invisible] (4) at (0:4) {};
		\node[invisible] (5) at (-60:4) {};
		\node[invisible] (6) at (-120:4) {};
		
		\node[dot, fill=white] (a) at (-2.75,.25) {};
		\node[dot, fill=black!50] (b) at (-1.5,2) {};
		\node[dot, fill=white] (c) at (1.5,2) {};
		\node[dot, fill=black!50] (d) at (2.75,.25) {};
		\node[dot, fill=black!50] (e) at (-1.5,-.8) {};
		\node[dot, fill=white] (f) at (-.45,.75) {};
		\node[dot, fill=black!50] (g) at (.45,.75) {};
		\node[dot, fill=white] (h) at (1.5,-.8) {};
		\node[dot, fill=black!50] (i) at (.5,-2) {};
		\node[dot, fill=white] (j) at (-.5,-2) {};
		
		\node[dot,fill=white] (b') at (-1.75,2.75) {};
		\node[dot,fill=white] (d') at (3.4,.125) {};
		\node[dot,fill=white] (i') at (1.25,-2.75) {};
		
		\draw (1) to (a) to (b) to (b') to (2);
		\draw (b) to (f) to (g) to (c) to (3);
		\draw (c) to (d) to (d') to (4);
		\draw (d) to (h) to (i) to (i') to (5);
		\draw (i) to (j) to (6);
		\draw (j) to (e) to (a);
		\draw (e) to (f);
		\draw (g) to (h);	
\end{scope}
\end{tikzpicture}
\caption{A 2-colored graph embedded in the disc.}
\label{fig: intrograph}
\end{subfigure}
\begin{subfigure}[b]{.45\textwidth}
\centering
\begin{tikzpicture}[baseline=(current bounding box.center)]
\begin{scope}[xshift=3in,scale=.55]
		\path[use as  bounding box] (-4.5,-4.5) rectangle (4.5,4.5);
		\draw[fill=black!10] (0,0) circle (4);
		\node[invisible] (1) at (180:4) {};
		\node[invisible] (2) at (120:4) {};
		\node[invisible] (3) at (60:4) {};
		\node[invisible] (4) at (0:4) {};
		\node[invisible] (5) at (-60:4) {};
		\node[invisible] (6) at (-120:4) {};
		
		\node[left] at (1) {$2$};
		\node[above left] at (2) {$3$};
		\node[above right] at (3) {$4$};
		\node[right] at (4) {$5$};
		\node[below right] at (5) {$6$};
		\node[below left] at (6) {$1$};
		
		\node[dot, fill=white] (a) at (-2.75,.25) {};
		\node[dot, fill=black!50] (b) at (-1.5,2) {};
		\node[dot, fill=white] (c) at (1.5,2) {};
		\node[dot, fill=black!50] (d) at (2.75,.25) {};
		\node[dot, fill=black!50] (e) at (-1.5,-.8) {};
		\node[dot, fill=white] (f) at (-.45,.75) {};
		\node[dot, fill=black!50] (g) at (.45,.75) {};
		\node[dot, fill=white] (h) at (1.5,-.8) {};
		\node[dot, fill=black!50] (i) at (.5,-2) {};
		\node[dot, fill=white] (j) at (-.5,-2) {};
		
		\node[dot,fill=white] (b') at (-1.75,2.75) {};
		\node[dot,fill=white] (d') at (3.4,.125) {};
		\node[dot,fill=white] (i') at (1.25,-2.75) {};
		
		\draw (1) to (a);
		\draw[matching] (a) to (b);
		\draw (b) to (b');
		\draw[matching] (b') to (2);
		\draw (b) to (f);
		\draw[matching] (f) to (g);
		\draw (g) to (c) to (3);
		\draw[matching] (c) to (d);
		\draw (d) to (d');
		\draw[matching] (d') to (4);
		\draw (d) to (h);
		\draw[matching] (h) to (i);
		\draw (i) to (i');
		\draw[matching] (i') to (5);
		\draw (i) to (j);
		\draw (j) to (6);
		\draw[matching] (j) to (e);
		\draw (e) to (a);
		\draw (e) to (f);
		\draw (g) to (h);	
\end{scope}
\end{tikzpicture}
\caption{A matching with boundary $356$.}
\label{fig: intromatching}
\end{subfigure}
\caption{A graph and a matching.}
\end{figure}
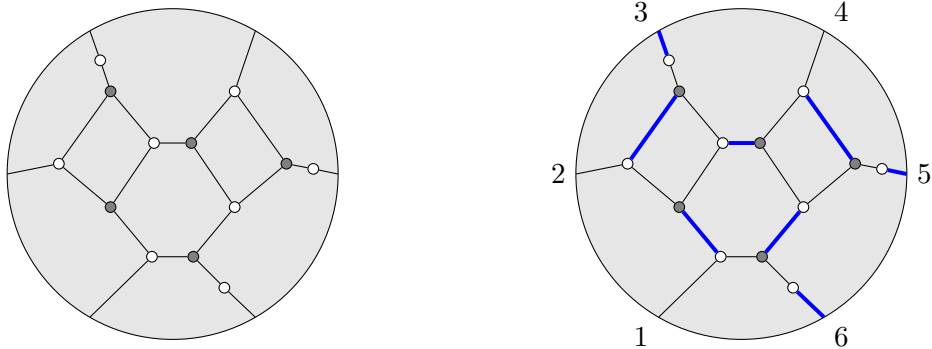

A \newword{matching} of $G$ is a collection of edges in $G$ which cover each internal vertex exactly once.  For a matching $M$, we let $\partial M$ denote the subset of the boundary vertices covered by $M$, which we identify with a subset of $[n]:=\{1,2,...,n\}$ (e.g. Figure \ref{fig: intromatching}).  That is,
\[ \partial M := \{ i : \text{vertex $i$ is covered by }M \} \subset [n] \]
The cardinality $k$ of $\partial M$ is constant for any matching of $G$, and given by
\[ k := (\#\text{ of white vertices}) - (\# \text{ of black vertices})\]
As long as $G$ admits a matching, the graph $G$ determines a positroid (Theorem~\ref{theorem is a positroid}) defined as
\[ \cM := \{ I\subset [n] \mid \text{there exists a matching $M$ with $\partial M=I$}\} \]
A \newword{reduced graph} is a graph $G$ as defined above, such that the number of \emph{faces} of $G$ (that is, components of the complement) is minimal among all graphs with the same positroid as $G$.



The matchings of $G$ with a fixed boundary may be collected into a \newword{partition function} as follows.  Let $\{z_e\}$ be a set of formal variables indexed by edges $e$ of $G$.  For a matching $M$ of $G$, define $z^M:= \prod_{e\in M} z_e$, and for a $k$-element subset $I$ of $[n]$, define
\[ D_I := \sum_{\stackrel{\text{matchings }M }{\text{with } \partial M=I}} z^M \]
%
Plugging complex numbers into the formal variables realizes $D_I$ as a regular function $\CC^E\rightarrow \CC$, where $E$ denotes the set of edges of $G$.  Running over all $k$-elements subsets of $[n]$, the partition functions define a regular map
\[ \CC^E\longrightarrow \CC^{\binom{[n]}{k}} \]
The partition functions are not algebraically independent, so this map lands in a subvariety.
\begin{TheoremPluck}
For any graph $G$ as above, the partition functions satisfy the Pl\"ucker relations.  Therefore, the map $\CC^E\rightarrow \CC^{\binom{[n]}{k}}$ with coordinates $\{D_I\}$
has image contained $\widetilde{Gr(k,n)}\subset \CC^{\binom{[n]}{k}}$.\footnote{Throughout this introduction, results which are proved later in the paper are numbered according to where their proofs can be found.}
\end{TheoremPluck}
\noindent The correct attribution for this result is difficult, see the discussion near the proof.

This map is almost never injective because of the following \newword{gauge transformations}: if $v$ is an internal vertex of $G$, $(z_e)$ is a point of $\CC^E$, and $t$ is a nonzero complex number, then define a new point $(z_e')$ of $\CC^E$ by
\[ z'_e = \begin{cases} t z_e & v \in e \\ z_E & \mbox{otherwise} \end{cases}. \]
Since each matching of $G$ contains exactly one edge covering $v$, we know that $(z')^M=t(z^M)$ and that $D_I(z') = tD_I(z)$.

The gauge transformations can be encoded more elegantly as follows.
The group $\mathbb{G}_m^E$ acts on $\CC^E$ by scaling the individual coordinates; in this way, $\mathbb{G}_m^E$ may be identified with ways to assign a nonzero `weight' to each edge.  Letting $V$ denote the set of internal vertices of $G$, the action of $\mathbb{G}_m^V$ by gauge transformations is equivalent to a map of algebraic groups
\[ \mathbb{G}_m^V\longrightarrow \mathbb{G}_m^E\]
where the coordinate at each edge is the product of the coordinates at its endpoints.


Before this paper, the following was known but not written explicitly; see Remark~\ref{Kuo credit}. Let $\mathbb{G}_m^{V-1}$ denote the subgroup of $\mathbb{G}_m^V$ such that the product of the coordinates is $1$; equivalently, this is the subgroup of the gauge group which leaves the partition functions invariant.
\begin{Proposition}
For a graph $G$ with positroid $\cM$, the map $\mathbb{G}_m^E\rightarrow \CC^{\binom{[n]}{k}}$ given in Pl\"ucker coordinates by the partition functions $D_I$ factors through $\mathbb{G}_m^E/\mathbb{G}_m^{V-1}$ and lands in $\tPi(\cM)$.
\end{Proposition}

When $G$ is reduced, we sharpen this to the following.
\begin{PropInCellOpen}
For a reduced graph $G$ with positroid $\cM$, the map $\mathbb{G}_m^E\rightarrow \CC^{\binom{[n]}{k}}$ given in Pl\"ucker coordinates by the partition functions $D_I$ factors through $\mathbb{G}_m^E/\mathbb{G}_m^{V-1}$ and lands in $\tPio(\cM)$, giving an inclusion
\[ \tDD:\mathbb{G}_m^E/\mathbb{G}_m^{V-1} \longrightarrow \tPio(\cM)\]
The map $\tDD$ descend to a well-defined quotient inclusion
\[ \DD: \mathbb{G}_m^E/\mathbb{G}_m^V \longrightarrow \Pio(\cM)\]
\end{PropInCellOpen}

\noindent We will refer to the maps $\DD$ and $\tDD$ as \newword{boundary measurement maps}. The map $\DD$ is equal to the boundary measurement map of Postnikov~\cite{Pos}; see the proof of Theorem~\ref{theorem obey Plucker relations} for a discussion of the equivalence between Postnikov's definition and our own.

\begin{myexample}
Consider the graph $G$ in Figure \ref{fig: intrograph}. Of all the $3$-element subsets of $[6]$, only $\{1,2,3\}$ is not the boundary of a matching. 
The open positroid variety $\Pi^\circ(\cM)$ is defined inside $Gr(3,6)$ by the vanishing of the Pl\"ucker coordinate $\Delta_{123}$ and the non-vanishing of $\Delta_{124}, \Delta_{234}, \Delta_{345}, \Delta_{456}, \Delta_{156}$, and $\Delta_{126}$.\footnote{The non-vanishing of these Pl\"ucker coordinates removes subspaces with a smaller positroid envelope than $\cM$.} As a consequence, the closure $\Pi(\cM)$ of $\Pio(\cM)$ is the Schubert divisor in $Gr(3,6)$.

\begin{figure}[h!t]
\centering
\begin{tikzpicture}
\begin{scope}[scale=.55]
		\draw[fill=black!10] (0,0) circle (4);
		\node[invisible] (1) at (180:4) {};
		\node[invisible] (2) at (120:4) {};
		\node[invisible] (3) at (60:4) {};
		\node[invisible] (4) at (0:4) {};
		\node[invisible] (5) at (-60:4) {};
		\node[invisible] (6) at (-120:4) {};
		
		\node[left] at (1) {$2$};
		\node[above left] at (2) {$3$};
		\node[above right] at (3) {$4$};
		\node[right] at (4) {$5$};
		\node[below right] at (5) {$6$};
		\node[below left] at (6) {$1$};
		
		\node[dot, fill=white] (a) at (-2.75,.25) {};
		\node[dot, fill=black!50] (b) at (-1.5,2) {};
		\node[dot, fill=white] (c) at (1.5,2) {};
		\node[dot, fill=black!50] (d) at (2.75,.25) {};
		\node[dot, fill=black!50] (e) at (-1.5,-.8) {};
		\node[dot, fill=white] (f) at (-.45,.75) {};
		\node[dot, fill=black!50] (g) at (.45,.75) {};
		\node[dot, fill=white] (h) at (1.5,-.8) {};
		\node[dot, fill=black!50] (i) at (.5,-2) {};
		\node[dot, fill=white] (j) at (-.5,-2) {};
		
		\node[dot,fill=white] (b') at (-1.75,2.75) {};
		\node[dot,fill=white] (d') at (3.4,.125) {};
		\node[dot,fill=white] (i') at (1.25,-2.75) {};
		
		\draw (1) to (a) to (b) to (b') to (2);
		\draw (b) to (f) to (g) to (c) to (3);
		\draw (c) to (d) to (d') to (4);
		\draw (d) to (h) to (i) to (i') to (5);
		\draw (i) to (j) to (6);
		\draw (j) to (e) to (a);
		\draw (e) to (f);
		\draw (g) to (h);	
		
		\draw (1) to node[above,blue] {$i$} (a) to node[above,blue] {$d$} (b) to node[right,blue] {$b$} (b') to node[right,blue] {$a$} (2);
		\draw (b) to node[above,blue] {$e$} (f) to node[above,blue] {$f$} (g) to node[above,blue] {$g$} (c) to node[left,blue] {$c$} (3);
		\draw (c) to node[above,blue] {$h$} (d) to node[above,blue] {$n$} (d') to node[above,blue] {$o$} (4);
		\draw (d) to node[above,blue] {$m$} (h) to node[above,blue] {$r$} (i) to node[above,blue] {$t$} (i') to node[above,blue] {$u$} (5);
		\draw (i) to node[above,blue] {$q$} (j) to node[above,blue] {$s$} (6);
		\draw (j) to node[above,blue] {$p$} (e) to node[above,blue] {$j$} (a);
		\draw (e) to node[above,blue] {$k$} (f);
		\draw (g) to node[above,blue] {$l$} (h);
	
\end{scope}
\end{tikzpicture}
\caption{A general set of edge weights.}
\label{fig: edgeweights}
\end{figure}
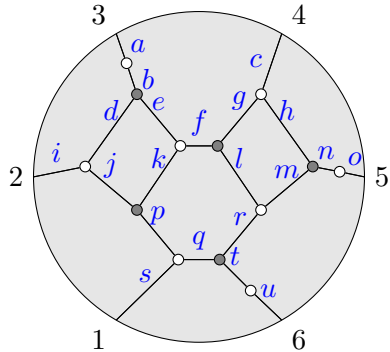

Let us describe a general point in $\mathbb{G}_m^E$ by assigning an indeterminant weight in $\mathbb{G}_m$ to each edge in $G$, as in Figure \ref{fig: edgeweights}. By Theorem \ref{theorem obey Plucker relations}, there exists a $3\times 6$ matrix such that, for any $I\in \binom{[6]}{3}$, the minor with columns in $I$ is equal to $D_I$. One such matrix is given below in \eqref{eq: matrix}.\footnote{Note that such a matrix is not uniquely determined; however, its row-span is.}
\begin{equation}\label{eq: matrix}
\begin{bmatrix}
1 & 0 & -\frac{aep}{bks} & 0 & \frac{fmop}{klns} & \frac{klqu+fpru}{klst}  \\
0 & 1 & \frac{adk+aej}{bik} & 0 & -\frac{fjmo}{ikln} & -\frac{fjru}{iklt}  \\
0 & 0 & 0 & bciklnst & bikost(hl+gm) & bgiknrsu 
\end{bmatrix}
\end{equation}
The boundary measurement map $\mathbb{D}$ for $G$ is the map which sends the edge weights given in Figure \ref{fig: edgeweights} to the row-span of the matrix in \eqref{eq: matrix}. 
\end{myexample}

\subsection{Pl\"ucker coordinates associated to faces}


In \cite{Pos}, Postnikov showed how a reduced graph determines a collection of \newword{strands}: oriented curves in the disc beginning and ending at boundary vertices of $G$ (e.g. Figure \ref{fig: introstrands}).  The details of this construction may be found in Section~\ref{sec strand}.

\begin{figure}[h!t]

\begin{subfigure}[b]{.3\textwidth}
\centering
\begin{tikzpicture}[baseline=(current bounding box.center)]
\begin{scope}[scale=.5]
		\path[use as  bounding box] (-4.5,-4.5) rectangle (4.5,4.5);
		\draw[fill=black!10] (0,0) circle (4);
		\node[invisible] (1) at (180:4) {};
		\node[invisible] (2) at (120:4) {};
		\node[invisible] (3) at (60:4) {};
		\node[invisible] (4) at (0:4) {};
		\node[invisible] (5) at (-60:4) {};
		\node[invisible] (6) at (-120:4) {};
		
		\node[left] at (1) {$2$};
		\node[above left] at (2) {$3$};
		\node[above right] at (3) {$4$};
		\node[right] at (4) {$5$};
		\node[below right] at (5) {$6$};
		\node[below left] at (6) {$1$};
				
		\node[dot, fill=white] (a) at (-2.75,.25) {};
		\node[dot, fill=black!50] (b) at (-1.5,2) {};
		\node[dot, fill=white] (c) at (1.5,2) {};
		\node[dot, fill=black!50] (d) at (2.75,.25) {};
		\node[dot, fill=black!50] (e) at (-1.5,-.8) {};
		\node[dot, fill=white] (f) at (-.45,.75) {};
		\node[dot, fill=black!50] (g) at (.45,.75) {};
		\node[dot, fill=white] (h) at (1.5,-.8) {};
		\node[dot, fill=black!50] (i) at (.5,-2) {};
		\node[dot, fill=white] (j) at (-.5,-2) {};
		
		\node[dot,fill=white] (b') at (-1.75,2.75) {};
		\node[dot,fill=white] (d') at (3.4,.125) {};
		\node[dot,fill=white] (i') at (1.25,-2.75) {};
		
		\draw (1) to (a) to node[invisible] (E) {} (b) to node[invisible] (F) {}  (b') to (2);
		\draw (b) to node[invisible] (G) {} (f) to node[invisible] (H) {} (g) to node[invisible] (I) {} (c) to (3);
		\draw (c) to node[invisible] (A) {} (d) to node[invisible] (J) {} (d') to (4);
		\draw (d) to node[invisible] (B) {} (h) to node[invisible] (C) {} (i) to node[invisible] (K) {} (i') to (5);
		\draw (i) to node[invisible] (D) {} (j) to (6);
		\draw (j) to node[invisible] (L) {} (e) to node[invisible] (M) {} (a);
		\draw (e) to node[invisible] (N) {} (f);
		\draw (g) to node[invisible] (O) {} (h);	
		
		\draw[oriented,red,out=285,in=80] (3) to (A);
		\draw[oriented,red,out=260,in=100] (A) to (B);
		\draw[oriented,red,out=280,in=0] (B) to (C);
		\draw[oriented,red,out=180,in=45] (C) to (D);
		\draw[oriented,red,out=225,in=0] (D) to (6);
		
		\draw[oriented,red,out=225,in=-55] (4) to (J);
		\draw[oriented,red,out=125,in=0] (J) to (A);
		\draw[oriented,red,out=180,in=0] (A) to (I);
		\draw[oriented,red,out=180,in=45] (I) to (H);
		\draw[oriented,red,out=225,in=0] (H) to (N);
		\draw[oriented,red,out=180,in=0] (N) to (M);
		\draw[oriented,red,out=180,in=-35] (M) to (1);
		
		\draw[oriented,red,out=180,in=-90] (5) to (K);
		\draw[oriented,red,out=90,in=-90] (K) to (C);
		\draw[thick,red,out=90,in=-90] (C) to (O);
		\draw[oriented,red,out=90,in=-90] (O) to (I);
		\draw[oriented,red,out=90,in=215] (I) to (3);
		
		\draw[oriented,red,out=80,in=-90] (6) to (L);
		\draw[thick,red,out=90,in=-90] (L) to (N);
		\draw[oriented,red,out=90,in=-90] (N) to (G);
		\draw[oriented,red,out=90,in=-30] (G) to (F);
		\draw[oriented,red,out=200,in=240] (F) to (2);
		
		\draw[oriented,red,out=55,in=180] (1) to (E);
		\draw[oriented,red,out=0,in=180] (E) to (G);
		\draw[oriented,red,out=0,in=135] (G) to (H);
		\draw[oriented,red,out=-45,in=180] (H) to (O);
		\draw[oriented,red,out=0,in=180] (O) to (B);
		\draw[oriented,red,out=0,in=210] (B) to (J);
		\draw[oriented,red,out=30,in=120] (J) to (4);
		
		\draw[oriented,red,out=-30,in=60] (2) to (F);
		\draw[oriented,red,out=240,in=90] (F) to (E);
		\draw[oriented,red,out=-90,in=90] (E) to (M);
		\draw[oriented,red,out=-90,in=180] (M) to (L);
		\draw[oriented,red,out=0,in=135] (L) to (D);
		\draw[oriented,red,out=-45,in=180] (D) to (K);
		\draw[oriented,red,out=0,in=90] (K) to (5);
		
\end{scope}
\end{tikzpicture}
\caption{The strands of the graph.}
\label{fig: introstrands}
\end{subfigure}
\begin{subfigure}[b]{.3\textwidth}
\centering
\begin{tikzpicture}[baseline=(current bounding box.center)]
\begin{scope}[scale=.5]
		\path[use as  bounding box] (-4.5,-4.5) rectangle (4.5,4.5);
		\draw[fill=black!10] (0,0) circle (4);
		\node[invisible] (1) at (180:4) {};
		\node[invisible] (2) at (120:4) {};
		\node[invisible] (3) at (60:4) {};
		\node[invisible] (4) at (0:4) {};
		\node[invisible] (5) at (-60:4) {};
		\node[invisible] (6) at (-120:4) {};
		
		
		\node[dot, fill=white] (a) at (-2.75,.25) {};
		\node[dot, fill=black!50] (b) at (-1.5,2) {};
		\node[dot, fill=white] (c) at (1.5,2) {};
		\node[dot, fill=black!50] (d) at (2.75,.25) {};
		\node[dot, fill=black!50] (e) at (-1.5,-.8) {};
		\node[dot, fill=white] (f) at (-.45,.75) {};
		\node[dot, fill=black!50] (g) at (.45,.75) {};
		\node[dot, fill=white] (h) at (1.5,-.8) {};
		\node[dot, fill=black!50] (i) at (.5,-2) {};
		\node[dot, fill=white] (j) at (-.5,-2) {};
		
		\node[dot,fill=white] (b') at (-1.75,2.75) {};
		\node[dot,fill=white] (d') at (3.4,.125) {};
		\node[dot,fill=white] (i') at (1.25,-2.75) {};
		
		\draw (1) to (a) to (b) to (b') to (2);
		\draw (b) to (f) to (g) to (c) to (3);
		\draw (c) to node[invisible] (A) {} (d) to (d') to (4);
		\draw (d) to node[invisible] (B) {} (h) to node[invisible] (C) {} (i) to (i') to (5);
		\draw (i) to node[invisible] (D) {} (j) to (6);
		\draw (j) to (e) to (a);
		\draw (e) to (f);
		\draw (g) to (h);	
		

	\node[red] at (150:3.25) {\footnotesize $345$};
	\node[red] at (90:2.5) {\footnotesize $456$};
	\node[red] at (30:3.25) {\footnotesize $156$};
	\node[red] at (-30:3) {\footnotesize $126$};
	\node[red] at (-90:3) {\footnotesize $124$};
	\node[red] at (-150:3) {\footnotesize $234$};
	\node[red] at (-1.6,.5) {\footnotesize $346$};
	\node[red] at (0,-.65) {\footnotesize $246$};
	\node[red] at (1.6,.5) {\footnotesize $256$};
\end{scope}
\end{tikzpicture}
\caption{Target-labeling of the faces.}
\label{fig: introtarget}
\end{subfigure}
\begin{subfigure}[b]{.3\textwidth}
\centering
\begin{tikzpicture}[baseline=(current bounding box.center)]
\begin{scope}[scale=.5]
		\path[use as  bounding box] (-4.5,-4.5) rectangle (4.5,4.5);
		\draw[fill=black!10] (0,0) circle (4);
		\node[invisible] (1) at (180:4) {};
		\node[invisible] (2) at (120:4) {};
		\node[invisible] (3) at (60:4) {};
		\node[invisible] (4) at (0:4) {};
		\node[invisible] (5) at (-60:4) {};
		\node[invisible] (6) at (-120:4) {};
		
		
		\node[dot, fill=white] (a) at (-2.75,.25) {};
		\node[dot, fill=black!50] (b) at (-1.5,2) {};
		\node[dot, fill=white] (c) at (1.5,2) {};
		\node[dot, fill=black!50] (d) at (2.75,.25) {};
		\node[dot, fill=black!50] (e) at (-1.5,-.8) {};
		\node[dot, fill=white] (f) at (-.45,.75) {};
		\node[dot, fill=black!50] (g) at (.45,.75) {};
		\node[dot, fill=white] (h) at (1.5,-.8) {};
		\node[dot, fill=black!50] (i) at (.5,-2) {};
		\node[dot, fill=white] (j) at (-.5,-2) {};
		
		\node[dot,fill=white] (b') at (-1.75,2.75) {};
		\node[dot,fill=white] (d') at (3.4,.125) {};
		\node[dot,fill=white] (i') at (1.25,-2.75) {};
		
		\draw (1) to (a) to (b) to (b') to (2);
		\draw (b) to (f) to (g) to (c) to (3);
		\draw (c) to node[invisible] (A) {} (d) to (d') to (4);
		\draw (d) to node[invisible] (B) {} (h) to node[invisible] (C) {} (i) to (i') to (5);
		\draw (i) to node[invisible] (D) {} (j) to (6);
		\draw (j) to (e) to (a);
		\draw (e) to (f);
		\draw (g) to (h);	
		
	\node[red] at (150:3.25) {\footnotesize $126$};
	\node[red] at (90:2.5) {\footnotesize $236$};
	\node[red] at (30:3.25) {\footnotesize $234	$};
	\node[red] at (-30:3) {\footnotesize $345$};
	\node[red] at (-90:3) {\footnotesize $456$};
	\node[red] at (-150:3) {\footnotesize $156$};
	\node[red] at (-1.6,.5) {\footnotesize $136$};
	\node[red] at (0,-.65) {\footnotesize $356$};
	\node[red] at (1.6,.5) {\footnotesize $235$};
\end{scope}
\end{tikzpicture}
\caption{Source-labeling of the faces.}
\label{fig: introsource}
\end{subfigure}
\caption{Two ways to associate a $k$-element subset of $[n]$ to a face.}
\end{figure}
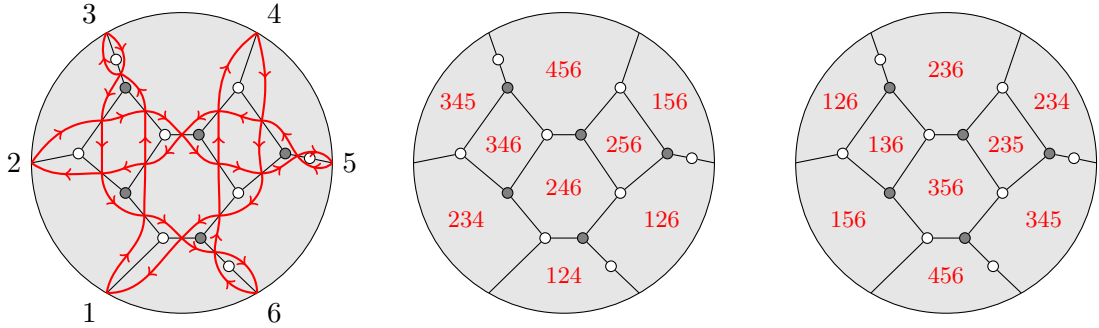

The strands do not self-intersect (except possibly at the boundary), so each one subdivides the disc into two components.  The orientation of a strand distinguishes these components as the `left side' and the `right side'.  One may check that each face of $G$ is on the left side of exactly $k$-many strands, where $k$ again denotes the number of white vertices minus the number of black vertices.

There are two natural ways to use a collection of $k$-many strands to determine a $k$-element subset of $[n]$: identify each strand either with the index of its \emph{source} vertex, or with the index of its \emph{target} vertex. 
In this paper, we will be forced to work with both conventions.
Given a face $f$ of $G$, define the following two $k$-element subsets of $[n]$ (e.g. Figures \ref{fig: introtarget} and \ref{fig: introsource}).
\[ \tI(f) := \{ i\in [n] \mid \text{$f$ is to the left of the strand ending at vertex $i$} \} \]
\[ \sI(f) := \{ i\in [n] \mid \text{$f$ is to the left of the strand starting at vertex $i$} \} \]
For any $k$-element subset $I$ of $[n]$, let $\Delta_I$ denote the Pl\"ucker coordinate on $\widetilde{Gr(k,n)}$ indexed by $I$.  Hence, each face $f$ in $G$ determines two Pl\"ucker coordinates, given by $\Delta_{\sI(f)}$ and $\Delta_{\tI(f)}$.

Letting $F$ denote the set of faces of $G$, this determines a pair of regular maps
\[ \tF:\tPio(\cM) \longrightarrow \CC^F\]
\[ \sF:\tPio(\cM) \longrightarrow \CC^F\]
where the coordinate corresponding to a face is the appropriate Pl\"ucker coordinate.

\subsection{Extremal matchings} In the next two sections, we introduce maps which will relate the domains and images of $\tDD$, $\sF$, and $\tF$.

For any face $f$ of a reduced graph, define two matchings $\vecM(f)$ and $\cevM(f)$ of $G$ such that 
\[ \partial \vecM(f) = \tI(f) \text{ and }  \partial \cevM(f) = \sI(f). \]
An edge $e$ of $G$ appears in $\vecM(f)$ if and only if the face $f$ is contained in the ``downstream wedge" bounded by the two half strands flowing out of $e$ and the edge of the disc (see Figure~\ref{fig: downstreamwedge}).
The edge $e$ appears in $\cevM(f)$ if $f$ is in the analogous upstream edge.
For proofs that $\vecM(f)$ and $\cevM(f)$ are matchings and have the stated boundaries, see Theorem~\ref{thm: minmatchprop}.

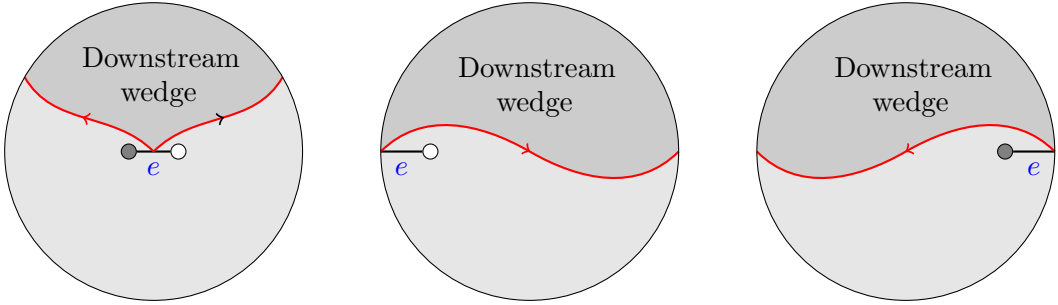
\begin{figure}[h!t]
\begin{tikzpicture}
\begin{scope}[scale=.66]
	\path[fill=black!10] (0,0) circle (3);
	\node[mutable,fill=black!50] (a) at (-.5,0) {};
	\node[mutable,fill=white] (b) at (.5,0) {};
	\draw[thick] (a) to node[below,blue] {$e$} (b);
	\path[fill=black!20] (150:3) to [out=-60,in=135] (0,0) to [out=45,in=240] (30:3) arc (30:150:3);
	\node[rounded corners] at (.1,1.5) {
		\begin{minipage}{.8in}\centering Downstream wedge\end{minipage}
	};
	\draw[antioriented] (150:3) to [out=-60,in=135] (0,0);
	\draw[oriented] (0,0) to [out=45,in=240] (30:3);
	\draw (0,0) circle (3);
\end{scope}
\begin{scope}[xshift=5cm,scale=.66]
	\path[fill=black!10] (0,0) circle (3);
	\node[mutable,fill=white] (b) at (-2,0) {};
	\draw[thick] (-3,0) to node[below,blue] {$e$} (b);
	\path[fill=black!20] (3,0) to[out=225,in=-30] (0,0) to [out=150,in=45] (-3,0) arc (180:0:3);
	\node[rounded corners] at (.1,1.3) {
		\begin{minipage}{.8in}\centering Downstream wedge\end{minipage}
	};
	\draw[antioriented,out=225,in=-30] (3,0) to (0,0) to [out=150,in=45] (-3,0);
	\draw (0,0) circle (3);
\end{scope}
\begin{scope}[xshift=10cm,scale=.66]
	\path[fill=black!10] (0,0) circle (3);
	\node[mutable,fill=black!50] (a) at (2,0) {};
	\draw[thick] (a) to node[below,blue] {$e$} (3,0);
	\path[fill=black!20] (-3,0) to[out=-45,in=-150] (0,0) to [out=30,in=135] (3,0) arc (0:180:3);
	\node[rounded corners] at (.1,1.3) {
		\begin{minipage}{.8in}\centering Downstream wedge\end{minipage}
	};
	\draw[antioriented] (-3,0) to[out=-45,in=-150] (0,0) to [out=30,in=135] (3,0);
	\draw (0,0) circle (3);
\end{scope}
\end{tikzpicture}
\caption{The downstream wedge of an edge $e$}
\label{fig: downstreamwedge}
\end{figure}

\begin{myexample}
The matching given in Figure \ref{fig: intromatching} is the matching $\vecM(f)$, where $f$ is the interior hexagonal face. The boundary $356$ of $\vecM(f)$ coincides with the source-labeling of $f$, as shown in Figure \ref{fig: introsource}.
\end{myexample}

Let $\rM$ and $\lM$ be the monomial maps $\mathbb{G}_m^E\longrightarrow \mathbb{G}_m^F$ where, for each face $f$, the $f$-coordinate of $\rM(z)$ is $z^{-\vecM(f)}$ and the $f$-coordinate of $\lM(z)$ is $z^{-\cevM(f)}$.

\begin{PropInverseTorii}
For a reduced graph $G$, the maps $\rM$ and $\lM$ descend to well-defined isomorphisms
\[ \mathbb{G}_m^E/\mathbb{G}_m^{V-1} \stackrel{\sim}{\longrightarrow} \mathbb{G}_m^F\]
\end{PropInverseTorii}

\noindent We denote the inverses of $\rM$ and $\lM$ by $\lpartial$ and $\rpartial$, respectively.  Justification for this notation and an explicit formula for $\lpartial$ and $\rpartial$ are given in Section~\ref{sec extremal}.

\subsection{The twists of a positroid variety} \label{sec twist intro}

We now define a pair of mutually inverse automorphisms $\rt$ and $\lt$ of $\tPio(\cM)$, called the \emph{right twist} and \emph{left twist}, respectively.  
The definitions of the twists are elementary, and use none of the combinatorics or geometry we have built up so far. 

Let $A$ denote a $k\times n$ matrix of rank $k$. In this introduction, we will assume for simplicity that $A$ has no zero columns.   Let $A_i$ denote the $i$th column of $A$, with indices taken cyclically; that is, $A_{i+n}=A_i$.
The \newword{right twist} $\rt(A)$ of $A$ is the $k\times n$ matrix such that, for all $i$, the $i$th column $\rt(A)_i$ satisfies the relations
\[ \langle \rt(A)_i \mid A_i \rangle = 1, \text{ and}\]
\[ \langle \rt(A)_i \mid A_j \rangle =0 \text{ if $A_j$ is not in the span of $\{A_i,A_{i+1},\ldots, A_{j-2}, A_{j-1}\}$ } \]
Similarly, the \newword{left twist} of $A$ is the $k\times n$ matrix $\lt(A)$ defined on columns by the relations
\[ \langle \lt(A)_i \mid A_i \rangle = 1, \text{ and}\]
\[ \langle \lt(A)_i \mid A_j \rangle =0 \text{ if $A_j$ is not in the span of $\{A_{j+1}, A_{j+2},\ldots,  A_{i-1},A_{i} \}$ } \]
The reader is cautioned that these operations are only piecewise continuous on the space of matrices.

\begin{myexample} Each of the following matrices is the right twist of the matrix to its left, and the left twist of the matrix to its right.
\[\begin{tikzpicture}[baseline={([yshift=-3pt]current bounding box.center)}]
	\node (a) at (-5,0) {$\gmat{
1 & 0 & 1 & 0 & 1 \\
-1 & 1 & 0 & 0 & 0 \\
1 & -1 & 0 & 1 & 1
}$};
	\node (b) at (0,0) {$
\gmat{ 
1 & 0 & 1 & -1 & 0 \\
0 & 1 & 1 & 0 & 1 \\ 
0 & 0 & 0 & 1 & 1
}$};
	\node (c) at (5,0) {$\gmat{
1 & -1 & 0 & 0 & 0 \\
0 & 1 & 1 & -1 & 0 \\
1 & -1 & 0 & 1 & 1 
}$};
	\draw[|-angle 90,out=5,in=175] (a) to node[above] {$\rt$} (b);
	\draw[|-angle 90,out=5,in=175] (b) to node[above] {$\rt$} (c);
	\draw[|-angle 90,out=185,in=-5] (c) to node[below] {$\lt$} (b);
	\draw[|-angle 90,out=185,in=-5] (b) to node[below] {$\lt$} (a);
\end{tikzpicture} \qedhere\]
\end{myexample}

As the example suggests, the two twists are inverse to each other.
\begin{TheoremInverseTwists}
If $A$ is a $k\times n$ matrix of rank $k$, then $\rt(\lt(A))=\lt(\rt(A))=A$.
\end{TheoremInverseTwists}

The set of $k\times n$ matrices of rank $k$ naturally projects onto $\widetilde{Gr(k,n)}$ and $Gr(k,n)$, in the latter case sending a matrix to the span of its rows.  The twists descend to well-defined maps on these spaces as well (see Proposition~\ref{prop: equivariance}).  The twists become continuous when restricted to an individual positroid variety. More specifically:

\begin{CorollaryTwistQuotient}
For each positroid $\cM$, the twists $\rt$ and $\lt$ restrict to mutually inverse, regular automorphisms of $\tPio(\cM)$ and $\Pio(\cM)$.
\end{CorollaryTwistQuotient}


\subsection{The main theorem} \label{sec main theorem intro}

%
%
%
%
We are now in a position to state the main theorem.

\begin{theoremMT}
Let $G$ be a reduced graph with positroid $\cM$. The following diagram commutes, where dashed arrows denote rational maps.
\[\begin{tikzpicture}
	\node (F1) at (-3,0) {$\mathbb{G}_m^F$};
	\node (E) at (0,0) {$\mathbb{G}_m^E/\mathbb{G}_m^{V-1}$};
	\node (F2) at (3,0) {$\mathbb{G}_m^F$};
	\node (P1) at (-3,-2) {$\tPio(\cM)$};
	\node (P2) at (0,-2) {$\tPio(\cM)$};
	\node (P3) at (3,-2) {$\tPio(\cM)$};
	\draw[-angle 90,out=15,in=170] (F1) to node[above] {$\rpartial$} (E);
	\draw[-angle 90,out=10,in=165] (E) to node[above] {$\rM$} (F2);
	\draw[-angle 90,out=195,in=-10] (F2) to node[below] {$\lpartial$} (E);
	\draw[-angle 90,out=190,in=-15] (E) to node[below] {$\lM$} (F1);
	\draw[dashed,-angle 90] (P1) to node[left] {$\tF$} (F1);
	\draw[-angle 90] (E) to node[left] {$\tDD$} (P2);
	\draw[dashed,-angle 90] (P3) to node[right] {$\sF$} (F2);
	\draw[-angle 90,out=15,in=165] (P1) to node[above] {$\rt$} (P2);
	\draw[-angle 90,out=15,in=165] (P2) to node[above] {$\rt$} (P3);
	\draw[-angle 90,out=195,in=-15] (P3) to node[below] {$\lt$} (P2);
	\draw[-angle 90,out=195,in=-15] (P2) to node[below] {$\lt$} (P1);

\end{tikzpicture}\]
More specifically, the diagram commutes as a diagram of rational maps, and any composition of maps beginning in the top row is regular.
\end{theoremMT}

The morphisms in this diagram either commute or anticommute with the $\mathbb{G}_m$ action on each variety, and so the diagram descends to a commutative diagram on the quotients.
\[\begin{tikzpicture}
	\node (F1) at (-3,0) {$\mathbb{G}_m^F/\mathbb{G}_m$};
	\node (E) at (0,0) {$\mathbb{G}_m^E/\mathbb{G}_m^{V}$};
	\node (F2) at (3,0) {$\mathbb{G}_m^F/\mathbb{G}_m$};
	\node (P1) at (-3,-2) {$\Pio(\cM)$};
	\node (P2) at (0,-2) {$\Pio(\cM)$};
	\node (P3) at (3,-2) {$\Pio(\cM)$};
	\draw[-angle 90,out=10,in=170] (F1) to node[above] {$\rpartial$} (E);
	\draw[-angle 90,out=10,in=170] (E) to node[above] {$\rM$} (F2);
	\draw[-angle 90,out=190,in=-10] (F2) to node[below] {$\lpartial$} (E);
	\draw[-angle 90,out=190,in=-10] (E) to node[below] {$\lM$} (F1);
	\draw[dashed,-angle 90] (P1) to node[left] {$\tF$} (F1);
	\draw[-angle 90] (E) to node[left] {$\DD$} (P2);
	\draw[dashed,-angle 90] (P3) to node[right] {$\sF$} (F2);
	\draw[-angle 90,out=15,in=165] (P1) to node[above] {$\rt$} (P2);
	\draw[-angle 90,out=15,in=165] (P2) to node[above] {$\rt$} (P3);
	\draw[-angle 90,out=195,in=-15] (P3) to node[below] {$\lt$} (P2);
	\draw[-angle 90,out=195,in=-15] (P2) to node[below] {$\lt$} (P1);

\end{tikzpicture}\]

As a corollary, we obtain a combinatorial formula for the Pl\"ucker coordinates of a twisted point as a Laurent polynomial in the Pl\"ucker coordinates of the original point (Proposition~\ref{dimer sum is twist}).

\begin{myexample}
Let us consider the theorem in the running example of Figure \ref{fig: intrograph}. The boundary measurement map $\mathbb{D}$ sends the edge weights in Figure \ref{fig: edgeweights} to the row-span of the matrix in \eqref{eq: matrix}. The right twist of this matrix is given below in \eqref{eq: twistedmatrix}.

\begin{equation}\label{eq: twistedmatrix}
\begin{bmatrix}
1 & \frac{dks+ejs}{eip} & \frac{bjs}{adp} & \frac{hrs}{cmq} & 0 &0  \\
0 & 1 & \frac{bi}{ad} & \frac{fhipr+ikq(hl+gm)}{cfjmq} & \frac{gikn}{fhjo} & 0  \\
0 & 0 & 0 & \frac{1}{bciklnst} & \frac{1}{bhiklots} & \frac{1}{bgiknrsu} 
\end{bmatrix}
\end{equation}

To determine the value of $\sF$ at the point in $\Pi^\circ(\cM)$ defined by this matrix, we compute the nine minors with columns given by  the source labels of faces in $G$ (cf. Figure \ref{fig: introsource}).

\[\begin{array}{ccc}
\Delta_{156}= \frac{1}{bfhjorsu}  & \Delta_{126} = \frac{1}{bgiknrsu} & \Delta_{236}= \frac{1}{aegipnru}  \\
\Delta_{234}= \frac{1}{aceilpnt} & \Delta_{345} = \frac{1}{acdfmopt} & \Delta_{456} = \frac{1}{bcfjmpqu} \\
\Delta_{136}= \frac{1}{adgknrsu} & \Delta_{356} = \frac{1}{adfhporu} & \Delta_{235} = \frac{1}{aehilpot}
\end{array}\]

We see that, for each face $f$ in $G$, the value of $\Delta_{\sI(f)}$ on the matrix in \eqref{eq: twistedmatrix} is the reciprocal of the product of the edge weights in the extremal matching $\vecM(f)$. This is equivalent to the equality $\rM=\sF\circ \rt \circ \DD$, and thus the commutativity of the right square in Theorem \ref{main theorem}.
\end{myexample}

\subsection{Outline of paper} \label{sec outline}

The previous introduction presented as much background material as we need to state our results; we now begin filling in the additional background we need to prove them.
In Section~\ref{sec positroid}, we present the variety of combinatorial and geometric tools we will need for working with positroids.
In Section~\ref{sec matchings}, we discuss combinatorics related to matchings of planar graphs.
In Section~\ref{sec strand}, we explain the results we will need from Postnikov's theory of alternating strand diagrams.

The next two sections discuss prerequisite results which are largely original to this paper.
Section~\ref{sec extremal} discusses the combinatorics of the extremal matchings $\vecM$ and $\cevM$. 
Section~\ref{sec twist} defines the twist maps and proves many lemmas about them.
With these sections, we conclude the presentation of background material and move to the proof of the main results.

In Section~\ref{sec main theorem}, we restate our main results and several corollary results. In Section~\ref{sec bridge}, we introduce \emph{bridge decompositions}, a technical tool for building reduced graphs out of smaller reduced graphs. Finally, in Section~\ref{sec main thm proof}, we complete the proof of Theorem \ref{main theorem}.

We conclude with two appendices. 
Appendix~\ref{sec nice twist} considers several cases and examples where the twist map takes a particularly elegant form, making connections with matrix factorizations and with various enumerative results in matching theory.
Appendix~\ref{sec Propp} discusses connections between our extremal matchings (Section~\ref{sec extremal}) and work of Propp and of Kenyon and Goncharov.

\subsection{Acknowledgments} The second author has been thinking about these issues for a long time and has discussed them with many people.  He particularly recalls helpful conversations with Sergey Fomin, Suho Oh, Alex Postnikov, Jim Propp, Jeanne Scott, Kelli Talaska, Dylan Thurston and Lauren Williams. 
Both authors are thankful to Rachel Karpman for comments on an earlier draft of this paper.

These results were accepted for presentation at FPSAC 2016, and Section~\ref{Introduction} is similar in both substance and language to the results in our FPSAC extended abstract. The new material in this paper is the proofs in the remaining sections.

\section{The many definitions of positroid and positroid variety} \label{sec positroid}

Given a $k\times n$ matrix of rank $k$, its \newword{column matroid} $\cM\subset \binom{[n]}{k}$ is the set of subsets $J\subset [n]$ indexing collections of columns which form a basis.
A \newword{positroid} is a matroid $\cM$ with a `totally non-negative' representation; that is, $\cM$ is the column matroid of a matrix whose maximal minors are non-negative real.
(See \cite{Pos}, \cite{Oh11}, \cite{KLS13}, \cite{ARW13} for other, equivalent, definitions.)
In contrast with the difficult general problem of characterizing representable matroids, positroids can be explicitly classified by several equivalent combinatorial objects, which we now recall. See~\cite[Section 3]{KLS13} for further discussion.

\subsection{Classification of positroids} \label{sec positroid combinatorics}

For each $a\in [n]$, let $\prec_a$ denote the linear ordering on $[n]$ given below:
\[ a\prec_a a+1 \prec_a ...\prec_a n \prec_a 1 \prec_a 2\prec_a ...\prec_a a-1 \]
Extend this to a partial ordering on $\binom{[n]}{k}$, where $B\preceq_a C$ means that
\[\forall i,\;\;\;  b_i\preceq_a c_i,  \text{ where } B=\{b_1\prec_a b_2\prec_a....\prec_a b_k\}\text{ and } C=\{c_1\prec_ac_2 \prec_a...\prec_a c_k\}\]

For each $a\in [n]$, a matroid $\cM$ has a unique $\prec_a$-minimal element we denote by $\vecI_a$, the $a$-minimal basis.
The sequence $\veccI=(\vecI_1,\vecI_2,...,\vecI_n )$ of minimal bases of $\cM$ has the property that, for all $a\in [n]$, 
\begin{itemize}
	\item if $a\in \vecI_a$, then $(\vecI_a\setminus \{a\}) \subset \vecI_{a+1}$, and
	\item if $a\not\in \vecI_a$, then $\vecI_a=\vecI_{a+1}$.
\end{itemize}
See~\cite[Lemma 16.3]{Pos}. The index $a+1$ is taken modulo $n$.

An $[n]$-indexed collection $\veccI=\{\vecI_1,\vecI_2,...,\vecI_n\}\subset \binom{[n]}{k}$ satisfying this property is called a \newword{Grassman necklace}. 
Given a Grassman necklace $\veccI$, there is a unique largest matroid $\cM$ whose set of $a$-minimal bases is $\veccI$.  The construction is direct: define $\cM$ to be those $k$-element sets $J\subset [n]$ for which $\vecI_a\preceq_a J$ for all $a$.
The resulting collection of sets is not just a matroid; it is a positroid by the following theorem.
\begin{theorem}\label{thm: PostGrass}\cite[Theorem 6]{Oh11}
For a Grassman necklace $\veccI$, let $\cM$ be the set of $k$-element subsets $J$ of $[n]$ for which $\vecI_a\preceq_a J$ for all $a$.  Then $\cM$ is a positroid.  Every positroid can be realized by some Grassman necklace in this way.
\end{theorem}
\noindent As a corollary, the map sending a positroid to its Grassman necklace of $a$-minimal bases is a bijection between the set of positroids and the set of Grassman necklaces.

Another consequence of the theorem is that  every matroid is contained in a unique minimal positroid. This positroid can be constructed by first finding the Grassman necklace $\cI$ of $a$-minimal bases, and then applying the construction in the theorem.

It will be convenient to also consider the dual notion corresponding to maximal bases.  For a matroid $\cM$, let $\cevI_a$ denote the unique $\prec_{a+1}$-maximal basis in $\cM$.  The collection of all maximal bases $\cevcI:= \{\cevI_1,\cevI_2,...,\cevI_n\}$ has the property that, for all $a\in [n]$, 
\begin{itemize}
	\item if $a\in \cevI_a$, then $(\cevI_a\setminus \{ a\}) \subset \vecI_{a-1}$, and
	\item if $a\not\in \cevI_a$, then $\cevI_a=\cevI_{a-1}$.
\end{itemize}
An $[n]$-indexed collection $\cevcI=\{\cevI_1,\cevI_2,...,\cevI_n\}\subset \binom{[n]}{k}$ satisfying this property is called a \newword{reverse Grassman necklace}.
By a symmetric analog of Theorem \ref{thm: PostGrass}, reverse Grassman necklaces are in bijection with positroids, and so they are also in bijection with Grassman necklaces.  


Grassman necklaces are equivalent to certain permutations of $\ZZ$, which we now define.  A \newword{bounded affine permutation}\footnote{Bounded affine permutations are more evocatively called \newword{juggling patterns} in \cite{KLS13}.} of type $(k,n)$ is a bijection $\pi:\ZZ\rightarrow \ZZ$ such that:
\begin{itemize}
	\item for all $a\in \ZZ$, $\pi(a+n) = \pi(a)+n$,
	\item for all $a\in \ZZ$, $a\leq \pi(a) \leq a+n$, and
	\item $\frac{1}{n} \sum_{a=1}^n (\pi(a)-a) =k$.
\end{itemize}
A Grassman necklace $\veccI$ in $\binom{[n]}{k}$ defines the following bounded affine permutation $\pi$ of type $(k,n)$.
\begin{itemize}
	\item If $a \in \vecI_a$, then $a < \pi(a) \leq a+n$ and $\pi(a)$ is determined by the relation:
	\[ \vecI_{a+1} \equiv  (\vecI_a \setminus \{ a \}) \cup \{ \pi(a) \} \text{ (mod $n$)}\]
	\item If $a\not\in \vecI_a$, then $\pi(a)=a$.
\end{itemize}

\begin{prop}[{\cite[Corollary 3.13]{KLS13}}]
This construction defines a bijection from Grassman necklaces in $\binom{[n]}{k}$ to bounded affine permutations of type $(k,n)$.
\end{prop}

If $\veccI$ is the Grassman necklace of the column matroid of a matrix $A$, then $\pi$ can be constructed directly from $A$ by the following recipe, which we learned from Allen Knutson.
We leave the proof to the reader.
\begin{lemma} \label{lem Knutson pi rule}
With the above notation, $\pi(a)$ is the minimal $r \geq a$ for which 
\[A_a \in \text{span}(A_{a+1},A_{a+2},...,A_{r}).\]
\end{lemma}

We note some degenerate cases: $\pi(a)=a$ if and only if $A_a=0$; $\pi(a)=a+1$ if and only if $A_a$ and $A_{a+1}$ are parallel and not zero, $\pi(a)=a+n$ if and only if $A_a$ is not in $\text{span}(A_{a+1}, A_{a+2}, \ldots, A_{a+n-1})$.

\begin{remark}
The analogous construction for the reverse Grassmann necklace $\cevcI$ yields the inverse permutation $\pi^{-1}:\ZZ \rightarrow \ZZ$.
\end{remark}

If $b < a \leq \pi(a) < \pi(b)$, we say that \newword{$a$ $\pi$-implies $b$}, and write $a \impto_{\pi} b$.
Note that $\impto_{\pi}$ defines a poset structure on $[n]$. We define the \newword{length} of $\pi$, written $\ell(\pi)$ to be $\# \{ (a,b) : 1 \leq a \leq n,\ a \leq b \leq b+n,\ a \impto_{\pi} b \}$. This is the length of $\pi$ as an element of the affine symmetric group, as discussed in~\cite[Section 3.2]{KLS13}.
In~\cite[Section 5]{Pos}, $a \to \pi(a)$ and $b \to \pi(b)$ are called ``aligned".
See Lemma~\ref{label implication} for a justification of the notation $\impto$.

The following description of $\vecI$ using both $\pi$ and simple geometric properties of $A$ is often more convenient than computing $\vecI$ in terms of solely $\pi$ or $A$.

\begin{lemma} \label{lem imp Grass necklace}
Fix a bounded affine permutation $\pi$, and an integer $a$.
\begin{itemize}
	\item The set $\vecI_a$ is the disjoint union of $\{b \mid a \impto_{\pi} \pi^{-1}(b)\}$ and the $a$-minimal subset of \\ $(A_{a}, A_{a+1}, \ldots, A_{\pi(a)-1})$ that is a basis for $\mathrm{span}(A_{a}, A_{a+1}, \ldots, A_{\pi(a)-1})$.
	\item The set $\cevI_a$ is the disjoint union of $\{b \mid a \impto_{\pi} b\}$ and the $(a+1)$-maximal subset of \\ $(A_{\pi^{-1}(a)+1}, \ldots, A_{a-1},A_{a})$ that is a basis for $\mathrm{span}(A_{\pi^{-1}(a)+1}, \ldots, A_{a-1},A_{a})$.
\end{itemize}
\end{lemma}

\begin{remark}
If $\pi(a)=a$, then the latter sets are empty.
\end{remark}

\begin{proof}
We prove the first statement, the second is similar. 

Let $J$ be the $a$-minimal basis among $(A_{a}, A_{a+1}, \ldots, A_{\pi(a)-1})$. 
Since $\vecI_a$ is the $a$-minimal basis for $\CC^k$ among $(A_a, A_{a+1}, \ldots, A_{a+n-1})$, we have $J = \vecI_a \cap \{ a, a+1, \ldots, \pi(a)-1 \}$.
So it remains to show that $\vecI_a \cap \{ \pi(a), \pi(a)+1, \ldots, a+n-1 \} = \{ b \in [n] \mid a \impto_{\pi} \pi^{-1}(b) \}$.

Suppose that $a \impto_{\pi} \pi^{-1}(b)$, so $\pi^{-1}(b) < a \leq \pi(a) < b$. By the definition of $\pi(\pi^{-1}(b))=b$, the vector $A_{b}$ is not in the span of $\{ A_{\pi^{-1}(b)+1}, A_{\pi^{-1}(b)+2}, \ldots, A_{b-1} \}$. Restricting to the subset starting at $A_a$, we see that $A_{b}$ is not in the span of $\{ A_{a}, A_{a+1}, \ldots, A_{\pi(b)-1} \}$, and so $\pi(b) \in \vecI_a$.

Conversely, suppose that $b \in \vecI_a \cap \{ \pi(a), \pi(a)+1, \ldots, a+n-1 \}$. Then $A_{b}$ is not the span of $\{ A_a, A_{a+1}, \ldots, A_{b-1} \}$. On the other hand, $A_b$ is in the span of $\{ A_{\pi^{-1}(b)}, A_{\pi^{-1}(b)+1}, \ldots, A_{b-1} \}$ by Lemma \ref{lem Knutson pi rule}. We deduce that $\pi^{-1}(b)<a$ and thus $a \impto_\pi \pi^{-1}(b)$.
\end{proof}

\begin{cor} \label{cor imp indep} 
The sets $\{ A_{\pi(b)} \mid a \impto_{\pi} b \}$ and $\{ A_{b} \mid a \impto_{\pi} b \}$ are each linearly independent.
\end{cor}

\begin{proof}
We have just shown that they are contained in the bases $\vecI_a$ and $\cevI_a$, respectively.
\end{proof}


In summary, a $k\times n$ matrix $A$ of rank $k$ determines the following equivalent combinatorial objects.

\begin{prop}\label{prop: equivposi}
Let $A$ be a $k\times n$ matrix of rank $k$.  Then each of the following objects associated to $A$ can be reconstructed from each other.
\begin{enumerate}
	\item The unique minimal positroid $\cM$ containing the column matroid of $A$.
	\item The Grassman necklace $\veccI=\{\vecI_1,\vecI_2,...,\vecI_n\}$, where $\vecI_a$ is the $a$-minimal basis of the columns of $A$.
	\item The reverse Grassman necklace $\cevcI=\{\cevI_1,\cevI_2,...,\cevI_n\}$, where $\cevI_a$ is the $a+1$-maximal basis of the columns of $A$.
	\item The bounded affine permutation $\pi:\ZZ\rightarrow \ZZ$ defined by
	\[ A_a \not\in \text{span} \{ A_{a+1},A_{a+2},...,A_{\pi(a)-1} \}\text{ and } A_a \in \text{span} \{ A_{a+1},A_{a+2},...,A_{\pi(a)} \}\ \]
\end{enumerate}
\end{prop}

\subsection{Several flavors of positroid variety} \label{pos vars}

For fixed $k\leq n$, 
\begin{itemize}
	\item let $\Mat(k,n)$ denote the variety of complex $k\times n$ matrices, 
	\item let $\Mato(k,n)$ denote the variety of complex $k\times n$ matrices of rank $k$,   
	\item let $Gr(k,n)$ denote the $(k,n)$-Grassmannian: the variety of $k$-planes in $\CC^n$, and
	\item let $\widetilde{Gr(k,n)}$ denote the affine cone over the Pl\"ucker embedding of $Gr(k,n)$. 
\end{itemize}
The general linear group $GL_k$ acts freely on $\Mato(k,n)$ by left multiplication, 
and there are standard isomorphisms
\[ 
\begin{array}{rcl}
SL_k\backslash \Mato(k,n) &\garrow{\sim}& \widetilde{Gr(k,n)} \setminus \{ 0 \} \\
 GL_k\backslash \Mato(k,n) &\garrow{\sim}& Gr(k,n) \\
 \end{array} \]
sending a matrix $A$ to the exterior product of the rows of $A$, and to the row-span of $A$, respectively.
 
For a positroid $\cM\subset \binom{[n]}{k}$, define the following locally closed subvariety of $\Mato(k,n)$.
 \[ \Mato(\cM) : =  \{ A\in \Mato(k,n) \mid \text{$\cM$ is the minimal positroid containing the column matroid of $A$} \} \]
 By Proposition \ref{prop: equivposi}, this could be equivalently defined as the set of matrices with a fixed Grassman necklace, reverse Grassman necklace, or bounded affine permutation.
 
These subvarieties fit into a decomposition of $\Mato(k,n)$. 
\[ \Mato(k,n) = \bigsqcup_{\text{positroids }\cM} \Mato(\cM)\]
The action of $GL_n$ preserves these subvarieties, and so we may consider their quotient varieties.
\[ \tPio(\cM) := SL_k\backslash \Mato(\cM) \subset \widetilde{Gr(k,n)} \]
\[ \Pio(\cM) := GL_k\backslash \Mato(\cM) \subset {Gr(k,n)} \]
The variety $\Pio(\cM)$ is called the \newword{open positroid variety} associated to the positroid $\cM$.
Again, these subvarieties fit into decompositions.
\[ \widetilde{Gr(k,n)} = \bigsqcup_{\text{positroids }\cM} \tPio(\cM),\;\;\;\;\;\; Gr(k,n) = \bigsqcup_{\text{positroids }\cM} \Pio(\cM) \]
\begin{remark}
These decompositions can also be defined as the common refinement of all cyclic permutations of the Schubert decompositions of $\widetilde{Gr(k,n)}$ and $Gr(k,n)$, by \cite[Lemma~5.3]{KLS13}.
\end{remark}

We write $\Pi(\cM)$, $\tPi(\cM)$ and $\Mat(\cM)$ for the closures of $\Pio(\cM)$, $\tPio(\cM)$ and $\Mato(\cM)$ in $Gr(k,n)$, $\widetilde{Gr(k,n)}$ and $\Mat(k,n)$ respectively. 
The reduced ideal of $\tPi(\cM)$ in $\widetilde{Gr(k,n)}$ is generated by the Pl\"ucker coordinates $\Delta_I$ for $I \not \in \cM$, by \cite[Theorem~5.15]{KLS13}.
Each of $\Pi(\cM)$, $\tPi(\cM)$ and $\Mat(\cM)$ has codimension $\ell(\pi)$ in $Gr(k,n)$, $\widetilde{Gr(k,n)}$ and $\Mat_{k \times n}$ respectively. 
See~\cite{KLS13} for this and many other many excellent properties of these varieties.

\section{Matchings of bipartite graphs in the disc} \label{sec matchings}

In this section, we define the \emph{boundary measurement map}, which uses the matchings on a bipartite graph $G$ embedded in a disc to define a map from an algebraic torus to $\tPi(\cM)$. 


\subsection{Positroids from matchings} \label{sec matching to positroid}

Throughout this paper, $G$ will denote a bipartite graph embedded in the disc, with the following additional data.
\begin{itemize}
	\item A coloring of each internal vertex as either \emph{black} or \emph{white}, such that adjacent internal vertices do not have the same color.\footnote{Unlike the introduction, we do not assume that the internal vertices adjacent to the boundary are white.}
	\item An indexing of the boundary vertices $1,2,...,n$ in clockwise order.
\end{itemize}
Additionally, we make the following assumptions.
\begin{itemize}
\item Every boundary vertex has degree 1 and is adjacent to an internal vertex.
\item There is at least one matching of $G$.
\end{itemize}
We let $\partial G$ and $V$ denote the sets of boundary and internal vertices of $G$, respectively.  The set of edges and faces of $G$ will be denoted $E$ and $F$, respectively.

A \newword{matching} of $G$ is a subset $M\subset E$ such that each internal vertex of $G$ is contained in a unique edge in $M$.  Given a matching $M$, define its \newword{boundary} $\partial M$ by
\begin{multline*}
\partial M = \{ i \in [n] \mid \text{vertex $i$ is contained in $M$ and $i$ is adjacent to a white internal vertex} \} \\  \cup  \{ i \in [n] \mid\text{vertex $i$ is not contained in $M$ and $i$ is adjacent to a black internal vertex} \}
\end{multline*}
Note that, if each boundary vertex is adjacent to a white vertex, then $\partial M \subset [n]$ indexes the boundary vertices contained in edges of $M$.  The boundary of each matching of $G$ has size
\begin{align*}
k :=  \#\text{(white vertices)} - \#\text{(black vertices)} + \#\text{(black vertices adjacent to the boundary)}
 \end{align*}
 

Not every $k$-element subset of $[n]$ may be the boundary of some matching of $G$. The following theorem gives a remarkable characterization the possible boundaries of matchings of $G$.
 
 \begin{thm} \label{theorem is a positroid}
 For a graph $G$ as above, the set
 \[ \cM  :=\left\{ I \in \binom{[n]}{k} \mid \text{there exists a matching $M$ with } \partial M=I \right\} \]
 is a positroid.  Every positroid can be realized by some graph $G$.
 \end{thm}
 
 \noindent The positroid $\cM$ will be called the \newword{positroid of $G$}.

\begin{proof}[Proof Sketch]
We need to translate between the language of matchings used in this paper and the language of loop erased walks through perfectly oriented graphs used in~\cite{Pos}.
In the language of loop erased walks, \cite[Theorem~4.11]{Pos} says that, for every perfectly oriented planar graph $H$, there is a positroid whose nonzero coordinates are targets of 
loop erased walks in $H$ and \cite[Theorem~4.12]{Pos} says that every positroid is the nonzero targets of the loop erased walks in some graph $H$.
\cite{Tal08} shows that targets of loop erased walks in $H$ are the same as targets of noncrossing paths through $H$.
\cite{PSW09} describes how to translate between flows in perfectly oriented graphs and matchings in bipartite graphs which have at least one matching.
\end{proof}

For the majority of the paper, we focus on those graphs which realize their positroid efficiently. A \newword{reduced graph} will be a graph $G$ satisfying the above assumptions, and such that
\begin{itemize}
	\item Every component of $G$ contains at least one boundary vertex.
	\item Every internal vertex of degree $1$ is adjacent to a boundary vertex.\footnote{Such a vertex is called a \emph{lollipop}; see Section \ref{sec bridge}.}
	\item The number of faces of $G$ is minimal among graphs with the same positroid.
\end{itemize}
An equivalent characterization of reduced graphs is given in Theorem \ref{thm strands work}.

\begin{remark}
Our version of `graphs' adds a bipartite assumption to Postnikov's \emph{plabic graphs}; this is necessary for matchings to work as desired. Our version of `reduced' combines Postnikov's \emph{leafless} and \emph{reduced} assumptions \cite[Definition 12.5]{Pos}, for simplicity.
%
\end{remark}

 
 \subsection{The boundary measurement map}
The positroid of $G$ gives a rough characterization of what matchings occur in $G$, in terms of possible boundaries.  We can get a much finer description of the matchings in $G$ using \emph{partition functions}.
Let $\{z_e\}_{e\in E}$ be a set of variables indexed by the set of edges $E$.  Each matching $M\subset E$ defines a monomial 
\[z^M := \prod_{e\in M} z_e\]
For any $k$-element subset $I\subset [n]$, the \newword{partition function} $D_I$ is the sum of the monomials of matchings with boundary $I$.
\[ D_I : = \sum_{\substack{ \text{matchings }M \\ \text{with }\partial M = I}} z^M \]

Each partition function defines a polynomial map $\CC^E\rightarrow \CC$, and collectively, they define a polynomial map $\CC^E\rightarrow \CC^{\binom{[n]}{k}}$.  Letting $\mathbb{G}_m$ denote the non-zero complex numbers, we will be interested in the restriction to $\mathbb{G}_m^E\subset \CC^E$, denoted by
\[ \tDD : \mathbb{G}_m^E\longrightarrow \CC^{\binom{[n]}{k}} \]
Specifically, $\tDD$ sends a point $(z_e)_{e\in E}\in \mathbb{G}_m^E$ to 
\[\left( D_I(z_e)\right)_{I\in \binom{[n]}{k}} \in \CC^{\binom{[n]}{k}}\]
Generally, there are numerous relations among the polynomials $D_I$, and so this map is far from dense.  Its image is characterized by the following theorem.

\begin{theorem} \label{theorem obey Plucker relations}
For a graph $G$ with positroid $\cM$, the map $\tDD$ lands in $\tPi(\cM)$; that is,
\[ \tDD:\mathbb{G}_m^E \longrightarrow \tPi(\cM) \subset \CC^{\binom{[n]}{k}} \]
\end{theorem}

\begin{proof}[Proof sketch]
The image of $\tDD$ lands in $\widetilde{Gr(k,n)}$; see \cite[Corollary 5.6]{Pos} (in the language of loop erased walks) or \cite[Theorem 2.4]{Lam16}.

We must further check that $\tDD(\GG_m^E)$ lands in $\tPi(\cM)$. By \cite[Theorem 5.15]{KLS13}, $\tPi(\cM)$ is cut out of $\widetilde{Gr(k,n)}$ by the vanishing of the Pl\"ucker coordinates $p_I$ for $I \not \in \cM$. 
By definition of the positroid $\cM$ associated to $G$, if $I \not \in \cM$, then there are no matchings of $G$ with boundary $I$, so $D_I(z)=0$ for those $I$.
\end{proof}

\begin{remark} \label{Kuo credit}
It is difficult to say who deserves the credit for Theorems~\ref{theorem is a positroid} and~\ref{theorem obey Plucker relations}. 
As discussed in the proofs, Postnikov~\cite{Pos} proved these results in the language of loop erased walks and Talaska~\cite{Tal08} transformed them to the language of flows. Postnikov, Speyer and Williams~\cite{PSW09} were the first to transform flows to matchings but did not point out these particular consequences.

The fact that the matching partition functions obey three term Pl\"ucker relations was observed earlier by Kuo~\cite{Kuo04}; this fact is often referred to as Kuo condensation by connoisseurs of matchings. Kuo's result strongly suggests that the map $\tDD$ lands in $\widetilde{Gr(k,n)}$ but does not prove it, since the ideal of $\widetilde{Gr(k,n)}$ is not generated by the $3$-term Pl\"ucker relations. (For example, the point $p_{123}=p_{456}=1$, all other $p_{ijk}=0$ in $\CC^{\binom{[6]}{3}}$ obeys all three term Pl\"ucker relations but is not in $\widetilde{Gr(3,6)}$.) 
The lecture notes of Thomas Lam~\cite{Lam16} may be the first place that these results appear explicitly in public.
See also~\cite{Spe15} for a short direct proof that the partition functions of matchings are the Pl\"ucker coordinates of a point on the Grassmannian.
\end{remark}

\subsection{Gauge transformations} \label{sec gauge}

Each internal vertex $v\in V$ determines an action of $\mathbb{G}_m$ on $\mathbb{G}_m^E$ by \newword{gauge transformation}.  If $v$ is an internal vertex of $G$ and $t$ a nonzero complex number, then send $(z_e)$ in $\GG_m^E$ to $(z_e')$ in $\GG_m^E$ by
\[ z'_e = \begin{cases} t z_e & v \in e \\ z_e & \mbox{otherwise} \end{cases}. \]
The gauge transformations combine to give an action of $\mathbb{G}_m^V$ on $\mathbb{G}_m^E$, called the \newword{gauge action}.

Let $\pi:\mathbb{G}_m^V\rightarrow \mathbb{G}_m$ be the map which sends $(t_v)_{v\in V}$ to the product of the coordinates $\prod_{v\in V} t_v$.  This is a group homomorphism, so its kernel is a subgroup of $\mathbb{G}_m^V$ which merits its own notation.\footnote{This notation is potentially misleading; there is no distinguished choice of isomorphism $\GG_m^{V-1}\simeq (\GG_m)^{|V|-1}$.}
\[ G_m^{V-1} := \mathrm{ker}(\pi:\mathbb{G}_m^V\rightarrow \mathbb{G}_m) \]
If $M$ is a matching of $G$, then the gauge action by $t=(t_v)_{v\in V}\in \GG_m^V$ acts on each of the previously defined functions as follows.
\[(t\cdot z)^M = \pi(t)z^M\hspace{.75cm} D_I(t\cdot z) = \pi(t)D_I(z) \hspace{.75cm} \tDD(t\cdot z) = \pi(t)\tDD(z) \]
Consequently, each of these functions is $\GG_m^{V-1}$-invariant, and so $\tDD$ descends to a map
\[ \tDD: \GG_m^E/\GG_m^{V-1} \longrightarrow \tPi(\cM) \]

The rational projection $\tPi(\cM)\dashedrightarrow \Pi(\cM)$ quotients by the action of simultaneously scaling of the Pl\"ucker coordinates. Hence, the composition $\GG_m^E\rightarrow \tPi(\cM)\dashedrightarrow \Pi(\cM)$ is invariant under the full gauge group $\GG_m^V$, and so this composition descends to a rational map, which we denote by $\DD$.
\[ \DD: \GG_m^E/\GG_m^{V} \;\tikz[baseline={([yshift=-3pt]current bounding box.center)}]{\draw[dashed,->] (0,0) to (.5,0);}\; \Pi(\cM) \]
%
Postnikov called $\DD$ the \emph{boundary measurement map} of $G$ and studied many of its properties, particularly its relation to total positivity. As an abuse of terminology, we refer to both $\DD$ and $\tDD$ as \newword{boundary measurement maps}.

\begin{remark}
In general, the map $\DD$ can be rational, as it is not defined on the $\tDD$-preimage of the origin in $\tPi(\cM)$. However, Proposition~\ref{in open cell} will imply that this preimage is empty whenever $G$ is reduced, and thus $\DD$ is defined on all of $\GG_m^E/\GG_m^V$. Until this issue is resolved, we will dodge it by stating our results in terms of $\tDD$, but the map $\DD$ provides much of our motivation.
\end{remark}


%
%

\subsection{Transformations between planar graphs} \label{sec transformations}

There are several local manipulations of a planar graph $G$ which do not change the corresponding positroid $\cM$.
The study of such transformations was systematized by Postnikov~\cite{Pos} and we follow his terminology\footnote{Our transformations differ from Postnikov's slightly, because our graphs are required to be bipartite.}; see Ciucu~\cite{Ciu98} and Propp~\cite{Pro03} for earlier precedents.

  In each case, the transformation will produce a new graph written $G'$, with edge set written $E'$, etc.  Additionally, for each transformation, we define a map $\mu:\mathbb{G}_m^E/\mathbb{G}_m^{V-1}\rightarrow \mathbb{G}_m^{E'}/\mathbb{G}_m^{V'-1}$ such that $\tDD'\circ \mu=\tDD$.\footnote{For the last move (\emph{urban renewal}), the map $\mu$ will be rational, and therefore only defined on a dense subset.}  In each case, the map $\mu$ will be defined in terms of a map $\mu_e:\mathbb{G}_m^E\rightarrow \mathbb{G}_m^{E'}$.  Points in $\mathbb{G}_m^E$ are equivalent to assigning a non-zero complex number to each edge in $E$, so the map $\mu_e$ will be defined by manipulating these \newword{edge weights}.  

Postnikov describes two classes of transformations -- \newword{moves}, which do not change the dimension of the torus $\mathbb{G}_m^E/\mathbb{G}_m^{V-1}$, and \newword{reductions} which do. We will only need moves, the reductions may be found in \cite[Section 12]{Pos}. The inverse of each move is likewise considered a move.


\begin{itemize}
	\item \emph{Contracting/expanding a vertex.} Any degree 2 internal vertex not adjacent to the boundary can be deleted, and the two adjacent vertices merged, as in Figure \ref{fig: equiv}.  This operation can also be reversed, by splitting an internal vertex into two vertices and inserting a degree 2 vertex of the opposite color between them (and giving each new edge a weight of $1$).
	
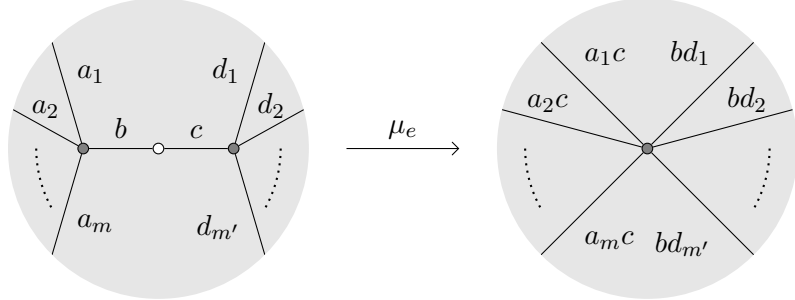
\begin{figure}[h!t]
\begin{tikzpicture}
\begin{scope}[scale=.5]
\begin{scope}
	\path[fill=black!10] (0,0) circle (4);
	\path[clip] (0,0) circle (4);
	\node[dot,fill=black!50] (a) at (-2,0) {};
	\node[dot,fill=white] (b) at (0,0) {};
	\node[dot,fill=black!50] (c) at (2,0) {};
	\draw (a) to node[above right] {$a_1$} (135:4);
	\draw (a) to node[above] {$a_2$} (165:4);
	\draw[dotted,thick] (180:3.25) arc (180:210:3.25);
	\draw (a) to node[below right] {$a_m$} (225:4);
	\draw (a) to node[above] {$b$} (b) to node[above] {$c$} (c);
	\draw (c) to node[above left] {$d_1$} (45:4);
	\draw (c) to node[above] {$d_2$} (15:4);
	\draw[dotted,thick] (0:3.25) arc (0:-30:3.25);
	\draw (c) to node[below left] {$d_{m'}$} (-45:4);
\end{scope}
\draw[-angle 90] (5,0) to node[above] {$\mu_e$} (8,0);
\begin{scope}[xshift=13cm]
	\path[fill=black!10] (0,0) circle (4);
	\path[clip] (0,0) circle (4);
	\node[dot,fill=black!50] (a) at (0,0) {};
	\draw (a) to (135:1.5) to node[above right] {$a_1c$} (135:4);
	\draw (a) to (165:1.5) to node[above] {$a_2c$} (165:4);
	\draw[dotted,thick] (180:3.25) arc (180:210:3.25);
	\draw (a) to (225:1.5) to node[below right] {$a_mc$} (225:4);
	\draw (a) to (45:1.5) to node[above left] {$bd_1$} (45:4);
	\draw (a) to (15:1.5) to node[above] {$bd_2$} (15:4);
	\draw[dotted,thick] (0:3.25) arc (0:-30:3.25);
	\draw (a) to (-45:1.5) to node[below left] {$bd_{m'}$} (-45:4);
\end{scope}
\end{scope}
\end{tikzpicture}
\caption{Contracting a degree 2 white vertex}
\label{fig: equiv}
\end{figure}

Here, the map $\mu:\mathbb{G}_m^E/\mathbb{G}_m^{V-1}\rightarrow \mathbb{G}_m^{E'}/\mathbb{G}_m^{V'-1}$ is induced by the map $\mu_e$ in Figure \ref{fig: equiv}.  The map $\mu$ is a regular isomorphism, and so the boundary measurement maps $\tDD$ and $\tDD'$ have the same image.  Notice that, by repeatedly expanding vertices of degree $\geq4$, we may always arrive at a graph with vertex degrees no more than $3$.

	\item \emph{Removing/adding a boundary-adjacent vertex.} Any degree 2 internal vertex adjacent to the boundary can removed, and the two adjacent edges can be made into one edge, as in Figure \ref{fig: boundaryvert}.  This operation can also be reversed, by adding a degree 2 vertex in the middle of a boundary-adjacent edge (and giving the new boundary-adjacent edge a weight of $1$).
	
\begin{figure}[h!t]
\begin{tikzpicture}
\begin{scope}[scale=.45]
\begin{scope}
	\path[clip] (0,0) circle (4);
	\draw[thick,fill=black!10] (-5,-5) rectangle (2,5);
	\node[dot,fill=black!50] (a) at (-2,0) {};
	\node[dot,fill=white] (b) at (0,0) {};
	\draw (a) to node[above right] {$a_1$} (135:4);
	\draw (a) to node[above] {$a_2$} (165:4);
	\draw[dotted,thick] (180:3.25) arc (180:210:3.25);
	\draw (a) to node[below right] {$a_m$} (225:4);
	\draw (a) to node[above] {$b$} (b) to node[above] {$c$} (2,0);
\end{scope}
\draw[-angle 90] (5,0) to node[above] {$\mu_e$} (8,0);
\begin{scope}[xshift=13cm]
	\path[clip] (0,0) circle (4);
	\draw[thick,fill=black!10] (-5,-5) rectangle (2,5);
	\node[dot,fill=black!50] (a) at (0,0) {};
	\draw (a) to (135:1) to node[above right] {$a_1c$} (135:4);
	\draw (a) to (165:1) to node[above] {$a_2c$} (165:4);
	\draw[dotted,thick] (180:3.25) arc (180:210:3.25);
	\draw (a) to (225:1) to node[below right] {$a_mc$} (225:4);
	\draw (a) to node[above] {$b$} (2,0);
\end{scope}
\end{scope}
\end{tikzpicture}
\caption{Removing a degree 2 white vertex adjacent to the boundary}
\label{fig: boundaryvert}
\end{figure}
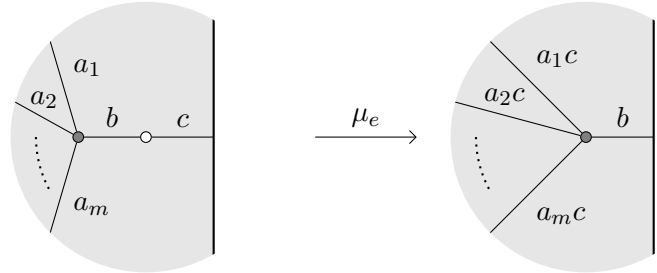

Here, the map $\mu:\mathbb{G}_m^E/\mathbb{G}_m^{V-1}\rightarrow \mathbb{G}_m^{E'}/\mathbb{G}_m^{V'-1}$ is induced by the map $\mu_e$ in Figure \ref{fig: boundaryvert}.  The map $\mu$ is a regular isomorphism, and so the boundary measurement maps $\tDD$ and $\tDD'$ have the same image.  Notice that, by adding white vertices between black vertices and the boundary as necessary, we may always arrive at a graph with only white vertices adjacent to the boundary.

	\item \emph{Urban renewal.} At an internal face of $G$ with four edges, the transformation in Figure \ref{fig: urban} is called \newword{urban renewal}.  
\begin{figure}[h!t]
\begin{tikzpicture}
\begin{scope}[scale=.65]
\begin{scope}
	\path[fill=black!10] (0,0) circle (4);
	\path[clip] (0,0) circle (4);
	\node[dot,fill=black!50] (a) at (-1.5,1.5) {};
	\node[dot,fill=white] (b) at (1.5,1.5) {};
	\node[dot,fill=white] (d) at (-1.5,-1.5) {};
	\node[dot,fill=black!50] (c) at (1.5,-1.5) {};
	\draw (a) to node[above right] {$a_1$} (135:4);
	\draw (b) to node[above left] {$a_2$} (45:4);
	\draw (d) to node[below right] {$a_4$} (225:4);
	\draw (c) to node[below left] {$a_3$} (-45:4);
	\draw (a) to node[above] {$b_1$} (b) to node[right] {$b_2$} (c) to node[below] {$b_3$} (d) to node[left] {$b_4$} (a);
\end{scope}
\draw[-angle 90] (5,0) to node[above] {$\mu_e$} (7,0);
\begin{scope}[xshift=12cm]
	\path[fill=black!10] (0,0) circle (4);
	\path[clip] (0,0) circle (4);
	\node[dot,fill=black!50] (a') at (-2,2) {};
	\node[dot,fill=white] (b') at (2,2) {};
	\node[dot,fill=white] (d') at (-2,-2) {};
	\node[dot,fill=black!50] (c') at (2,-2) {};
	\node[dot,fill=white] (a) at (-1,1) {};
	\node[dot,fill=black!50] (b) at (1,1) {};
	\node[dot,fill=black!50] (d) at (-1,-1) {};
	\node[dot,fill=white] (c) at (1,-1) {};
	\draw (a) to (a') to node[above right] {$a_1$} (135:4);
	\draw (b) to (b') to node[above left] {$a_2$} (45:4);
	\draw (d) to (d') to node[below right] {$a_4$} (225:4);
	\draw (c) to (c') to node[below left] {$a_3$} (-45:4);
	\draw (a) to node[above] {$\frac{b_3}{b_1b_3+b_2b_4}$} (b) to node[right] {$\frac{b_4}{b_1b_3+b_2b_4}$} (c) to node[below] {$\frac{b_1}{b_1b_3+b_2b_4}$} (d) to node[left] {$\frac{b_2}{b_1b_3+b_2b_4}$} (a);
\end{scope}
\end{scope}
\end{tikzpicture}
\caption{Urban renewal at a square face (unlabeled edges have weight $1$)}
\label{fig: urban}
\end{figure}
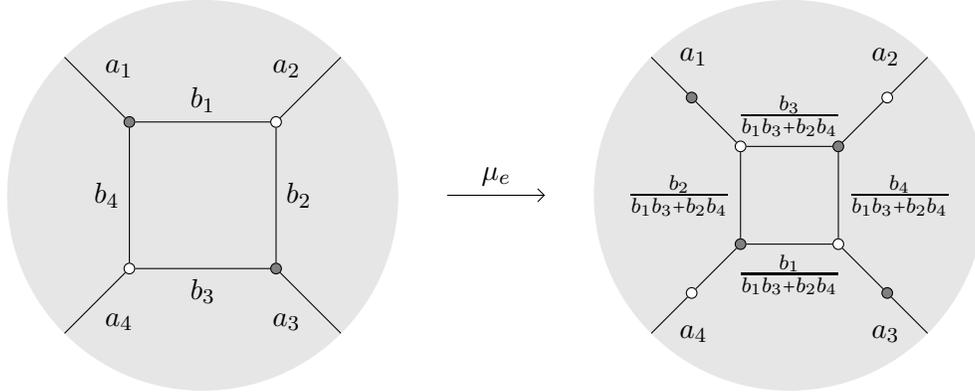

Here, the definition of $\mu$ is more subtle.  Let $\mu_e:\GG_m^E\dashedrightarrow \GG_m^{E'}$ be the rational map described in Figure \ref{fig: urban}.
It is easy to check that $\mu_e$ descends to a rational map $\overline{\mu}_e: \GG_m^E/\GG_m^V \longrightarrow \GG_m^{E'}/\GG_m^{V'}$. 
However, it does not descend to a rational map $\GG_m^E/\GG_m^{V-1} \longrightarrow \GG_m^{E'}/\GG_m^{V'-1}$. 
To fix this, choose an arbitrary vertex $v$ of $G'$ and let $\widehat{\mu}_{e,v}$ be the rational map $\GG_m^E\rightarrow \GG_m^{E'}$ which first applies $\mu_e$ and then acts at vertex $v$ by the gauge transformation by $b_1 b_3+b_2 b_4$. Then $\widehat{\mu}_{e,v}$ descends to a map on the quotient tori, from $\GG_m^E/\GG_m^{V-1}$ to $\GG_m^{E'}/\GG_m^{V'-1}$, and this quotient map is independent of the choice of $v$.
This quotient map is the map $\mu$. If we are only working with $\GG_m^E/\GG_m^V$, and hence with $Gr(k,n)$ rather than $\widetilde{Gr(k,n)}$, we may think in terms of the simpler map $\mu_e$.

%

\end{itemize}

\begin{thm}[{\cite[Theorem 12.7]{Pos}}]
Any two reduced graphs with the same positroid can be transformed into each other by the above moves.
\end{thm}
%

We conclude the section with the following observation which will be of use in Sections~\ref{sec extremal} and~\ref{sec lattice}.

\begin{lem} \label{clean dual}
In a reduced graph, every face is a disc and no edge separates a face from itself, except edges connecting the boundary to a degree $1$ vertex.
\end{lem}

\begin{proof}
First, if some face is not a disc, then there is a component of $G$ which is not connected to the boundary. This contradicts part of the definition of reducedness.

Now, suppose that edge $e$ separates some face $F$ from itself. In $G \setminus e$, the face $F$ becomes an annulus, so there is a component $H$ of $G \setminus e$ which is not connected to the boundary.
If $H$ is a single vertex, then by the definition of reducedness, $e$ connected $H$ to the boundary, and $e$ is one of the allowed exceptions. If $H$ is a tree with more than one vertex, then it has at least two leaves. One of those leaves must not be an endpoint of $e$, and that leaf is a leaf in $G$ which is not adjacent to the boundary; contradiction. 

Now, suppose that $H$ is not a tree. If $H$ contains as many black as white vertices, then no matching of $G$ uses $e$. In this case, $G$ has matchings with the same boundaries as $G \setminus (H \cup e)$ and the latter graph has fewer faces.
If the number of black and white vertices of $H$ differ by one, then every matching of $G$ uses $e$. Let $f_1$, $f_2$, \dots, $f_r$ be the other edges incident on the endpoint of $e$ not in $G$. Then $G$ has matchings with the same boundaries as $G \setminus (H \cup e \cup \bigcup f_i)$ and the latter graph has fewer faces.
If the number of black and white vertices of $H$ differ by more than one, then $G$ has no matchings at all. We have reached a contradiction in every case.
\end{proof}

\section{Strands and Postnikov diagrams} \label{sec strand}

\subsection{Postnikov's theory of strands}

A graph $G$ satisfying the assumptions of the previous section is equivalent to a collection of oriented \newword{strands} in the disc connecting the marked points on the boundary, satisfying certain restrictions on how they are allowed to cross.

The strands of a graph $G$ satisfying the assumptions of Section \ref{sec matching to positroid} are constructed as follows.  Each edge intersects two strands as in Figure \ref{fig: strands}: at a internal edge, the two strands cross transversely at the midpoint; at a boundary edge, the two strands terminate at the boundary vertex.
These strands are connected to each to each other in the most natural way, so that each corner of each face is cut off by a segment of a strand; see the example in Figure~\ref{fig: introstrands}.  We consider strands up to \emph{ambient homotopy}: homotopies which don't change the intersections.

\begin{figure}[h!t]
\begin{tikzpicture}
\begin{scope}
\begin{scope}[scale=.35]
	\node at (270:5) {At an internal edge};
	\path[clip] (0,0) circle (4);
	\draw[thick,fill=black!10] (-5,-4.5) rectangle (5,5);
	\node[dot,fill=black!50] (a) at (0,2) {};
	\node[dot,fill=white] (b) at (0,-2) {};
	\draw (b) to (a);
	\draw (a) to (60:4);
	\draw (a) to (80:4);
	\draw[dotted,thick] (90:3.25) arc (90:108:3.25);
	\draw (a) to (120:4);
	\draw (b) to (240:4);
	\draw (b) to (260:4);
	\draw[dotted,thick] (270:3.25) arc (270:288:3.25);
	\draw (b) to (300:4);
	\draw[oriented,out=-60,in=135] (130:4) to (0,0);
	\draw[oriented,out=-45,in=120] (0,0) to (-50:4);
	\draw[oriented,out=60,in=225] (-130:4) to (0,0);
	\draw[oriented,out=45,in=240] (0,0) to (50:4);
\end{scope}
\begin{scope}[xshift=5cm,scale=.35]
	\node at (270:4.5) {\shortstack{At a boundary edge \\ adjacent to a black vertex}};
	\path[clip] (0,0) circle (4);
	\draw[thick,fill=black!10] (-5,-2) rectangle (5,5);
	\node[dot,fill=black!50] (a) at (0,1) {};
	\draw (0,-2) to (a);
	\draw (a) to (60:4);
	\draw (a) to (80:4);
	\draw[dotted,thick] (90:3.25) arc (90:108:3.25);
	\draw (a) to (120:4);
	\draw[oriented,out=-75,in=135] (130:4) to (0,-2);
	\draw[oriented,out=45,in=255] (0,-2) to (50:4);
\end{scope}
\begin{scope}[xshift=10cm,scale=.35]
	\node at (270:4.55) {\shortstack{At a boundary edge \\ adjacent to a white vertex}};
	\path[clip] (0,0) circle (4);
	\draw[thick,fill=black!10] (-5,-2) rectangle (5,5);
	\node[dot,fill=white] (a) at (0,1) {};
	\draw (0,-2) to (a);
	\draw (a) to (60:4);
	\draw (a) to (80:4);
	\draw[dotted,thick] (90:3.25) arc (90:108:3.25);
	\draw (a) to (120:4);
	\draw[antioriented,out=-75,in=135] (130:4) to (0,-2);
	\draw[antioriented,out=45,in=255] (0,-2) to (50:4);
\end{scope}
\end{scope}
\end{tikzpicture}
\caption{Strands in neighborhood of each type of edge}
\label{fig: strands}
\end{figure}
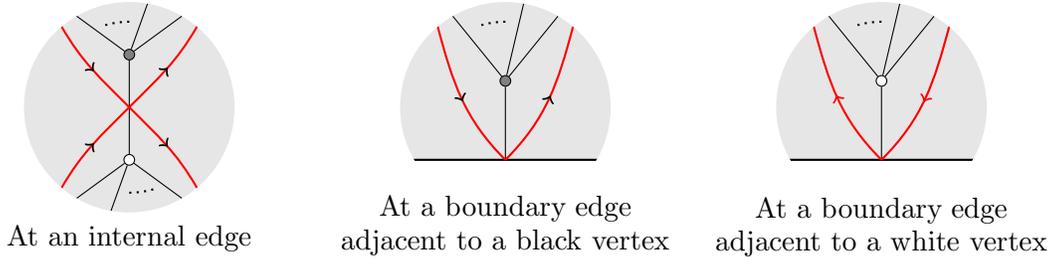

The resulting collection of oriented, immersed curves will have the following properties.
\begin{enumerate}
	\item Each strand either begins and ends at marked points on the boundary, or is a closed loop.
	\item Intersections between strands are `generic' with respect to homotopy; that is, all intersections are transverse crossings between two strands, and there are finitely many intersections.
	\item Following any given strand, the other strands alternately cross it from the left and from the right.
\end{enumerate}
Furthermore, if $G$ is reduced, the strands satisfy additional properties.
\begin{enumerate}
\setcounter{enumi}{3}
	\item No strand is a closed loop.
	\item No strand intersect itself, except a strand which begins and ends at the same marked point.
	\item If we consider any two strands $\gamma$ and $\delta$ and their finite list of intersection points, then they pass through their intersection points in opposite orders.
\end{enumerate}
Postnikov has demonstrated that these properties characterize strands of planar graphs.

\begin{thm}[{\cite[Corollary~14.2]{Pos}}] \label{thm strands work}
A collection of oriented, immersed curves in a marked disc which satisfies properties $(1)-(3)$ are the strands of a unique graph $G$ satisfying the assumptions of Section \ref{sec matching to positroid}.  The graph $G$ is reduced if and only if the strands satisfy properties $(4)-(6)$.
In this case, the strand starting at boundary vertex $a$ ends at boundary vertex $\pi(a) \bmod n$, where $\pi$ is the associated decorated permutation.\footnote{If $\pi(a) \equiv a \bmod n$ then the component containing boundary vertex $a$ will have a single internal vertex $v$; we will have $\pi(a)=a$ if $v$ is black and $\pi(a)=a+n$ if $v$ is white.}
\end{thm}



\subsection{Face labels} \label{sec face labels}

Let $G$ be a reduced graph.  By Properties $(1)$, $(4)$, and $(5)$, each strand in the corresponding Postnikov diagram divides the disc into two connected components: the component to the \emph{left} of the strand, and the component to the \emph{right} of the strand.
At each face, we may consider the set of strands on which it is on the left side.\footnote{A strand may cut off a corner of a face of $G$; in this case, we ask which side the remainder of the face is on.}  

\begin{prop}
Each face of a reduced graph $G$ is to the left of $k$ strands, where $k$ is the rank of the positroid of $G$.
\end{prop}

\begin{proof}
If $F_1$ and $F_2$ are adjacent faces of $G$ separated by an edge $e$, then there are two strands which pass through $e$, and the labels of $F_1$ and $F_2$ differ by deleting one of these strands and inserting the other. So all faces have labels of the same size. For the boundary faces, this is verified as a portion of~\cite[Proposition~8.3.(1)]{OPS15}.
\end{proof}

To each face $f$ of $G$, we would like to associate a $k$-element subset of $[n]$, and hence a Pl\"ucker coordinate on $\widetilde{Gr(k,n)}$.  The proposition gives two equally natural ways to do this.
\begin{itemize}
	\item \newword{Target-labeling.} Label each face by the targets of the strands it is left of.
	\[ \tI(f) := \{ a\in [n] \mid \text{$f$ is to the left of the strand ending at vertex $a$} \} \]
	\item \newword{Source-labeling.} Label each face by the sources of the strands it is left of.
	\[ \sI(f) := \{ a\in [n] \mid \text{$f$ is to the left of the strand beginning at vertex $a$} \} \]
\end{itemize}
We will use both of these conventions, and so we avoid choosing a preferred convention.
The $\bullet$ in our notation is meant to help the reader recall which notation refers to which convention.

These face labels generalize Grassman necklaces and  reverse Grassman necklaces as follows.
\begin{prop}\label{prop: boundarylabel}
Let $G$ be a reduced graph with positroid $\cM$, and let $\veccI$ and $\cevcI$ be the Grassman necklace and reverse Grassman necklace of $\cM$ (see Proposition \ref{prop: equivposi}).  If $f$ is the boundary face in $G$ between boundary vertices $i$ and $i+1$, then 
$\tI(f) = \vecI_{i+1},\text{ and } \sI(f) = \cevI_i$.
\end{prop}
\begin{proof}
Straightforward from the description of the starting and ending points of the strands at the end of Theorem~\ref{thm strands work}.
\end{proof}

\begin{remark}
The clash of notation in Proposition \ref{prop: boundarylabel} is unfortunate, but the notation $\sI$ and $\tI$ will work well with other notation that will come up more often than Grassman necklaces.
\end{remark}

We can now explain the motivation for the notation $a \impto_{\pi} b$. The following lemma was pointed out to us by Suho Oh.
\begin{lemma} \label{label implication}
Let $\pi$ be a decorated permutation and suppose that $a \impto_{\pi} b$. Let $G$ be a reduced graph for $\pi$ and let $f$ be a face of $G$. If $a \in \sI(f)$ then $b \in \sI(f)$.
\end{lemma}

\begin{proof}
Because the graph $G$ is reduced, the strands $a \to \pi(a)$ and $b \to \pi(b)$ cannot cross (see, e.g. \cite[Lemma 3.1]{MS14}). Therefore, any face to the left of $a \to \pi(a)$ is also to the left of $b \to \pi(b)$.
\end{proof}

Each face $f$ determines two Pl\"ucker coordinates, $\Delta_{\tI(f)}$ and $\Delta_{\sI(f)}$, on $\widetilde{Gr(k,n)}$, which we restrict to $\tPio(\cM)$ (where $\cM$ is the positroid of $G$).  We combine these into two maps to $\CC^F$ as follows.
\[ \begin{tikzpicture}[baseline=-.25cm]
	\node (a) at (0,0) {$\tPio(\cM)$};
	\node (b) at (2,0) {$\CC^F$};
	\draw[->,out=20,in=150] (a) to node[above] {$\tF$} (b);
	\draw[->,out=-20,in=-150] (a) to node[below] {$\sF$} (b);
\end{tikzpicture} 
\hspace{2cm}
\begin{array}{c}
\tF(p) := \left(\Delta_{\tI(f)}(p)\right)_{f\in F} \\
\sF(p) := \left(\Delta_{\sI(f)}(p)\right)_{f\in F}
\end{array}
\]
We emphasize these maps are only defined for \emph{reduced} graphs $G$.

\begin{remark}
There is an expectation that the set of Pl\"ucker coordinates $\{\Delta_{\tI(f)}\}_{f\in F}$ should be a cluster for a cluster structure on the coordinate ring of $\tPio(\cM)$.\footnote{A complete description of the conjectural cluster structure may be found in \cite{MS14}. Leclerc~\cite{Lec14} has recently placed a cluster structure on $\tPio(\cM)$, which we expect to coincide with this one, but the details are not yet checked.}  Scott proved this was true for the `uniform positroid' $\cM=\binom{[n]}{k}$, which corresponds to the dense positroid variety in $\widetilde{Gr(k,n)}$ \cite{Sco06}.  One consequence of this expectation is that these Pl\"ucker coordinates should satisfy a `Laurent phenomenon'.  Geometrically, this means that restricting $\tF$ to a rational map
\[ \begin{tikzpicture}[baseline=-.25cm]
	\node (a) at (0,0) {$\tPio(\cM)$};
	\node (b) at (2,0) {$\mathbb{G}_m^F$};
	\draw[->,dashed] (a) to node[above] {$\tF$} (b);
\end{tikzpicture} \]
this map should be an isomorphism from its domain of definition to $\mathbb{G}_m^F$.  By symmetry, the same result should hold for the source-labeled map $\sF$.  Theorem \ref{main theorem} will confirm these expectations.
\end{remark}

\begin{remark} \label{source labeling cluster}
The two conventions (target-labelling and source-labelling) do not always give the same cluster structure!  However, examples suggest that each cluster variable in the target-labelling cluster structure is a monomial in the frozen variables times a cluster variable for the source-labeling cluster structure. Chris Fraser~\cite{Fra15} has recently developed a theory of maps of cluster algebras which take cluster variables to cluster variables times monomials in frozen variables.
\end{remark}

\section{Matchings associated to faces} \label{sec extremal}

The goal of this section is to construct an isomorphism $\mathbb{G}_m^E/\mathbb{G}_m^{V-1}\rightarrow \GG_m^F$, or, equivalently, a system of coordinates on $\mathbb{G}_m^E/\mathbb{G}_m^{V-1}$ parameterized by $F$.
It will be more natural to work on the level of character lattices of these tori, where is equivalent to giving an isomorphism of lattices
\[\ZZ^F\oplus \ZZ^V \garrow{\sim} \ZZ \oplus \ZZ^E  \]
In this section, we construct such a map and its inverse.  The most important part of this map is the restriction $\ZZ^F\rightarrow \ZZ^E$, which is defined by a special matching associated to each face. We explicitly construct these matchings using \emph{downstream wedges}, though each may be defined abstractly as the minimal matching with a given boundary under a partial order (Remark \ref{rem: minimalmatching} and Appendix \ref{sec Propp}).

\subsection{Downstream wedges and a pair of inverse matrices}

Let $e$ be an edge in a reduced graph $G$.  There are two strands in the corresponding Postnikov diagram which intersect $e$, and by Property $(6)$ of Postnikov diagrams, the `downstream' half of each of these strands do not intersect each other except at $e$.  Hence, they divide the disc into two components; the \newword{downstream wedge} of $e$ is the component which does not contain $e$ (see Figure \ref{fig: downstreamwedge}).

A face $f$ is \newword{downstream} from an edge $e$ if the face $f$ is contained in the downstream wedge of $e$, ignoring any corners of $f$ that are cut off.  We can use this to distinguish between the two faces adjacent to an edge.  
We say that $f$ is \newword{directly downstream} of $e$ if $f$ is downstream of $e$ and $e$ is in the boundary of $f$.

Given a face $f\in F$ and an edge $e\in E$, define
\[ U_{ef}:= \left\{\begin{array}{cc}
1 & \text{$f$ is downstream from $e$} \\
0 & \text{otherwise} \end{array}\right\} \]
\[ \partial_{fe}:= \left\{\begin{array}{cc}
1 & \text{$e$ is an internal edge in the boundary of $f$} \\
1 & \text{$e$ is an external edge, and $f$ is directly downstream from $e$} \\
0 & \text{otherwise} \end{array}\right\} \]
Let $U_{E,F}$ and $\partial_{F,E}$ be the matrices with the above entries.

Similarly, given an internal vertex $v\in V$ and an edge $e\in E$, define
\[ U_{ev}:= \left\{\begin{array}{cc}
1 & \text{$v$ is downstream from $e$} \\
0 & \text{otherwise} \end{array}\right\} \]
\[ \partial_{ve}:= \left\{\begin{array}{cc}
1 & \text{$v$ is in the boundary of $e$} \\
0 & \text{otherwise} \end{array}\right\} \]
Let $U_{E,V}$ and $\partial_{V,E}$ be the matrices with the above entries.

For each face $f$, let 
\[ B_f := \text{\# of edges $e$ such that $f$ is directly downstream from $e$}\]
So $B_f$ is $\tfrac{1}{2} \# \partial f$ rounded either up or down. In particular, if $f$ is internal then $\# \partial f$ is even and $B_f = \tfrac{1}{2} \# \partial f$.

Let $B_{F,1}$ denote the $|F|\times 1$-matrix with entries $\{B_f\}$.  Finally, for any finite sets $A$ and $B$, let $1_{A,B}$ denote the $|A|\times |B|$-matrix of ones.

\begin{lemma}\label{lemma: downmatrix}
The pair of matrices
\[ \gmat{ 1_{F,1}-B_{F,1} &-\partial_{F,E} \\
1_{V,1} & \partial_{V,E} } \ \textrm{and} \ \gmat{1_{1,F} & 1_{1,V}\\ -U_{E,F} & -U_{E,V}}\]
are mutually inverse.
\end{lemma}

Once we have proved that these matrices are inverse, we will denote them by $\Monom$ and $\Monom^{-1}$ respectively.

\begin{remark}
The formulas defining the entries of $\Monom$ do not refer to strands, so this matrix makes sense for any bipartite graph embedded in a disc, reduced or not.
It it not hard to show that $\Monom$ is invertible and $\Monom^{-1}$ has integer entries whenever all components of $G$ are connected to the boundary of the disc. However, the entries of $\Monom^{-1}$ may not lie in $\{ -1,0,1 \}$ in this generality.
\end{remark}

\begin{proof}[Proof of Lemma \ref{lemma: downmatrix}]
The reduced graph $G$ gives a cellular decomposition of the disc (Lemma~\ref{clean dual}). It has $|E|+n$ edges (counting the $n$ boundary edges between boundary vertices) and $|V|+n$ vertices (counting the $|V|$ internal vertices and the $n$ boundary vertices).  Since the Euler characteristic of the disc is $1$,
\[ |F| - (|E| +n) + (|V|+n) =1 \]
Hence, $|F|+|V|=|E|+1$, and so both matrices in the statement of the lemma are square.  If one of the matrices is the left inverse of the other, then it is also the right inverse.  We check that
\begin{equation}\label{eq: matprod}
\gmat{1_{1,F} & 1_{1,V}\\ -U_{E,F} & -U_{E,V}}
\gmat{ 1_{F,1}-B_{F,1} & -\partial_{F,E} \\ 1_{V,1} & \partial_{V,E} }
\end{equation}
is the identity on each block.  

\emph{Upper left block.}  Since each edge has a unique face directly downstream, the sum $\sum_{f\in F} B_f$ counts each edge exactly once, so it is equal to $|E|$.

The upper left entry in the product \eqref{eq: matprod} is
\[ 1_{1,F}(1_{F,1}-B_{F,1})+1_{1,V}1_{V,1} = \sum_{f\in F} (1-B_f)+\sum_{v\in V} 1 = |F|-|E|+|V| =1 \]

\emph{Upper right block.}  If $e$ is an internal edge, then there are two faces $f$ with $\partial_{fe}=1$, and two vertices $v$ with $\partial_{ve}=1$.  If $e$ is a external edge, then there is one face $f$ with $\partial_{fe}=1$, and one vertex $v$ with $\partial_{ve}=1$.  In either case,
\[ -\sum_{f\in F}\partial_{fe} + \sum_{v\in V} \partial_{ve} =0 \]
and so the $1\times |E|$-matrix $1_{1,F}\partial_{F,E}-1_{1,V}\partial_{V,E}$ is zero.

\emph{Lower left block.}  Fix an edge $e\in E$, and consider the closure of the union of all the faces in the downstream wedge of $e$. 
This is homotopy equivalent to the downstream wedge itself; in particular it has Euler characteristic $1$.  
This closure has a cellular decomposition $\Delta$, given by the restriction of the graph $G$ and the boundary of the disc.  

The sum $\sum_{v\in V} U_{ev}$ counts the number of internal vertices of $G$ in the downstream wedge of $e$.  These vertices are all in $\Delta$, but this count misses two types of vertices.  Specifically,
\[ \text{(\# vertices in $\Delta$)} = A+B+\sum_{v\in V} U_{ev}\]
where $A$ is the number of internal vertices of $G$ which are contained in $\Delta$ but not in the downstream wedge of $e$, and $B$ is the number of boundary vertices of the disc contained in $\Delta$.  

Similarly, the sum $\sum_{f\in F} U_{ef}B_f$ counts edges in $G$ whose directly downstream face is in $\Delta$.  These are all in $\Delta$, but this count misses two kinds of edges in $\Delta$: edges in $G\cap \Delta$ whose directly downstream face is not in $\Delta$, and boundary edges of the disc which are contained in $\Delta$.  

There are $A$-many edges in $G\cap \Delta$ whose directly downstream face is not in $\Delta$.  To see this, observe that each vertex counted by $A$ is adjacent to two edges in $\Delta$; one of these edges has its directly downstream face in $\Delta$ and the other does not.  
Since there are $(B-1)$-many boundary edges contained in $\Delta$,
\[ \text{(\# edges in $\Delta$)} = A+(B-1) +\sum_{f\in F} U_{ef}B_f\]

It follows that 
\[ \sum_{v\in V}U_{ev}-\sum_{f\in F} U_{ef}B_f = \text{(\# vertices in $\Delta$)} - \text{(\# edges in $\Delta$)} -1\]
Since $\sum_{f\in F}U_{ef}$ is the number of faces in $\Delta$, we see that
\[ \sum_{f\in F} U_{ef}(1-B_f) +\sum_{v\in V} U_{ev} =  \sum_{f\in F} U_{ef}-\sum_{f\in F}U_{ef}B_f +\sum_{v\in V} U_{ev} = \chi(\Delta)-1=0\]
This holds for any edge, and so the $|E|\times 1$-matrix $U_{E,F}(1_{F,1}-B_{F,1})+U_{E,V}1_{V,1}$ is zero.

\emph{Lower right block.}
Fix an edge $e\in E$, and consider the sum 
\begin{equation}\label{eq: chamber1}
\sum_{f\in F} U_{ef}\partial_{fe'}-\sum_{v\in V} U_{ev}\partial_{ve'} 
\end{equation}
for all possible $e'$ in $E$.  When $e'=e$, the product $U_{ef}\partial_{fe}=1$ when $f$ is the face directly downstream from $e$, and all other terms are $0$, so Formula \eqref{eq: chamber1} evaluates to $1$.  For all other $e'$, the number of faces such that $U_{ef}\partial_{fe'}=1$ is equal to the number of vertices such that $U_{ev}\partial_{ve'}=1$; hence, Formula \eqref{eq: chamber1} evaluates to $0$.  As a consequence, the $|E|\times |E|$-matrix in the lower right of the product \eqref{eq: matprod} is
\[ U_{E,F}\partial_{F,E} - U_{E,V}\partial_{V,E} = Id_{E,E}\]
%
%
Hence, the product \eqref{eq: matprod} is the identity matrix.  
\end{proof}

The proof of Lemma \ref{lemma: downmatrix} is delightfully efficient.  The content of the lemma is eight identities relating the block entries, but the proof only had to verify four of them.  The other four identities are free; they are encoded in the following equation.
\begin{equation}\label{eq: othermatprod}
\gmat{ 1_{F,1}-B_{F,1} & -\partial_{F,E} \\ 1_{V,1} & \partial_{V,E} }
\gmat{1_{1,F} & 1_{1,V}\\ -U_{E,F} & -U_{E,V}}
=\gmat{Id_{F,F} & 0 \\ 0 & Id_{V,V} }
\end{equation}
In the next sections, we will reap the benefits of these identities.

\subsection{Matchings from downstream wedges}

To any face $f$, we may associate the set of edges such that $f$ is in its downstream wedge.
\[ \vecM(f) := \{e\in E \mid \text{$f$ is in the downstream wedge of $e$}\}\]
As a mnemonic, the arrow points towards $f$, just as the strands are directed from edge $e$ in the general direction of face $f$.

Two of the four identities contained in Equation \eqref{eq: othermatprod} have essential consequences for $\vecM(f)$.
\begin{thm}\label{thm: minmatchprop}
For any face $f\in F$, the set $\vecM(f)$ is a matching of $G$ with boundary $\sI(f)$, the source-indexed face label of $f$.  There are $B_f$-many edges $e$ in $\vecM(f)$ such that $\partial_{fe}=1$, and for any other face $f'\in F$, there are $(B_{f'}-1)$-many edges in $\vecM(f)$ such that $\partial_{f'e}=1$.
\end{thm}
\noindent At an internal face $f'$, the theorem states that the matching $\vecM(f)$ contains one fewer than half the edges in the boundary of $f'$, except when $f'=f$, in which case the matching contains half of the edges in the boundary of $f$ (the maximum possible for a matching).
\begin{proof}
First, the lower left block in Equation \eqref{eq: othermatprod} implies that, for any $v\in V$ and $f\in F$, we have
\[ 1=\sum_{e\in E} \partial_{ve}U_{ef} = \sum_{e\in \vecM(f)} \partial_{ve}\]
Equivalently, for any $v\in V$, the set $\vecM(f)$ contains one edge adjacent to $v$; hence, $\vecM(f)$ is a matching.

Inspecting Figure \ref{fig: downstreamwedge}, we see a face $f$ is in the downstream wedge of an edge connecting boundary vertex $a$ to a white vertex whenever $f$ is to the left of the strand beginning at $a$.  Similarly, $f$ is in the downstream wedge of an edge connecting boundary vertex $a$ to a black vertex whenever $f$ is to the right of the strand beginning at $a$.  Hence, the boundary of the matching $\vecM(f)$ is $\sI(f)$.

Next, the upper left block in Equation \eqref{eq: othermatprod}, implies that, for any $f,f'\in F$,
\[ B_{f'}-1+\delta_{ff'} =\sum_{e\in E} \partial_{f'e}U_{ef} =\sum_{e\in \vecM(f)} \partial_{f'e}\]
Hence, when $f=f'$, this sum is $B_f$, and when $f\neq f'$, this sum is $B_{f'}-1$.
\end{proof}

\begin{remark} \label{rem: minimalmatching}
Appendix \ref{sec Propp} demonstrates that the matching $\vecM(f)$ is the \emph{minimal matching} among all matchings with boundary $\sI(f)$, for a partial ordering generated by \emph{swiveling} matchings.
\end{remark}

\subsection{A torus isomorphism from minimal matchings} \label{sec character basis}

For any matching $M$, the associated monomial $z^M$ on $\GG_m^E$ is invariant under the action of the restricted gauge group $\GG_m^{V-1}$. Therefore, we may define a map
\[ \rM:\GG_m^E/\GG_m^{V-1} \longrightarrow \GG_m^F \]
whose coordinate at each face is the inverse to the matching associated to that face.
\begin{align*}
\left(\rM(z)\right)_{f\in F} &= z^{-\vecM(f)} 
= \prod_{e\in E} z_e^{-U_{ef}} = \prod_{\stackrel{e\in E}{\text{$f$ downstream from $e$}}} z_e^{-1}
\end{align*}

\begin{prop}\label{prop: isotori}
The map $\rM$ is an isomorphism between $\GG_m^E/\GG_m^{V-1}$ and $\GG_m^F$.
\end{prop}

\begin{proof}
%
Let $m\in \ZZ^E$ be such that $z^m$ is $\GG_m^{V-1}$-invariant. Consequently, there is some $d\in \ZZ$ such that, for each vertex $v$, the total degree of the edges adjacent to $v$ is $d$. Equivalently,
\[ \partial_{V,E}\cdot m = d\cdot 1_{V,1} \]
It follows that 
\[ \begin{pmatrix}
1_{F,1}-B_{F,1} & -\partial_{F,E} \\ 1_{V,1} & \partial_{V,E}
\end{pmatrix}
\begin{pmatrix}
-d \\ m 
\end{pmatrix} 
= \begin{pmatrix}
d (B_{F,1}- 1_{F,1})-(\partial_{F,E}\cdot m) \\ 0
\end{pmatrix}\]
By Lemma \ref{lemma: downmatrix}, this is equivalent to the matrix identity
\[ 
\begin{pmatrix}
-d \\ m 
\end{pmatrix} 
= \begin{pmatrix}
1_{1,F} & 1_{1,V}\\ -U_{E,F} & -U_{E,V}
\end{pmatrix}
\begin{pmatrix}
d (B_{F,1}- 1_{F,1})-(\partial_{F,E}\cdot m) \\ 0
\end{pmatrix}\]
This implies the following equality.
\[ m = -U_{E,F}(d (B_{F,1}-1_{F,1})-(\partial_{F,E}\cdot m) ) \]
Consequently,
\begin{align}\label{eq: Mpullback}
\rM(z)^{d (B_{F,1}-1_{F,1})-(\partial_{F,E}\cdot m)} = z^{-U_{E,F}(d\cdot (B_{F,1}-1_{F,1})-(\partial_{F,E}\cdot m) )} = z^m
\end{align}
Hence, every character on $\GG_m^E/\GG_m^{V-1}$ is the pullback of a character along $\rM$, and so the pullback map $\rM^*$ on character lattices is a surjection. Since a surjection between lattices of the same dimension is an isomorphism, $\rM^*$ is an isomorphism of lattices and $\rM$ is an isomorphism of tori.
\end{proof}

\begin{corollary}
The monomials $z^{\vecM(f)}$, as $f$ ranges over $F$, form a basis of characters of $\GG_m^E/\GG_m^{V-1}$.
\end{corollary}
\begin{proof}
The set of coordinates $x_f^{-1}$ (as $f$ runs over $F$) is a basis of characters for the torus $\GG_m^F$.
The pullback of these functions along $\rM$ are the minimal matching monomials $z^{\vecM(f)}$.
\end{proof}

Equation \eqref{eq: Mpullback} in the proof of Proposition \ref{prop: isotori} provides an explicit formula for writing a character of $\GG_m^E/\GG_m^{V-1}$ in terms of the $z^{\vecM(f)}$. We highlight a special case of this.

\begin{corollary}\label{coro: facePartition}
Let $z\in \GG_m^E$, and let $M$ be a matching of $G$. Then
\[ z^M =  \rM(z)^{(B_{F,1}-1_{F,1}) - \partial_{F,E}\cdot M} = \prod_{f\in F} \left(z^{\vecM(f)}\right)^{\#\{ e \in M \ : \ \partial_{fe}=1 \} - (B_f-1)}\]
For any $I\in \binom{[n]}{k}$, we have
\[ \DD_I\left(z\right) =  \sum_{\stackrel{\text{matchings }M }{\text{with } \partial M=I}}\rM(z)^{(B_{F,1}-1_{F,1}) - \partial_{F,E}\cdot M} = \sum_{\stackrel{\text{matchings }M }{\text{with } \partial M=I}} \prod_{f \in F} \left(z^{\vecM(f)}\right)^{\#\{ e \in M \ : \ \partial_{fe}=1 \} - (B_f-1)} \]
\end{corollary}

\begin{remark}
The proposal that matchings of a planar graph should be described by a generating functions whose variables are assigned to faces, and where the exponent of a face $f$ should be $B_f -1 - \#\{ e \in M \ : \ \partial_{fe}=1 \}$, first occurred in~\cite{Spe07}. 
\end{remark}


\subsection{The inverse map}

We now consider the inverse map to $\rM$, for which we use a separate notation.
\[ \lpartial:= \rM^{-1}:\GG_m^F \rightarrow \GG_m^E/\GG_m^{V-1} \]
Unfortunately, there is no natural lift of $\lpartial$ to a map $\GG_m^F\rightarrow \GG_m^E$, and so there is no natural way to write $\lpartial$ in terms of coordinates on $\GG_m^E$. The best we can do is the following, which involves a gauge transformation at an arbitrary vertex.

\begin{prop}\label{prop: partial}
For any $x\in \GG_m^F$, any lift of $\lpartial(x)$ to $\GG_m^E$ is $\GG_m^V$-equivalent to $\lpartial'(x)\in \mathbb{G}_m^E$, defined by\footnote{This justifies the notation $\lpartial$, as it is gauge-equivalent to the map on character lattices given by $-\partial_{F,E}$.}
\[ \lpartial'(x)_e:= \prod_{f\in F} x_f^{-\partial_{fe}} =
\left\{\begin{array}{cc}
\frac{1}{x_{f_1}x_{f_2}}& \text{for $e$ an internal edge between faces $f_1$ and $f_2$} \\
\frac{1}{x_{f}}& \text{for $e$ an external edge with directly downstream face $f$} \\
\end{array}\right\}
\]
Furthermore, the gauge transformation of $\lpartial'(x)$ at any vertex by the value $\prod_{f\in F}x_f^{B_f-1}$ is $\GG_m^{V-1}$-equivalent to $\lpartial(x)$; that is, it has the same image in $\GG_m^E/\GG_m^{V-1}$ as $\lpartial(x)$.
\end{prop}
\begin{proof}
As before, let $z^m$ be $\GG_m^{V-1}$-invariant, so that there is some $d\in \ZZ$ such that the total degree of $m$ at each vertex is $d$. Regardless of what vertex we perform the gauge transformation at, the total degree of $m$ at that vertex is $d$, and so the gauge-transformation scales the value of $z^m$ by $\left(\prod_{f\in F}x_f^{B_f-1}\right)^d$. Therefore,
\[ \left( \left( \prod_{f\in F}x_f^{B_f-1}\right) \cdot \lpartial'(x)\right)^m = \left(\prod_{f\in F}x_f^{B_f-1}\right)^d \left(\prod_{f\in F} x_f^{-\partial_{fe}}\right)^m = x^{d(B_{F,1}-1_{F,1})-\partial_{F,E}\cdot m}\]
By Equation \eqref{eq: Mpullback}, this is equal to the value of $z^m$ on $\lpartial(x)$. Since this holds for all $\GG_m^{V-1}$-invariant characters, the image of this point in $\GG_m^E/\GG_m^{V-1}$ equals $\lpartial(x)$.

\end{proof}

\begin{remark} \label{rem: simpleboundary}
As a consequence, if we are only interested in the induced map\footnote{We continue to abuse notation by using the same notation for a map and its quotient by the scaling action.}
\[ \lpartial: \GG_m^F/\GG_m \longrightarrow \GG_m^E/\GG_m^V\]
obtained after quotienting by the action of $\GG_m$, then we may use the formula for $\lpartial'$ instead.
\end{remark}

We can use Proposition~\ref{prop: partial} to analyze a commonly used family of coordinates on $\GG_m^E/\GG_m^V$.

\begin{cor}\label{cor: monodromy}
Let $f$ be an interior face of $G$ with boundary edges $e_1$, $e_2$, \dots, $e_{2r}$. Let $f_i$ be the face on the other side of $e_i$ from $f$. Define $\alpha_f: \GG_m^E \to \GG_m$ by $\alpha_f(z) = \prod_{i=1}^{2r}  z(e_i)^{(-1)^i}$. Then 
\[\alpha_f(z) = \prod_{i=1}^{2r}   \left( z^{\rM(f_i)} \right)^{(-1)^{i+1}}\]
\end{cor}
\noindent One can write a similar formula when $f$ is a boundary face, with cases depending on how many black and how many white vertices border $f$.

The $\alpha_f$ are clearly $\GG_m^V$-invariant, and have often been used as $\GG_m^V$-invariant coordinates on $\GG_m^E$, and as rational coordinates on $\Pio(\cM)$ induced by $\DD$ \cite{Pos,AHBCGPT12,GK13}. 

\begin{remark}
Once we know that $\tDD$ is an inclusion (Proposition \ref{inversion}), the monodromy coordinates can be combined into a single rational function $\alpha:\tPio(\cM)\dashedrightarrow \GG_m^F$. 
Assuming the cluster structure on $\tPio(\cM)$ described in \cite{MS14}, there is a rational \emph{cluster ensemble} map $\chi:\tPio(\cM)\dashedrightarrow \mathcal{X}$, where $\mathcal{X}$ is the associated $\mathcal{X}$-cluster variety \cite{FG09}. Corollary \ref{cor: monodromy} may be reformulated to say that $\alpha$ is the pullback along $\chi\circ \rt$ of the cluster $\mathcal{X}\dashedrightarrow \mathbb{G}_m^F$ associated to the reduced graph $G$.\footnote{A subtle detail: the presence of `frozen variables' allows some choice in the cluster ensemble map $\chi$. The specific $\chi$ which relates monodromy coordinates to $\mathcal{X}$-cluster variables has $1$-dimensional fibers and $1$-codimensional image.}
\end{remark}

\subsection{Uniqueness of matchings for boundary faces}

We are now ready to show that the map $\tDD$ lands in $\Pio(\cM)$. We need one more combinatorial lemma.

\begin{prop} \label{prop: boundaryunique}
For a boundary face $f$, $\vecM(f)$ is the unique matching of $G$ with boundary $\sI(f)$.
\end{prop}
\begin{proof}
Suppose that $M$ is another matching with boundary $\sI(f)$.  The set of edges in one of $M$ and $\vecM(f)$ but not both is a disjoint union of closed cycles of even length; let $\gamma$ be one such closed cycle.  Let $2 \ell$ be the length of $\gamma$ and let $H$ be the graph surrounded by $\gamma$. Then the restriction of $\vecM(f)$ to $H$ gives a matching of $H$; call this matching $M'$. Note that $M' \cap \gamma$ consists of $\ell$ edges. 

Since $M'$ is a matching of $H$, we have
\[ \# M' = \frac{1}{2} \# \mathrm{Vertices}(H) . \]
Since $H$ is a disc, we have
\[ \# \mathrm{Vertices}(H) - \# \mathrm{Edges}(H) + \# \mathrm{Faces}(H) = 1  \]
and thus
\begin{equation}
\# M' = \frac{1}{2} \left( \# \mathrm{Edges}(H) - \# \mathrm{Faces}(H) + 1 \right) \label{oneHand}
\end{equation}

Since every face $f'$ of $H$ is an interior face of $G$, the boundary of each such $f'$ contains $\#\partial f'/2 - 1$ edges of $M'$ by Theorem \ref{thm: minmatchprop}.  Each edge of $M'$ is counted twice in this way except for the $\ell$ edges along $\gamma$, so we have
\begin{multline*}  2 \# M' - \ell = \sum_{f' \in \mathrm{Faces}(H)} \#\{ e \in M' \cap f' \} = \sum_{f' \in Faces(H)}  \left( \frac{\# \partial f'}{2} -1 \right)  = \\
\frac{1}{2} \left(\sum_{f' \in Faces(H)} \# \partial f'\right) - \# \mathrm{Faces}(H) = \frac{1}{2} (2 \# \mathrm{Edges}(H) - 2 \ell) - \# \mathrm{Faces}(H)  =  \\  \# \mathrm{Edges}(H) - \ell - \# \mathrm{Faces}(H)
\end{multline*}
and thus 
\begin{equation}
\# M' = \frac{1}{2} \left( \# \mathrm{Edges}(H) - \# \mathrm{Faces}(H)  \right) \label{otherHand}
\end{equation}
Equations~\eqref{oneHand} and~\eqref{otherHand} are obviously in conflict, and we have reached a contradiction.
\end{proof}

This has a geometric consequence.  When $f$ is a boundary face, the partition function $D_{\sI(f)}$ is a monomial, not just a polynomial, and so it takes non-zero values on all of $\mathbb{G}_m^E/\mathbb{G}_m^{V-1}$.  

\begin{prop} \label{in open cell}
Let $G$ be a reduced graph.
The boundary measurement map $\DD:\mathbb{G}_m^E/\mathbb{G}_m^{V-1}\rightarrow \tPi(\cM)$ lands inside the open positroid variety $\tPio(\cM)$.
\end{prop}

\begin{proof}
We already know that the boundary measurement map lands in the closed variety $\tPi(\cM)$.

By Proposition \ref{prop: boundarylabel}, the labels $\sI(f)$ of the boundary faces are the elements of the reverse Grassman necklace $(I_1, I_2, \ldots, I_n)$. 
The open positroid variety $\tPio(\cM)$ is the intersection of the cones on the permuted open Schubert cells for the $I_a$ (\cite[Theorem 5.1]{KLS13}).
The nonvanishing of $p_{I_a}$ is exactly what picks out the open Schubert cell for $I_a$ from the closed Schubert variety.
\end{proof}

\subsection{Upstream wedges and associated matchings}

Each of the constructions and definitions in this section has an analog, when `downstream' is replaced by `upstream'.  This is equivalent to the effect of reversing the orientations of the strands, taking the mirror image of the graph and relabeling boundary vertex $j$ as $n-j$ (with indices cyclic modulo $n$).

In this way, we associate matchings $\cevM(f)$ to each face $f$, which have boundary $\tI(f)$, and we use these matchings to construct a pair of inverse isomorphisms $\lM$ and $\rpartial$. All the analogous results go through \emph{mutatis mutandis}. We highlight the fact that the maps $\lpartial$ and $\rpartial$ are very close to each other, in that they only differ by values at boundary faces.

\begin{prop}\label{prop: doubletwist1}
For $x\in \mathbb{G}_m^F$ and for any face $f$ in $G$,
\[ (\rM\circ \rpartial(x))_f = x_f \prod_{i\in \sI(f)}\frac{x_{i_-}}{x_{i_+}}
\]
where $i_+$ and $i_-$ denote the boundary face clockwise and counterclockwise (respectively) from the edge adjacent to vertex $i$.
\end{prop}
\noindent Consequently, if $x\in \GG_m^F$ has value $1$ at every boundary face, then $\lpartial(x)=\rpartial(x)$.
\begin{proof}
%
By adding boundary-adjacent vertices as needed, we may assume that the only internal vertices adjacent to the boundary are white. As a consequence, each edge adjacent to vertex $i$ has $i_+$ downstream and $i_-$ upstream. 
%
It follows that the upstream analog $\rpartial'(x)$ of $\lpartial'(x)\in \GG_m^E$ satisfies
\[ \rpartial'(x)^{-\vecM(f)} = \lpartial'(x)^{-\vecM(f)} \prod_{i\in \sI(f)}\frac{x_{i_-}}{x_{i_+}}\]
Another consequence of our simplifying assumption is that $B_{f'}$ is always half the number of boundary edges in $f'$, and so it coincides with its upstream analog. Consequently, the gauge transformation in Proposition \ref{prop: partial} is the same in its upstream analog. We may then compute.
\begin{align*}
(\rM\circ \rpartial(x))_f &=\rpartial(x)^{-\vecM(f)}=\left(x^{B_{F,1}-1_{F,1}}\right)\rpartial'(x)^{-\vecM(f)}  \\
&=\left(x^{B_{F,1}-1_{F,1}}\right)\lpartial'(x)^{-\vecM(f)} \prod_{i\in \sI(f)}\frac{x_{i_-}}{x_{i_+}}= \lpartial(x)^{-\vecM(f)} \prod_{i\in \sI(f)}\frac{x_{i_-}}{x_{i_+}}
= x_f \prod_{i\in \sI(f)}\frac{x_{i_-}}{x_{i_+}} \qedhere
\end{align*}
\end{proof}


%
%
%

\section{The twist for positroid varieties} \label{sec twist}

This section defines the left and right twist of a $k\times n$-matrix of rank $k$ and collects its basic properties.  These operations on matrices descend to inverse automorphisms of each open positroid variety $\tPi^\circ(\cM)$, which will be used to relate the boundary measurement map of a reduced graph $G$ to the Pl\"ucker coordinates associated to faces.


For a $k\times n$ matrix $A$ and $a\in [n]$, define
\[ A_a := \text{the $a$th column of $A$} \]
We extend this notation to any $a\in \ZZ$ to be periodic modulo $n$.
For any $k$-element set $I\subset \ZZ$ we define
\[ \Delta_I(A) := \det(A_{i_1},A_{i_2},...,A_{i_k})\]
where $I=\{i_1<i_2<\cdots < i_k\} $. 


\subsection{Definition of the twists}

Let $\langle - \mid - \rangle$ denote the standard Euclidean inner product on $\CC^k$.

Recall that $\Mato(k,n)$ is the set of $k\times n$ complex matrices with rank $k$.
Given $A\in \Mato(k,n)$, define the \newword{right twist} of $A$ to be the $k\times n$ matrix $\rt(A)$ whose column $\rt(A)_a$ is the unique vector such that,
\[ \textrm{For}\ b \in \vecI_a,\ \textrm{we have}\ \langle \rt(A)_a, A_a \rangle = 
\begin{cases}
1 & a=b \\ 
0 & a \neq b \\
\end{cases} \]
Since $\vecI_a$ is a basis of $\CC^k$, this describes a unique vector.
Note that, if $A_a=0$ then $a \not \in \vecI_a$ and thus $\rt(A)_a$ is required to be perpendicular to a basis of $\CC^k$; we deduce that, if $A_a=0$ then $\rt(A)_a=0$.

We similarly define the \newword{left twist} $\lt(A)$ using the left Grassman necklace:
\[ \textrm{For}\ b \in \cevI_a,\ \textrm{we have}\ \langle \lt(A)_a, A_a \rangle = 
\begin{cases}
1 & a=b \\ 
0 & a \neq b \\
\end{cases} \]


Unwinding the definition of the Grassman necklace, we can restate these definitions. Assuming for simplicity that none of the $A_a$ are $0$, we have
\[ \langle \rt(A)_a, A_a \rangle = \langle \lt(A)_a, A_a \rangle = 1 \]
\[ \langle \rt(A)_a, A_b \rangle =0 \ \text{whenever} \ A_b \not \in \mathrm{span}(A_{a+1}, A_{a+2}, \ldots, A_{b-1}) \ \text{for}\ a<b \leq a+n\]
\[ \langle \lt(A)_a, A_b \rangle =0 \ \text{whenever} \ A_b \not \in \mathrm{span}(A_{b+1}, A_{b+2}, \ldots, A_{a-1}) \ \text{for}\ b<a  \leq b+n \]



The Grassman necklace and reverse Grassman necklace of $A$ are constant on the set $\Mato(\cM)$ consisting of matrices with the same positroid $\cM$ as $A$ (Proposition \ref{prop: equivposi}).  As a consequence, $\rt$ and $\lt$ are algebraic maps when restricted to $\Mato(\cM)$.

The torus $\mathbb{G}_m^n$ has a right action on $\Mato(k,n)$ by scaling the columns.

\begin{prop}\label{prop: equivariance}
For any $A\in \Mato(k,n)$, $\alpha\in GL_k$ and $\beta\in \mathbb{G}_m^n$, 
\[ \rt(\alpha A\beta) = (\alpha^{-1})^{\top}\rt(A)\beta^{-1} \]
\end{prop}
\begin{proof}
For an index $c$, let $\beta_c$ denote the $c$-th coordinate of $\beta$.  For any $a\in [n]$ and any $b\in \vecI_a$.
\begin{eqnarray*}
 \langle  ((\alpha^{-1})^{\top}\rt(A)\beta^{-1})_a\mid (\alpha A\beta)_{b}  \rangle &=&  \langle  \beta_a^{-1}\rt(A)_a \mid \beta_bA_{b}\rangle \\
  &=& \frac{\beta_b}{\beta_a}\langle  \rt(M)_a \mid M_{b} \rangle = \left\{\begin{array}{cc}
1 &  a=b \\
0 & \text{otherwise} \end{array}\right\}
\end{eqnarray*}
By the construction of the right twist, $\rt(\alpha A\beta) = (\alpha^{-1})^{\top}\rt(A)\beta^{-1}$.  
\end{proof}

Quotienting by $SL_k$, we see that the twists are algebraic maps $\tPio(\cM) \to \widetilde{Gr(k,n)}$. We will make a more precise statement in Corollary~\ref{twist quotient}.

\begin{remark} \label{remark homog}
The result on the $\GG_m^n$ action says that any formula for $\rt(A)_a$ or $\lt(A)_a$ must be homogenous of degree $-1$ in $A_a$, and homogenous of degree $0$ in the other columns $A$.
\end{remark}

\begin{remark} \label{rem: MStwist}
The twist of Marsh and Scott \cite{MS16} (which is defined for matrices with uniform positroid envelope and denoted $\overrightarrow{M}$) is related to our twist by rescaling the columns; specifically, for each $a$, $\overrightarrow{M}_a= \Delta_{I_a}(M) \rt(M)_a$.  We consider the simple homogeneity statement of Remark~\ref{remark homog} to be evidence that our choice of normalization is cleaner than theirs.
\end{remark}

\subsection{Twist identities} 

We prove a pair of identities relating a matrix and its right twist.
\begin{lemma} \label{lem more tau vanishing}
Let $\pi$ be the bounded affine permutation of $A$.
If $a< b< \pi(a)$, then
\[ \langle \rt(A)_a \mid A_b\rangle =\langle \rt(A)_b \mid A_{\pi(a)} \rangle = 0\]
\end{lemma}
\begin{proof}
Since $a < b < \pi(a)$, we have that $A_a$ is not $0$ and $A_{a+1}$ is not parallel to $A_a$.

Define $\vec{L}_a = \mathrm{span}(A_a, A_{a+1},\ldots, A_{\pi(a)-1})$, and $\vec{L}'_a = \mathrm{span}(A_{a+1},\ldots, A_{\pi(a)-1})$.
Using Lemma~\ref{lem Knutson pi rule}, $A_a$ is not in $\vec{L}'_a$, so $\vec{L}_a = \vec{L}'_a \oplus \mathrm{span}(A_a)$. 
Lemma~\ref{lem imp Grass necklace} tells us that $\{ A_c : c \in [a, \pi(a)) \cap \vecI_a \}$ is a basis of $\vec{L}_a$  so $\{ A_c : c \in (a, \pi(a)) \cap \vecI_a \}$ is a basis of $\vec{L}'_a$. Since $\rt(A)_a$ is orthogonal to $\{ A_c : c \in (a, \pi(a)) \cap \vecI_a \}$, we conclude that $\rt(A)_a$ is orthogonal to $\vec{L}'_a$. The vector $A_b$ lies in $\vec{L}'_a$, so $\langle \rt(A)_a, A_b \rangle = 0$. 

We now prove $\langle \rt(A)_b \mid A_{\pi(a)} \rangle = 0$.
If $b < \pi(a) < \pi(b)$, then this follows from the first paragraph.
If $b \leq \pi(b) < \pi(a)$, then $b \impto_{\pi} a$ so, by Lemma~\ref{lem imp Grass necklace}, $\pi(a) \in I_b \setminus b$ and we have $\langle \rt(A)_b \mid A_{\pi(a)} \rangle = 0$.
\end{proof}

\begin{lemma}\label{lemma: matrixformula}
For any $A\in \Mato(\cM)$, and any $I=\{i_1<i_2<...<i_k\}, J=\{j_1<j_2<...<j_k\}\subset \ZZ$,

\[ \Delta_I(\rt(A)) \Delta_J(A) = \det\gmat{
\langle \rt(A)_{i_1} \mid A_{j_1} \rangle & \langle \rt(A)_{i_1} \mid A_{j_2} \rangle & \cdots & \langle \rt(A)_{i_1} \mid A_{j_k} \rangle \\
\langle \rt(A)_{i_2} \mid A_{j_1} \rangle & \langle \rt(A)_{i_2} \mid A_{j_2} \rangle & \cdots & \langle \rt(A)_{i_2} \mid A_{j_k} \rangle \\
\vdots & \vdots & \ddots & \vdots \\
\langle \rt(A)_{i_k} \mid A_{j_1} \rangle & \langle \rt(A)_{i_k} \mid A_{j_2} \rangle & \cdots & \langle \rt(A)_{i_k} \mid A_{j_k} \rangle 
} \]
\end{lemma}
\begin{proof}
Consider the $n\times n$ matrix $\rt(A)^\top \cdot A$, and its restriction to the rows in $I$ and the columns in $J$.  The lemma is equivalent to the statement that the product of the determinants is equal to the determinant of the product.
\end{proof}

\subsection{Twists as inverse automorphisms}  

As previously observed, the twists are algebraic when restricted to matrices in $\Mato(\cM)$ for some positroid $\cM$.  The next proposition asserts that the twists are actually algebraic endomorphisms of this subvariety; that is, $\rt(A)$ and $\lt(A)$ have the same positroid envelope as $A$.

\begin{prop}\label{prop: positwist}
For any $A\in \Mato(\cM)$, the twists $\rt(A)$ and $\lt(A)$ are in $\Mato(\cM)$.
\end{prop}
\begin{proof}
Let $\veccI=\{\vecI_1,\vecI_2,...,\vecI_n\}$ be the Grassman necklace of $\cM$ and let $a\in [n]$.
The matrix that appears in Lemma \ref{lemma: matrixformula} for $I=J=\vecI_a$ is lower triangular with ones on the diagonal; hence, 
\begin{equation}\label{eq: twistboundary}
\Delta_{\vecI_a}(\rt(A)) =\frac{1}{\Delta_{\vecI_a}(A)}
\end{equation}
In particular, $\Delta_{\vecI_a}(\rt(A))$ is non-zero.

Now, let $J$ be a $k$-element subset of $[n]$, and suppose $J \not \in \cM$. We will show that $\Delta_J(\rt(A))=0$.

The hypothesis that $J \not \in \cM$ means that there is some $a$ for which $\vecI_a \not\preceq_a J$.
In other words, writing
\[ \vecI_a=\{ i_1\prec_a i_2 \prec_a ... \prec_a i_k\},\;\;\; J = \{ j_1 \prec_a j_2 \prec_a ... \prec_a j_k \}, \]
 there is some $b\in [k]$ such that $j_b \prec_a i_b$. 

By Lemma \ref{lemma: matrixformula}, the Pl\"ucker coordinate $\Delta_{J}(\rt(A)) $ is
\begin{equation}\label{eq: twistformulaA}
\frac{1}{\Delta_{\vecI_a}(A)}\det\gmat{
\langle \rt(A)_{j_1} \mid A_{i_1} \rangle &  \langle \rt(A)_{j_1} \mid A_{i_2} \rangle & \cdots &  \langle \rt(A)_{j_1} \mid A_{i_k} \rangle \\
\langle \rt(A)_{j_2} \mid A_{i_1} \rangle &  \langle \rt(A)_{j_2} \mid A_{i_2} \rangle & \cdots & 
\langle \rt(A)_{j_2} \mid A_{i_k} \rangle \\
\vdots & \vdots & \ddots & \vdots \\
\langle \rt(A)_{j_k} \mid A_{i_1} \rangle & \langle \rt(A)_{j_k} \mid A_{i_2} \rangle & \cdots & \langle \rt(A)_{j_k} \mid A_{i_k} \rangle}
\end{equation}
We claim that the top right $b\times (k-b+1)$ submatrix of \eqref{eq: twistformulaA} is zero; that is, for any $c,d\in [k]$ with $c\leq b\leq d$, we have
\[ \langle \rt(A)_{j_c} \mid A_{i_d}\rangle =0 \]
To see this, first observe that $A_{i_d}$ is not in the span of $\{A_a,....,A_{i_d-1}\}$ by the definition of $\vecI_a$.  We also have $j_c\preceq_a j_b \prec_a i_b \preceq_a i_d$, and so $A_{j_c}$ appears in the list $\{A_a,...,A_{i_d-1}\}$.  Therefore, $A_{i_d}$ is not in the span of $\{A_{j_c},...,A_{i_{d-1}}\}$, so $i_d\in \vecI_{j_c}$.  Then $\langle \rt(A)_{j_c} \mid A_{i_d} \rangle=0$ by the definition of the right twist.
Hence, the top right $b\times(k-b+1)$ submatrix of \eqref{eq: twistformulaA} is zero, and so $\Delta_J(\rt(A))=0$.  

By \cite{Oh11}, a $k\times n$ matrix $B$ is in $\Mato(\cM)$ if and only if
\[ \forall J \not\in \cM,\;\;\;   \Delta_J(B) =0\]
\[ \forall a\in [n],\;\;\; \Delta_{\vecI_a}(B) \neq0\]
Hence, we have checked that $\rt(A)\in \Mato(\cM)$.  The analogous result for $\lt(A)$ holds by a symmetric argument.
\end{proof}

We may improve this as follows.

\begin{thm}\label{thm: inversetwists}
For any positroid $\cM$, the twists $\rt$ and $\lt$ define inverse automorphisms of $\Mato(\cM)$.
\end{thm}

\begin{proof}
Let $\veccI=\{\vecI_1,\vecI_2,...,\vecI_n\}$ and $\cevcI=\{\cevI_1,\cevI_2,...,\cevI_n\}$ denote the Grassman necklace and reverse Grassman necklace of $\cM$, respectively.
Choose any $a\in [n]$ and any $b\in \vecI_a$.  

If $A\in \Mato(\cM)$, then $A_b \not \in \text{span}\{A_a,A_{a+1},...,A_{b-1}\}$, so
\[\dim(\text{span}\{A_a,A_{a+1},...,A_{b}\}) = \dim(\text{span}\{A_a,A_{a+1},...,A_{b-1}\})+1.\]
Let
$\dim(\text{span}\{A_a,A_{a+1},...,A_{b}\})=c$.
Hence, $c$ elements of $\cevI_{b}$ lie in $\{a,a+1,...,b \}$, and so 
\[ J:= \left(\cevI_b-\{b\}\right) \cap \{a,a+1,...,b-1\} \]
has $c-1$ elements.
The set $\{A_j\}_{j\in J}$ is part of a basis for $\CC^k$, so it is linearly independent, and hence it is a basis for the $(c-1)$ dimensional space 
\[ \text{span}\{A_a,A_{a-1},...,A_{b-1}\} \]
In particular, as long as $a\neq b$, $A_a$ is a linear combination of $\{A_j\}_{j\in J}$.  
By the construction of the left twist, $\langle \lt(A)_b \mid A_j\rangle=0$ whenever $j\in J$.  Hence, as long as $a\neq b$, $\langle \lt(A)_b \mid A_a \rangle =0$.  If $a=b$, then $\langle \lt(A)_b \mid A_a\rangle =1$.

Since $a$ and $b$ were arbitrary, the matrix $A$ satisfies all the identities which define $\rt(\lt(A))$.  Hence, $\rt(\lt(A))=A$. The argument that $\lt(\rt(A))=A$ is identical.
\end{proof}


\begin{cor} \label{twist quotient}
For any positroid $\cM$, the twists $\rt$ and $\lt$ descend to mutually inverse automorphisms of $\tPi^\circ(\cM)$ and $\Pi^\circ(\cM)$.
\end{cor}

\begin{proof}
Using Proposition~\ref{prop: equivariance}, this is the quotient of Theorem~\ref{thm: inversetwists} by $SL_k$ (in the case of $\tPio$) and $GL_k$ (in the case of $\Pio$).
\end{proof}

We conclude the section with a refinement of Lemma \ref{lem more tau vanishing} we will need later.


\begin{lemma} \label{lem even more tau vanishing}
For all $A$ and $a$, the set $\{ \rt(A)_b : a\impto_\pi b \}$ is a basis for $\mathrm{span}(A_{a},\ldots, A_{\pi(a)-1})^\perp$.
\end{lemma}
\noindent Here, $V^\perp$ denotes the orthogonal complement to $V$.


\begin{proof}
Let $\vec{L}_a$ denote $\mathrm{span}(A_{a},\ldots, A_{\pi(a)-1})$.  Let $c \in [a, \pi(a))$ and choose $b$ such that $a\impto_\pi b$. Then $b < a \leq c < \pi(a) < \pi(b)$ so, by Lemma~\ref{lem more tau vanishing}, $\langle \rt(A)_b, A_c \rangle=0$.
So, each $\rt(A)_b$ is orthogonal to $\vec{L}_a$.

By Proposition \ref{prop: positwist}, $\rt(A)$ has the same positroid envelope as $A$. In particular, $\{\rt(A)_b \mid b\in \cevI(a)\}$ must be a basis for $\CC^k$. By Lemma \ref{lem imp Grass necklace}.b applied to $\rt(A)$, $\{\rt(A)_b\mid a\impto_\pi b\}$ is a subset of this basis, and so is linearly independent. By Lemma \ref{lem imp Grass necklace}.a applied to $A$, 
\[ | \{b \mid a \impto_\pi \pi^{-1}(b)\}| = k - \dim(\vec{L}_a) \]
Since $|\{\rt(A)_b\mid a\impto_\pi b\}| = | \{b \mid a \impto_\pi b\}| = | \{b \mid a \impto_\pi \pi^{-1}(b)\}|$,
the cardinality of $\{\rt(A)_b \mid a \impto_\pi b\}$ is equal to the dimension of $\vec{L}_a^\perp$. Therefore, it is a basis for $\vec{L}_a^\perp$.
%
%
%
%
\end{proof}

\section{The main theorem}\label{sec main theorem}

We restate our main theorem, which is the commutativity of a diagram built out of the maps constructed in the last five sections. 

We reuse and reiterate much of the notation from the previous sections.
Let $G$ be a reduced graph with positroid $\cM$.
Let $\tPio(\cM)$ denote the open positroid variety of $\cM$ (Section~\ref{pos vars}).
Let $V$, $E$ and $F$ denote the vertices, edges and faces of $G$, and we define the tori $\GG_m^E/\GG_m^{V-1}$ and $\GG_m^F$ as in Section~\ref{sec gauge}.
We have the isomorphisms $\lM$, $\rM$, $\lpartial$ and $\rpartial$ between these tori from Section~\ref{sec character basis}.
Let $\tDD$ be the boundary measurement map (Section~\ref{sec matchings}); from Proposition~\ref{in open cell}, we can view $\tDD$ as a map to $\tPio(\cM)$.
 Let $\sF$ and $\tF$ be the source-labeled and target-labeled face Pl\"ucker maps (Section~\ref{sec face labels}).
Finally, let $\lt$ and $\rt$ be the left and right twists (Section~\ref{sec twist}).

\begin{Theorem}\label{main theorem}
The following diagram commutes, where dashed arrows denote rational maps:
\[\begin{tikzpicture}
	\node (F1) at (-3,0) {$\mathbb{G}_m^F$};
	\node (E) at (0,0) {$\mathbb{G}_m^E/\mathbb{G}_m^{V-1}$};
	\node (F2) at (3,0) {$\mathbb{G}_m^F$};
	\node (P1) at (-3,-2) {$\tPio(\cM)$};
	\node (P2) at (0,-2) {$\tPio(\cM)$};
	\node (P3) at (3,-2) {$\tPio(\cM)$};
	\draw[-angle 90,out=15,in=170] (F1) to node[above] {$\rpartial$} (E);
	\draw[-angle 90,out=10,in=165] (E) to node[above] {$\rM$} (F2);
	\draw[-angle 90,out=195,in=-10] (F2) to node[below] {$\lpartial$} (E);
	\draw[-angle 90,out=190,in=-15] (E) to node[below] {$\lM$} (F1);
	\draw[dashed,-angle 90] (P1) to node[left] {$\tF$} (F1);
	\draw[-angle 90] (E) to node[left] {$\tDD$} (P2);
	\draw[dashed,-angle 90] (P3) to node[right] {$\sF$} (F2);
	\draw[-angle 90,out=15,in=165] (P1) to node[above] {$\rt$} (P2);
	\draw[-angle 90,out=15,in=165] (P2) to node[above] {$\rt$} (P3);
	\draw[-angle 90,out=195,in=-15] (P3) to node[below] {$\lt$} (P2);
	\draw[-angle 90,out=195,in=-15] (P2) to node[below] {$\lt$} (P1);

\end{tikzpicture}\]
More precisely, the diagram commutes in the category of rational maps and any composition of maps starting in the top row is regular.
\end{Theorem}

\noindent We will defer the details of the proof until Section \ref{sec main thm proof}, and instead spend the rest of this section exploring the theorem. We begin with several remarks on the diagram itself.

\begin{remark}
Each of the spaces in the diagram has a natural $\GG_m$-action which commutes or anti-commutes with each of the maps, and so there is a quotient commutative diagram.
\[\begin{tikzpicture}
	\node (F1) at (-3,0) {$\mathbb{G}_m^F/\mathbb{G}_m$};
	\node (E) at (0,0) {$\mathbb{G}_m^E/\mathbb{G}_m^{V}$};
	\node (F2) at (3,0) {$\mathbb{G}_m^F/\mathbb{G}_m$};
	\node (P1) at (-3,-2) {$\Pio(\cM)$};
	\node (P2) at (0,-2) {$\Pio(\cM)$};
	\node (P3) at (3,-2) {$\Pio(\cM)$};
	\draw[-angle 90,out=10,in=170] (F1) to node[above] {$\rpartial$} (E);
	\draw[-angle 90,out=10,in=170] (E) to node[above] {$\rM$} (F2);
	\draw[-angle 90,out=190,in=-10] (F2) to node[below] {$\lpartial$} (E);
	\draw[-angle 90,out=190,in=-10] (E) to node[below] {$\lM$} (F1);
	\draw[dashed,-angle 90] (P1) to node[left] {$\tF$} (F1);
	\draw[-angle 90] (E) to node[left] {$\DD$} (P2);
	\draw[dashed,-angle 90] (P3) to node[right] {$\sF$} (F2);
	\draw[-angle 90,out=15,in=165] (P1) to node[above] {$\rt$} (P2);
	\draw[-angle 90,out=15,in=165] (P2) to node[above] {$\rt$} (P3);
	\draw[-angle 90,out=195,in=-15] (P3) to node[below] {$\lt$} (P2);
	\draw[-angle 90,out=195,in=-15] (P2) to node[below] {$\lt$} (P1);
\end{tikzpicture}\]
The bottom row of the diagram now takes place in the Grassmannian itself, and the maps $\lpartial$ and $\rpartial$ have a much simpler form (Proposition \ref{prop: partial} and Remark \ref{rem: simpleboundary}).
\end{remark}

\begin{remark}
The top row depends on the graph $G$, but the bottom row does not. If $G$ and $G'$ are two reduced graphs related by a move (Section~\ref{sec transformations}), then the birational map $\mu$ defined in that section gives a birational isomorphism between the center elements of the corresponding top rows, which commutes with the other maps in the diagrams.
\end{remark}

\begin{remark}
The right action of $\GG_m^n$ described in Proposition~\ref{prop: equivariance} can be extended to actions on the tori in the top row, coming from monomial maps from $\GG_m^n$ to $\GG_m^E$ and $\GG_m^F$. The vertical maps in the diagram commute with this $\GG_m^n$ action, and the horizontal maps anti-commute.
\end{remark}

\begin{remark}
We collect our justifications and mnemonics for our notation. 
Maps with rightward arrows always travel to the right in the diagram, or (in the case of the vertical map $\sF$) are in the right-hand edge.
The twists $\rt(A)_a$ and $\lt(A)_a$ depend on the columns of $A$ to the right and left of $A_a$, respectively.
In $\rM(f)$ and $\lM(f)$, the direction of the arrow recalls whether the strands points towards or away from face $f$. 
The maps $\lpartial$ and $\rpartial$ are inverse to $\rM$ and $\lM$. 
Finally, the bullet in the notation for $\sF$ and $\tF$ indicates whether we are using source or target labeled strands.
\end{remark}

\subsection{Inverting the boundary measurement map} \label{sec inversion}

Theorem~\ref{main theorem} implies that the boundary measurement maps are inclusions.
\begin{prop} \label{inversion}
For a reduced graph, the maps $\tDD : \GG_m^E/\GG_m^{V-1} \longrightarrow \tPio(\cM)$ and $\DD: \GG_m^E/\GG_m^V\longrightarrow \Pio(\cM)$ are open immersions. 
\end{prop}

\begin{proof}
The inverse rational map is $\lpartial \circ \sF \circ \rt$. Since this rational map is defined on the image of $\tDD$, the map $\tDD$ is an open immersion. The result for $\DD$ is identical.\end{proof}

We describe the inverse map $\lpartial \circ \sF \circ \rt$ in words: Given a point in the positroid variety, twist it, compute the Pl\"ucker coordinates given by the face labels, and then weight an edge by the reciprocal of the product of the adjacent faces, with the correction involving the gauge action described in Proposition~\ref{prop: partial}. If we only want an inverse map from $\Pio(\cM)$ to $\GG_m^E/\GG_m^V$, then the gauge correction can be omitted. 

\begin{remark}
This generalizes the main result of Talaska~\cite{Tal11}, who proves invertibility of the boundary map for Le-diagrams.
Talaska's description of the inverse does not involve the twist, but expresses the coordinates of $\GG_m^E/\GG_m^{V-1}$ directly as ratios of Pl\"ucker variables.
\end{remark}

\subsection{The Laurent phenomenon}

From Theorem~\ref{main theorem}, we see that the domain of definition of $\sF$ is the image of $\rt \circ \tDD$. 
Since $\tDD$ is injective (Proposition~\ref{inversion}) and $\rt$ is an isomorphism, this shows that the domain of definition of $\sF$ is a torus.
So any function on $\tPio(\cM)$ will restrict to a regular function on this torus, and hence to a Laurent polynomial in a basis of characters of this torus.
Theorem~\ref{main theorem} says that the Pl\"ucker coordinates $\Delta_{\sI(f)}$ are such a character basis for this torus, which proves the following.

\begin{prop} \label{Laurent}
A function in the coordinate ring of $\tPio(\cM)$ may be written as a Laurent polynomial in the functions $\{ \Delta_{\sI(f)} \}_{f \in F}$.
\end{prop}
\noindent
By a similar argument, such a function may also be written as a Laurent polynomial in $\{ \Delta_{\tI(f)} \}_{f \in F}$.

\begin{remark}
This verifies part of the `Laurent phenomenon' that would follow from the conjectural cluster structure on $\tPio(\cM)$; specifically, the Laurent phenomenon for those clusters represented by a reduced graph.
\end{remark}

\subsection{Laurent formulas for twisted Pl\"uckers}

The main theorem also provides explicit Laurent polynomials for certain functions on $\tPio(\cM)$; specifically, the twisted Pl\"ucker coordinates.

%

\begin{prop} \label{dimer sum is twist}
For any $J\in \binom{[n]}{k}$, we have
\[ \Delta_J \circ \lt =  \sum_{\substack{ \text{matchings }M \\ \text{with }\partial M = J}}  \prod_{f \in F} \Delta_{ \sI(f) }^{(B_f-1)-\#\{ e \in M \ : \ \partial_{fe}=1 \}} \]
\end{prop}
\noindent That is, each Laurent polynomial expressing a twisted Pl\"ucker coordinate in terms of the $\{\Delta_{\sI(f)}\}_{f\in F}$ is a partition function of matchings with fixed boundary.

\begin{proof}
We use the right-hand square in Theorem~\ref{main theorem}.  Applying $\Delta_J \circ \lt$ is equal to applying $\tDD \circ \lpartial \circ \sF$ and projecting on the $J$-th coordinate. 
The $J$-th coordinate of $\tDD$ is the partition function $\DD_J$, the sum of matching monomials over matchings with boundary $J$. Rewriting Corollary \ref{coro: facePartition}, for all $x\in \GG_m^F$, we have
\[\DD_J\left(\lpartial(x) \right) 
=  \sum_{\stackrel{\text{matchings }M }{\text{with } \partial M=J}}\prod_{f\in F} x_f^{(B_f-1) - \#\{ e \in M \ : \ \partial_{fe}=1 \}} \]
Precomposing both sides with $\sF$ completes the proof.
\end{proof}


\begin{remark}\label{rem: righttwist}
There is a similar formula for the Pl\"ucker coordinate of a right twist, using the left-hand square in Theorem \ref{main theorem}, which is even a sum over the same set of matchings.
\[ \Delta_J \circ \rt =  \sum_{\substack{ \text{matchings }M \\ \text{with }\partial M = J}}  \prod_{f \in F} \Delta_{ \tI(f) }^{(\widetilde{B}_f-1)-\#\{ e \in M \ : \ \widetilde{\partial}_{fe}=1 \}} \]
However, the reader is cautioned that $\widetilde{B}_f$ and $\widetilde{\partial}_{fe}$ here are the analogs of $B_f$ and $\partial_{fe}$ in which `downstream' has been replaced by `upstream'.
\end{remark}

\begin{remark}
Theorem \ref{main theorem} does not directly give a combinatorial description of the Laurent polynomials of the (untwisted) Pl\"ucker coordinates. 
\end{remark}

\subsection{The double twist} \label{double twist}

Theorem~\ref{main theorem} has interesting consequences for $\rt^2$, as we will now explain. 


\begin{prop} \label{prop: doubletwist}
Consider a positroid $\cM$ with permutation $\pi$ and Grassman necklace $\vecI_1,\vecI_2,...,\vecI_n$. Let $A \in \Mato(\cM)$. For any $I$ which occurs as the source-label of a face some reduced graph for $\cM$, we have
\[ \Delta_{I} (\rt^2(A)) = \Delta_{\pi(I)}(A) \prod_{i \in I} \frac{\Delta_{\vecI_{i}}(A)}{\Delta_{\vecI_{i+1}}(A)} \]
\end{prop}
\noindent An analogous result for $\Delta_I\circ \lt^2$ holds when $I$ is the target-label of a face in a reduced graph for $\cM$.


\begin{proof}
Fix a reduced graph $G$ with positroid $\cM$ and a face $f$ such that $I=\sI(f)$. Then
\begin{align*}
 \Delta_{\sI(f)}( \rt^2(A)) &= \left(\sF( \rt^2(A)) \right)_f \stackrel{\text{Thm. \ref{main theorem}}}{=} \left(\rM\circ \rpartial\circ \tF(A) \right)_f 
\stackrel{\text{Prop. \ref{prop: doubletwist1}}}{=} \left(\sF(A)\right)_f \prod_{i\in I} \frac{\left(\sF(A)\right)_{i_-}}{\left(\sF(A)\right)_{i_+}}\\
& = \Delta_{\tI(f)}(A) \prod_{i\in I} \frac{\Delta_{\tI(i_-)}(A)}{\Delta_{\tI(i_+)}(A)}  \stackrel{\text{Prop. \ref{prop: boundarylabel}}}{=} \Delta_{\tI(f)}(A) \prod_{i\in I} \frac{\Delta_{\vecI_{i}}(A)}{\Delta_{\vecI_{i+1}}(A)} 
\end{align*}
Since $\tI(f) = \pi(\sI(f))$, the result is proven. 
\end{proof}

We can give a geometric interpretation to Proposition~\ref{prop: doubletwist}. We define a map $\mu : \Mato(\cM) \to \Mato(k,n)$ as follows:
\[ \mu(A)_i = A_{\pi(i)} \frac{\Delta_{\vecI_{i}}(A)}{\Delta_{\vecI_{i+1}}(A)} (-1)^{\# \{j : i \impto_{\pi} j \} + (k-1) \delta( n \in [i, \pi(i)) ) } \]
here $\delta( n \in [i, \pi(i)) )$ is $1$ if $n \in [i, \pi(i))$ and $0$ otherwise.
It is easy to see that $\mu$ descends to a map $\Pio(\cM) \to Gr(k,n)$.

\begin{prop} \label{prop: geometricdoubletwist}
Let $A \in \Mato(\cM)$.  If $I$ is a source-label of a face for some reduced graph for $\cM$, then
\[ \Delta_I( \rt^2(A)) = \Delta_I ( \mu (A))\]
as functions $\Mato(\cM) \to \CC$ (or $\tPio(\cM) \to \CC$).
\end{prop}

\begin{proof}
We first check the result up to sign.
\[ \Delta_I(\mu(A)) = \det (\mu(A)_i)_{i \in I} = \pm \det \left(A_{\pi(i)} \frac{\Delta_{\vecI_{i}}(A)}{\Delta_{\vecI_{i+1}}(A)} \right)_{i \in I} =  \pm \Delta_{\pi(I)}(A) \prod_{i \in I}\frac{\Delta_{\vecI_{i}}(A)}{\Delta_{\vecI_{i+1}}(A)} =\pm  \Delta_I(\rt^2(A)) . \]

We now think about the signs. We emphasize that we consider $I$ specifically as a subset of $[n]$, and not some other lift modulo $n$.

There are two places where are signs are introduced. First, $\det(\mu(A)_i)_{i \in I}$ is ordered according to the linear order on $I$. When we reorder to the linear order on $\pi(I)$, we introduce the sign $(-1)^{\# \{ (j,i) \in I^2 : \ j<i,\ \pi(i) < \pi(j) \}} = (-1)^{\# \{ (j,i) \in I^2 : i \impto_{\pi} j \} }$. 
Since $I$ is a source-label, by Lemma~\ref{label implication}, if $i \in I$ and $i \impto_{\pi} j$ then $j \in I$. 
So the exponent can be rewritten as $\sum_{i \in I} \# \{ j \in [n] : i \impto_{\pi} j \}$. This is precisely the contribution from the $(-1)^{\# \{j : i \impto_{\pi} j \}}$ factor in the deifnition of $\mu$.

The second sign is introduced when we change from using the linear order on $\pi(I)$ to the linear order on $\pi(I)$ reduced modulo $n$ to lie in $[n]$. 
For each $i \in I$ obeying $i \leq n < \pi(i)$, this reordering introduces a sign of $(-1)^{k-1}$. This is the contribution from the $(-1)^{(k-1) \delta( n \in [i, \pi(i)) ) }$ factor.
\end{proof}

\begin{myexample}
Let 
\[ A = \begin{bmatrix}
p & q & 0 & -s \\
0 & 0 & r & t \\
\end{bmatrix} .\]
So 
\[ \begin{array}{rclrclrcl}
\Delta_{12}(A) &=& 0 &
\Delta_{13}(A) &=& pr &
\Delta_{14}(A) &=& pt \\
\Delta_{23}(A) &=& qr &
\Delta_{24}(A) &=& qt &
\Delta_{34}(A) &=& rs \\
\end{array} \]
The decorated permutation of $\pi$ is $\pi(1)=2$, $\pi(2)=4$, $\pi(3)=5$, $\pi(4) = 7$. The matroid $\cM$ is $\{ 13, 14, 23, 24, 34 \}$. 

Then
\[ \rt (A) = \begin{bmatrix}
p^{-1} & q^{-1} & \frac{t}{rs} & 0 \\
0 & 0 & r^{-1}  & t^{-1} \\
\end{bmatrix} \qquad 
 \rt^2 (A) = \begin{bmatrix}
p & q & \frac{rs}{t} &  0 \\
-\frac{pt}{s} & -\frac{qt}{s} & 0 & t \\
\end{bmatrix} .\]
Meanwhile,
\[ \mu(A) = \begin{bmatrix}
 \frac{p}{q} q & - \frac{q}{s} (-s)  & \frac{rs}{pt} p & 0 \\
0 & -\frac{q}{s} t  & 0 & \frac{t}{r} r  \\
\end{bmatrix} = \begin{bmatrix}
p & q & \frac{rs}{t}  & 0 \\
0 & -\frac{qt}{s}  & 0 & t  \\
\end{bmatrix}
. \]
There is a unique reduced graph for this permutation, with source-labels $14$, $23$, $24$, $34$. We see that $\Delta_I(\rt^2 (A)) = \Delta_I(\mu (A))$ for these $I$, but that $\Delta_I(\rt^2 (A))  \neq \Delta_I(\mu (A))$ for $I=12$ or $13$.
\end{myexample}


\section{Bridge decompositions} \label{sec bridge}

In order to prove Lemma~\ref{lemma: main}, we need one more tool known as \emph{bridge decompositions}. Bridge decompositions were introduced in~\cite{AHBCGPT12}; we will use~\cite{Lam16} as our reference for their properties. Essentially, adding bridges and adding lollipops are two ways to make a more complex reduced graph from a simpler one.

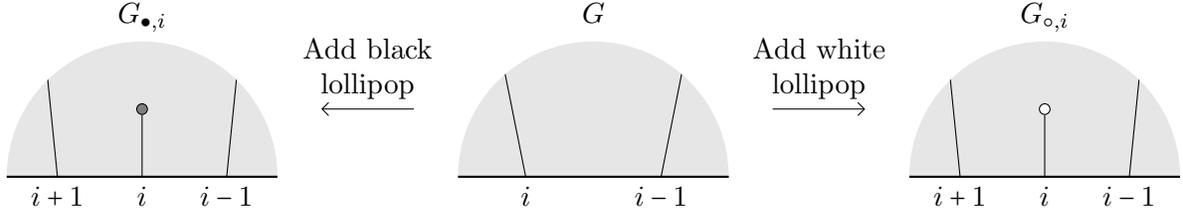
\begin{figure}[h!t]
\begin{tikzpicture}
	\node (2) at (0,-1) {
        \begin{tikzpicture}[scale=.45]
                	\path[use as bounding box] (-5,-1) rectangle (5,5);
		\node[above] at (0,4.25) {$G$};
                	\path[clip] (0,0) circle (4);
                	\draw[thick,fill=black!10] (-5,0) rectangle (5,5);
                	\node[coordinate] (c) at (-2,0) {};
                	\node[coordinate] (b) at (0,0) {};
                	\node[coordinate] (a) at (2,0) {};
                	\node[below] at (a) {$i-1$};
                	\node[below] at (c) {$i$};
                	\draw (c) to (-3,5);
                	\draw (a) to (3,5);
        \end{tikzpicture}};
	\node (1) at (-6,-1) {
        \begin{tikzpicture}[scale=.45]
                	\path[use as bounding box] (-5,-1) rectangle (5,5);
		\node[above] at (0,4) {$G_{\bullet,i}$};
                	\path[clip] (0,0) circle (4);
                	\draw[thick,fill=black!10] (-5,0) rectangle (5,5);
                	\node[coordinate] (c) at (-2.5,0) {};
                	\node[coordinate] (b) at (0,0) {};
                	\node[coordinate] (a) at (2.5,0) {};
                	\node[below] at (a) {$i-1$};
		\node[below] at (b) {$i$};
               	\node[below] at (c) {$i+1$};
              	\draw (c) to (-3,5);
                	\draw (a) to (3,5);
		\node[dot,fill=black!50] (d) at (0,2) {};
		\draw (b) to (d);
        \end{tikzpicture}};
	\node (3) at (6,-1) {
        \begin{tikzpicture}[scale=.45]
                	\path[use as bounding box] (-5,-1) rectangle (5,5);
		\node[above] at (0,4) {$G_{\circ,i}$};
                	\path[clip] (0,0) circle (4);
                	\draw[thick,fill=black!10] (-5,0) rectangle (5,5);
                	\node[coordinate] (c) at (-2.5,0) {};
                	\node[coordinate] (b) at (0,0) {};
                	\node[coordinate] (a) at (2.5,0) {};
                	\node[below] at (a) {$i-1$};
		\node[below] at (b) {$i$};
               	\node[below] at (c) {$i+1$};
                	\draw (c) to (-3,5);
                	\draw (a) to (3,5);
		\node[dot,fill=white] (d) at (0,2) {};
		\draw (b) to (d);
        \end{tikzpicture}};
        \draw[-angle 90] (2) to node[above,align=center] {Add black\\lollipop} (1);
        \draw[-angle 90] (2) to node[above,align=center] {Add white\\lollipop} (3);
\end{tikzpicture}
\caption{Adding lollipops}
\label{fig: lollipop}
\end{figure}


Let $G$ be a reduced graph with $n-1$ boundary vertices and bounded affine permutation $\pi$. Figure \ref{fig: lollipop} shows two new graphs $G_{\bullet,i}$ and $G_{\circ,i}$ on $n$ vertices; we say that they are the result of adding a \newword{black lollipop} or \newword{white lollipop} to $G$ in position $i$. 
Write $\sigma: \ZZ \to \ZZ$ for the order preserving injection whose image is $\{ j \in \ZZ : j \not \equiv i \bmod n \}$, with $\sigma(i+1)=i+1$. 
The following lemma is an immediate computation:

\begin{lemma} \label{lollipop}
The graphs $G_{\bullet, i}$ and $G_{\circ, i}$ are reduced. Writing $\pi_{\bullet,i}$ and $\pi_{\circ, i}$ for the corresponding bounded permutations. For $j \in \ZZ$, we have
\[ \pi_{\bullet,i}(j) = \begin{cases} i & j = i \\ \sigma(\pi(\sigma^{-1}(j))) & j \neq i \\ \end{cases} \qquad 
 \pi_{\circ,i}(j) = \begin{cases} i+n & j = i \\ \sigma(\pi(\sigma^{-1}(j))) & j \neq i \\ \end{cases} \]
Let $x$ be a point of $\widetilde{Gr(k,n-1)}$ parametrized by $G$ and let $x_{\bullet,i}$ and $x_{\circ, i}$ be the corresponding points of $\widetilde{Gr(k,n)}$ and  $\widetilde{Gr(k+1,n)}$. Then we have the equalities of Pl\"ucker coordinates
\[ \Delta_J(x_{\bullet,i}) = \begin{cases} \Delta_{\sigma^{-1}(J)}(x) & i \not \in J \\ 0 & i \in J \end{cases} \qquad
 \Delta_J(x_{\circ,i}) = \begin{cases} t \Delta_{\sigma^{-1}(J \setminus \{ i \} )}(x) & i  \in J \\ 0 & i \not \in J \end{cases} \]

\end{lemma}

Adding a lollipop is a trivial way to change a reduced graph; adding a \newword{bridge} is a less trivial way.
Let $G$ be a reduced graph with bounded affine permutation $\pi$. Define $s_i G$ and $G s_i$ to be the graphs shown in Figure \ref{fig: bridges}; we say that $s_i G$ is $G$ with a \newword{left bridge} added between $i$ and $i+1$ and $G s_i$ has a \newword{right bridge} added.
(If we have introduced an edge between two vertices of the same color, we contract it.)

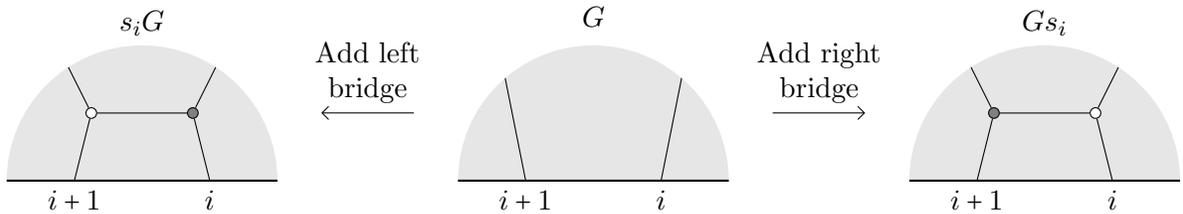
\begin{figure}[h!t]
\begin{tikzpicture}
	\node (2) at (0,-1) {
        \begin{tikzpicture}[scale=.45]
                	\path[use as bounding box] (-5,-1) rectangle (5,5);
		\node[above] at (0,4.25) {$G$};
                	\path[clip] (0,0) circle (4);
                	\draw[thick,fill=black!10] (-5,0) rectangle (5,5);
                	\node[coordinate] (c) at (-2,0) {};
                	\node[coordinate] (b) at (0,0) {};
                	\node[coordinate] (a) at (2,0) {};
                	\node[below] at (a) {$i$};
                	\node[below] at (c) {$i+1$};
                	\draw (c) to (-3,5);
                	\draw (a) to (3,5);
        \end{tikzpicture}};
	\node (1) at (-6,-1) {
        \begin{tikzpicture}[scale=.45]
                	\path[use as bounding box] (-5,-1) rectangle (5,5);
		\node[above] at (0,4) {$s_iG$};
                	\path[clip] (0,0) circle (4);
                	\draw[thick,fill=black!10] (-5,0) rectangle (5,5);
                	\node[coordinate] (c) at (-2,0) {};
                	\node[coordinate] (b) at (0,0) {};
                	\node[coordinate] (a) at (2,0) {};
                	\node[below] at (a) {$i$};
                	\node[below] at (c) {$i+1$};
		\node[dot,fill=white] (c') at (-1.5,2) {};
		\node[dot,fill=black!50] (a') at (1.5,2) {};
                	\draw (c) to (c') to (-3,5);
                	\draw (a) to (a') to  (3,5);
		\draw (a') to (c');
        \end{tikzpicture}};
	\node (3) at (6,-1) {
        \begin{tikzpicture}[scale=.45]
                	\path[use as bounding box] (-5,-1) rectangle (5,5);
		\node[above] at (0,4) {$Gs_i$};
                	\path[clip] (0,0) circle (4);
                	\draw[thick,fill=black!10] (-5,0) rectangle (5,5);
                	\node[coordinate] (c) at (-2,0) {};
                	\node[coordinate] (b) at (0,0) {};
                	\node[coordinate] (a) at (2,0) {};
                	\node[below] at (a) {$i$};
                	\node[below] at (c) {$i+1$};
		\node[dot,fill=black!50] (c') at (-1.5,2) {};
		\node[dot,fill=white] (a') at (1.5,2) {};
                	\draw (c) to (c') to (-3,5);
                	\draw (a) to (a') to  (3,5);
		\draw (a') to (c');
        \end{tikzpicture}};
        \draw[-angle 90] (2) to node[above,align=center] {Add left\\bridge} (1);
        \draw[-angle 90] (2) to node[above,align=center] {Add right\\bridge} (3);
\end{tikzpicture}
\caption{Adding bridges}
\label{fig: bridges}
\end{figure}

Let $s_i$ be the following permutation of $\ZZ$:
\[ s_i(j) = \begin{cases} j+1 & j \equiv i \bmod n \\ j-1 & j \equiv i+1 \bmod n \\ j & \mbox{otherwise} \end{cases}. \]
We now summarize the key properties of adding a bridge.

\begin{lemma} \label{bridge}
If $\pi(i) > \pi(i+1)$, then $G s_i$ is a reduced graph with bounded affine permutation $\pi \circ s_i$. If $\pi^{-1}(i) > \pi^{-1}(i+1)$, then $s_i G$ is a reduced graph with bounded affine permutation $s_i \circ \pi$.

Let $x$ be a point of $\widetilde{Gr(k,n)}$ parametrized by $G$ and let $y$ and $z$ be the points of $\widetilde{Gr(k,n)}$ and  $\widetilde{Gr(k,n)}$ corresponding to adding left and right bridges as shown. Then we have the equalities of Pl\"ucker coordinates
\[ \Delta_J(y) = \begin{cases}
\Delta_J(x) + t \Delta_{J \setminus \{ i \} \cup \{ i+1 \}}(x) & i \in J,\ i+1 \not \in J \\ \Delta_J(x) & \mbox{otherwise} \\ \end{cases} \]
\[  \Delta_J(z) = \begin{cases}
\Delta_J(x) + t \Delta_{J \setminus \{ i+1  \} \cup \{ i \}}(x) & i+1 \in J,\ i \not \in J \\ \Delta_J(x) & \mbox{otherwise} \\ \end{cases}. \]
\end{lemma}

The key point is that, by combining lollipops and bridges, we can build a reduced graph for any bounded affine permutation.

\begin{lemma} \label{bridge decompositions exist}
Let $\rho$ be a bounded affine permutation of type $(k,n)$, for $n>1$. Let $f$ be the number of faces in any reduced graph for $\rho$. Then (at least) one of the following holds:
\begin{enumerate}
\item There is some $i$ with $\rho(i) = i$. In this case, we can obtain a reduced graph for $\rho$ by adding a black lollipop to some reduced graph on $n-1$ vertices.
\item There is some $i$ with $\rho(i) = i+n$. In this case, we can obtain a reduced graph for $\rho$ by adding a white lollipop to some reduced graph on $n-1$ vertices.
\item There is some $i$ with $\rho(i) < \rho(i+1)$ and $s_i \rho$ a bounded affine permutation.  In this case, we can obtain a reduced graph for $\rho$ by adding a left bridge to some reduced graph for $s_i \circ \rho$, which will have $f-1$ faces.
\item There is some $i$ with $\rho^{-1}(i) < \rho^{-1}(i+1)$ and $\rho s_i$ a bounded affine permutation.  In this case, we can obtain a reduced graph for $\rho$ by adding a right bridge to some reduced graph for $\rho \circ s_i$, which will have $f-1$ faces.
\end{enumerate}
\end{lemma}

\begin{remark} In fact, as the reader will see from the proof, at least one of $(1)$, $(2)$, $(3)$ always holds, and at least one of $(1)$, $(2)$, $(4)$ always holds.
\end{remark}

\begin{proof}
If $\rho(i) = i$, then define the bounded affine permutation $\pi$ of type $(k,n-1)$ so that $\pi_{\bullet, i} = \rho$. Taking any reduced graph for $\pi$; adding a black lollipop gives a reduced graph for $\rho$.
Similarly, if $\rho(i) = i+n$, then we can define $\pi$ so that $\pi_{\circ, i} = \rho$; we can obtain a reduced graph for $\rho$ by adding a white lollipop to a reduced graph for $\pi$.

So we may assume that $i < \rho(i) < i+n$ for all $i$. This means that $s_i \circ \rho$ and $\rho \circ s_i$ will still be bounded affine permutations.
We know that $\rho(n) = \rho(0) + n > \rho(0)$. Therefore, for some $i$ between $0$ and $n-1$, we must have $\rho(i+1) > \rho(i)$ and case~$(3)$ applies for this $i$. 
For similar reasons, case~$(4)$ applies for some $i$.
\end{proof}

\section{Proof of the main theorem} \label{sec main thm proof}

Over the next several sections, we will prove Theorem \ref{main theorem}. The majority of the proof will consist of proving the following lemma.

\begin{lemma}\label{lemma: main}
Let $G$ be a reduced graph, and let $z\in \mathbb{G}_m^E$.  For any face $f\in F$, 
\[ \Delta_{\sI(f)}(\rt(\widetilde{\mathbb{D}}(z))) = \frac{1}{z^{\vecM(f)}} \]
\end{lemma}

\subsection{At a boundary face}

At a boundary face, Lemma \ref{lemma: main} follows directly from prior results.

\begin{prop}\label{prop: lemmaboundary}
If $f$ is a boundary face in $G$, then Lemma \ref{lemma: main} holds.
\end{prop}
\begin{proof}
Let $f$ be the boundary face between boundary vertices $a-1$ and $a$.  By Proposition \ref{prop: boundarylabel}, $\sI(f)=\cevI_a$.  By Proposition \ref{prop: boundaryunique}, $\vecM(f)$ is the unique matching with boundary $\cevI_a$, and so
\[ \Delta_{\cevI_a}(\tDD(z)) = z^{\vecM(f)} \]
By Theorem \ref{thm: inversetwists} and the analog of Equation \eqref{eq: twistboundary} for left twists,
\[ \Delta_{\cevI_a}(\tDD(z)) = \Delta_{\cevI_a}(\lt(\rt(\tDD(z)))) = \frac{1}{\Delta_{\cevI_a}(\rt(\tDD(z)))} \]
Combining the two equalities proves the proposition.
\end{proof}

This establishes Lemma \ref{lemma: main} for reduced graphs without internal faces.  This will be the base case of our inductive argument. 

\subsection{Move-equivalence}

\begin{lemma} \label{lemma: mutation invariance}
If Lemma \ref{lemma: main} holds for a reduced graph $G$, then it also holds for any reduced graph obtained from $G$ by:
\begin{enumerate}
	\item contracting or expanding a degree two vertex,
	\item removing or adding a boundary-adjacent degree two vertex,
	\item removing or adding a lollipop, or
	\item urban renewal.
\end{enumerate}
\end{lemma} 

\begin{proof}
We first consider the first three cases, which are easy. Let $G$ be the graph without the degree two vertex/lollipop in question and let $G'$ be the modified graph.
Then there is a straightforward bijection between matchings of $G$ and of $G'$.  This matching preserves the values of the boundary measurement map and takes the minimal matching of $G$ to the minimal matching of $G'$. 

We now consider the case of urban renewal. We will use the notations from Figure~\ref{fig: urban}. We denote the central square face by $s$. We write $G$ for the graph before mutation (left side of the figure) and $G'$ for the mutated graph (right side of the figure). We will use primed variables $V'$, $E'$, $F'$ for the sets of faces of $G'$. 
For a face $g \in F$, we write $g'$ for the corresponding face of $G'$, by the obvious bijection. 

Let $z$ be an element of $\GG_m^{E}/\GG_m^{V-1}$.
Let $w$ be a lift of $z$ to $\GG_m^E$. Let $w'$ be the element of $\GG_m^E$ given by the formulas on the right side of Figure~\ref{fig: urban}. 
Let $w''$ be the result of applying a gauge transformation by $b_1 b_3 + b_2 b_4$ to $w'$ at some vertex of $G$ and let $z''$ be the image of $w''$ in $\GG_m^{E'}/\GG_m^{V'-1}$.
For any   matching $M$ of $G$, we have
\[
(z'')^{M} = (w'')^M =  (b_1 b_3+b_2 b_4) (w')^M
\]

Set $q=\rt(\tDD(z))$. We know that $\tDD(z) = \tDD(z'')$ (Urban Renewal preserves boundary measurements) and $\rt$ is a well defined map, so we also have $q = \rt(\tDD(z''))$.

So our goal is to establish that 
\[ \Delta_{\sI(f)}(q) = \frac{1}{z^{\vecM(f)}} \ \forall_{f \in F} \quad \mathrm{implies} \quad \Delta_{\sI(f')}(q) = \frac{1}{(z'')^{\vecM(f')}} \ \forall_{f' \in F'} .\]

We split into two cases:

\textbf{Case 1:  $f' \neq s'$.} From Theorem~\ref{thm: minmatchprop}, the square $s$ must have one edge in the minimal matching $\vecM(f)$. Without loss of generality, let it be the edge $b_1$.
Looking at how strands change under urban renewal, we see that $\vecM(f')$ is the unique matching which agrees with $\vecM(f)$ at every edge which is in both $G$ and $G'$. (\emph{I.e.} all but the four edges in the left hand side of Figure~\ref{fig: urban} and the eight edges on the right hand side of Figure~\ref{fig: urban}.)
Let $u$ be the product of the weights on all edges that $\vecM(f)$ and $\vecM(f')$ have in common. Then $w^{\vecM(f)} = b_1 u$ and $(w')^{\vecM(f')}= \frac{b_1}{b_1 b_3 + b_2 b_4} u$ so $(w'')^{\vecM(f')} = b_1 u$. 

In this case, $\sI(f) = \sI(f')$; denote this common value by $I$. We are assuming we know $\Delta_I(q) = \frac{1}{z^{\vecM(f)}}$. We deduce
\[ \Delta_{\sI(f')}(q) = \Delta_I(q) = \frac{1}{z^{\vecM(f)}}= \frac{1}{w^{\vecM(f)}} = \frac{1}{b_1 u} = \frac{1}{(w'')^{\vecM(f')} } = \frac{1}{(z'')^{\vecM(f')}} \]
as desired.

\textbf{Case 2: $f'=s'$.} Let $f_1$, $f_2$, $f_3$ and $f_4$ be the faces of $G$ adjacent to $s$, with $f_i \cap s$ the edge weighted $b_i$. There is a $k-2$ element set $S$ and indices $(a,b,c,d)$ such that $\sI(s)$, $\sI(f_1)$, $\sI(f_2)$, $\sI(f_3)$, $\sI(f_4)$ and $\sI(s')$ are $Sac$, $Sab$, $Sbc$, $Scd$, $Sad$ and $Sbd$ respectively.

We have the Pl\"ucker relation
\[ \Delta_{\sI(s')}(q) = \Delta_{Sbd}(q) = \frac{\Delta_{Sab}(q) \Delta_{Scd}(q) + \Delta_{Sbc}(q) \Delta_{Sad}(q)}{\Delta_{Sac}(q)}. \]
Since all the terms on the right hand side label faces of $G$, our assumption to know Lemma~\ref{lemma: main} for $G$ gives 
\[ \Delta_{\sI(s')}(q) =  z^{\vecM(s)} \left( \frac{1}{z^{\vecM(f_1)} z^{\vecM(f_3)}} + \frac{1}{z^{\vecM(f_2)} z^{\vecM(f_4)}} \right) =    \frac{w^{\vecM(s)} }{w^{\vecM(f_1)} w^{\vecM(f_3)} } +\frac{w^{\vecM(s)} }{w^{\vecM(f_2)} w^{\vecM(f_4)}}  . \]
We want to show this equals
\[ \frac{1}{(z'')^{\vecM(s')}} =  \frac{1}{(w'')^{\vecM(s')}} = \frac{1}{(b_1 b_3 + b_2 b_4)\cdot  (w')^{\vecM(s')}}. \]
In other words, we want to show
\begin{equation}
  \frac{w^{\vecM(s)} (w')^{\vecM(s')} }{w^{\vecM(f_1)} w^{\vecM(f_3)}} +  \frac{w^{\vecM(s)} (w')^{\vecM(s')} }{w^{\vecM(f_2)} w^{\vecM(f_4)} } = \frac{1}{b_1 b_3 + b_2 b_4}  \label{Prove This}
  \end{equation}

  Let $\gamma_a$, $\gamma_b$, $\gamma_c$ and $\gamma_d$ be the halves of strands $a$, $b$, $c$ and $d$ running towards $b_1$ and $b_3$.
    Consider an edge $e$ of $G$, other then the ones labeled $b_1$, $b_2$, $b_3$, $b_4$. If $e$ does not lie on any of $\gamma_a$, $\gamma_b$, $\gamma_c$, $\gamma_d$
, then the weight $w_e$ occurs in either all the matching monomials of~(\ref{Prove This}), or none of them, and thus cancels out. 
If $e$ lies on one of these strands, then $w_e$ occurs once in each numerator and once in each denominator, so it cancels again. 
So the only terms that don't cancel from the left hand side of~(\ref{Prove This}) are the terms coming from the four edges of $s$. Adding them up, the left hand side of~(\ref{Prove This}) is
\[ \frac{(b_2 b_4) \cdot \frac{b_1 b_3}{(b_1 b_3 + b_2 b_4)^2}}{b_1 \cdot b_3} +  \frac{(b_1 b_3) \cdot \frac{b_2 b_4}{(b_1 b_3 + b_2 b_4)^2}}{b_2 \cdot b_4} = \frac{1}{b_1 b_3 + b_2 b_4}, \]
as desired.
\end{proof}

\subsection{Adding a left bridge} \label{subsec bridge}

\newcommand\bridge[1]{\widehat{#1}}

\begin{figure}[h!t]
\begin{tikzpicture}
\begin{scope}
	\begin{scope}\clip (0,0) circle (2);
	\draw[fill=black!10] (-3,-1) rectangle (3,3);
	\node[mutable,fill=white] (b) at (-1,.5) {};
	\node[mutable,fill=black!50] (c) at (1,.5) {};
	\draw (-1.5,-1) to node[right,blue] {$z_1$} (b) to node[below,blue] {$z_2$} (c) to node[left,blue] {$z_3$} (1.5,-1);
	\draw (-2,.5) to (b) to (-1.5,2);
	\draw (2,.5) to (c) to (1.5,2);
	\draw[antioriented,out=135,in=-45] (1.5,-1) to (0,.5);
	\draw[antioriented,out=135,in=300] (0,.5) to (-1.25,2);
	\draw[antioriented,out=45,in=225] (-1.5,-1) to (0,.5);
	\draw[antioriented,out=45,in=240] (0,.5) to (1.25,2);	
	\end{scope}
	\node[below] at (1.5,-1) {$b$};
	\node[below] at (-1.5,-1) {$b+1$};
\end{scope}
\end{tikzpicture}
\caption{A left bridge between $b$ and $b+1$.}
\label{fig: leftbridge}
\end{figure}
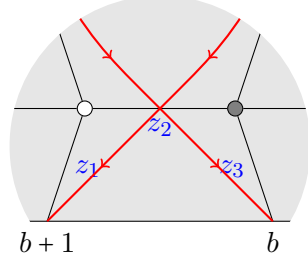

Let $\bridge{G}$ be a reduced graph with a left bridge between $b$ and $b+1$, and let $G$ be the graph of $\bridge{G}$ with the left bridge removed.   The number $n$ of boundary vertices and the cardinality $k$ of the boundary of any matching are the same for both $\bridge{G}$ and $G$.   
Any set of non-zero edge weights $\bridge{z}$ on $\bridge{G}$ restricts to a set of non-zero edge weights $z$ on $G$.

 Let $z_1$, $z_2$, $z_3$ be the weights on the edges in Figure \ref{fig: leftbridge}.
The image of the boundary measurement map on $z$ and $\bridge{z}$ are related as follows.

\begin{prop}\label{prop: bridge}
If a matrix $A$ represents $\tDD(z)$, then the matrix $\bridge{A}$ with
\[ \bridge{A}_a := \left\{\begin{array}{cc}
A_a & \text{if }a\neq b \\
A_b+\frac{z_2}{z_1z_3}A_{b+1} &\text{if }a=b\\
\end{array}\right\} \]
represents $\tDD(\bridge{z})$.
\end{prop}
\begin{proof}
A matching $M$ of $\bridge{G}$ which doesn't contain the bridge restricts to a matching of $G$, and all matchings of $G$ occur this way.  The associated monomials in weights coincide: $\bridge{z}^M=z^M$.

A matching $M$ of $\bridge{G}$ which contains the bridge cannot also contain the external edges at vertices $b$ and $b+1$.  Hence, there is a matching $M'$ of $G$ which is the restriction of $M$ together with the external edges at vertices $b$ and $b+1$.  We have
\[ \partial M' = (\partial M \setminus\{b\})\cup \{b+1\}\]
and every matching of $G$ whose boundary contains $b+1$ but not $b$ occurs this way.  The associated monomials are related by
\[ \bridge{z}^M = \frac{z_2}{z_1z_3}z^{M'}\]

Hence, $D_I(\bridge{z})=D_I(z)$ for all $I$ which either contain $b+1$ or don't contain $b$.  
For any $(k-1)$ element set $J\subset [n]$ disjoint from $b$ and $b+1$, 
\[ D_{J\cup\{b\}}(\bridge{z}) = D_{J\cup\{b\}}(z) + \frac{z_2}{z_1z_3}D_{J\cup \{b+1\} }(z) \]
It follows that the maximal minors of $\bridge{A}$ coincide with the partition functions of $\bridge{G}$ on $\bridge{z}$.
\end{proof}

\begin{lemma} \label{lem in span}
Let $A$ and $\bridge{A}$ be as in Proposition \ref{prop: bridge}, and let $\pi$ be the bounded affine permutation of $A$.  Let $\tau_a$ denote the columns of $\rt(A)$ and let $\bridge{\tau}_a$ denote the columns of $\rt(\bridge{A})$. 

With the above notations, $\bridge{\tau}_a-\tau_a$ is in the span of $\{ \tau_b : b \impfrom_{\pi} a \}$.
\end{lemma}

We write $\vecI_a$ and $\bridge{\vecI}_a$ for the Grassmann necklace of $\pi$ and $s_b \circ \pi$.

\begin{proof}
Set $c = \pi^{-1}(b)$ and $d = \pi^{-1}(b+1)$. 
Since we are assuming that $\ell(s_b \circ \pi) = \ell(\pi) +1$, we have $d < c \leq b < b+1$. 
We first identify a number of cases where $\bridge{\tau}_a = \tau_a$. 

\textbf{Case 1:} $a \in [b+1, d+n]$. In this case, $\vecI_a = \bridge{\vecI}_a$, and this set does not contain $b$. So $\tau_a$ and $\bridge{\tau}_a$ are defined by duality to the same basis, and $\tau_a = \bridge{\tau}_a$.

\textbf{Case 2:} $a \in (c, b)$. In this case, $\vecI_a = \bridge{\vecI}_a$. We have $b$ and $b+1 \in \vecI_a \setminus \{ a \}$. Although $A_b \neq \bridge{A}_b$, we have $\mathrm{span}(A_b, A_{b+1}) = \mathrm{span}(\bridge{A}_b, \bridge{A}_{b+1})$ so $\tau_a$ and $\bridge{\tau}_a$ are defined to be orthogonal to the same $k-1$ plane. This shows that $\tau_a$ and $\bridge{\tau}_a$ are proportional, and they both have dot product $1$ with $A_a = \bridge{A}_a$. 

\textbf{Case 3:} $a = b$. In this case, $\vecI_b = \bridge{\vecI}_b$. For $c \in \vecI_b \setminus \{ b \}$, we have $A_c = \bridge{A}_c$ so, as in case 2, $\tau_b$ and $\bridge{\tau}_b$ are defined to be orthogonal to the same $k-1$ plane, and are hence proportional. To see that the proportionality constant is the same, note that we have $1=\langle \tau_b, A_b \rangle $ and $1 = \langle \bridge{\tau}_b, A_b + \frac{z_2}{z_1 z_3} A_{b+1} \rangle$. But $b+1 \in \vecI_b$, so $\langle \bridge{\tau}_b, A_{b+1} \rangle = 0$ and we see that  $1=\langle \bridge{\tau}_b, A_b \rangle$, establishing $\tau_b = \bridge{\tau}_b$. 

In short, we have so far established $\tau_a = \bridge{\tau}_a$ for $a \in (c, d+n]$. Therefore, from now on, we are in

\textbf{Case 4:} $a \in (d, c]$. 

In this case, we have $\vecI_a = S \cup \{ b+1 \}$ and $\bridge{\vecI}_a = S \cup \{ b \}$ for some $k-1$ element subset $S$ of $[n] \setminus \{ b, b+1 \}$. We know that $\pi(a) \in [a, a+n]$ and, as $a \neq d$, we have $\pi(a) \neq b+1$. We break into two further cases:

\textbf{Case 4a:} $\pi(a) \in (b+1, a+n]$. In this case we claim that, one more time, we have $\tau_a = \bridge{\tau}_a$. We check that $\tau_a$ obeys the defining properties of $\bridge{\tau}_a$. We have $\langle \tau_a, \bridge{A}_a \rangle = \langle \tau_a, A_a \rangle =1$. Also, for $s \in S \setminus \{ a \}$, we have $\langle \tau_a, \bridge{A}_s \rangle = \langle \tau_a, A_s \rangle =0$. It remains to check that $\langle \tau_a, \bridge{A}_{b} \rangle = 0$.  Our assumption on $\pi(a)$ implies that $b$ and $b+1 \in (a, \pi(a))$ so, by Lemma~\ref{lem more tau vanishing}, we have $\langle \tau_a, A_b \rangle = \langle \tau_a, A_{b+1} \rangle =0$. Thus, $\langle \tau_a, \bridge{A}_{b} \rangle =\langle \tau_a, A_{b} \rangle + \frac{z_2}{z_1 z_3} \langle \tau_a, A_{b+1} \rangle = 0$ as desired.

Finally, we reach the sole case where $\tau_a \neq \bridge{\tau}_a$:

\textbf{Case 4b:} $\pi(a) \in [a, b]$. Define $\mu:= \bridge{\tau}_a - \tau_a$. The defining properties of $\bridge{\tau}_a$ and $\tau_a$ give $\langle \mu, A_s \rangle =0$ for $s \in S$. In particular, since $b$, $b+1 \not \in [a, \pi(a))$, we have $\langle \mu, A_s \rangle =0$ for $s \in \vecI_a \cap [a, \pi(a))$. But, by Lemma~\ref{lem even more tau vanishing}, this means that $\mu$ is in the span of $\{ \tau_b : b \impfrom_{\pi} a \}$, which is the desired conclusion.
\end{proof}

We can now establish the bridge case of the inductive step for our proof of Lemma \ref{lemma: main}.

\begin{lemma} \label{lemma add bridge}
If Lemma \ref{lemma: main} holds at each face in $G$, then it holds for each face in $\bridge{G}$.
\end{lemma}
\begin{proof}
%
%
We reuse the notations $A$, $\bridge{A}$, $\tau$ and $\bridge{\tau}$ of the previous Lemma.

Let us consider a face $\bridge{f}$ of $\bridge{G}$ which is not the boundary face between vertices $b$ and $b+1$.  Then $\bridge{f}$ corresponds to a face $f$ in $G$, and they have the same source-indexed face label. Define $I:= \sI(\bridge{f}) = \sI(f)$.
Furthermore, $\bridge{f}$ is not downstream from the bridge, and if $\bridge{f}$ is downstream from an edge $e$ in $\bridge{G}$, then $f$ is downstream from $e$ as an edge in $G$.  Hence, the minimal matchings coincide: $\vecM(\bridge{f})=\vecM(f)$.
Our assumption is that $\Delta_I(\tau) = 1/z^{\vecM(f)}$, and we have just shown $1/z^{\vecM(\bridge{f})}=1/z^{\vecM(f)}$.
So our goal is to prove that $\Delta_I(\tau) = \Delta_{I}(\bridge{\tau})$. 

By Lemma~\ref{label implication}, if $a \in I$ and $a \impto_{\pi} p$, then $p \in I$. Choose an order of $I$ refining the partial order $\impto_{\pi}$. By Lemma~\ref{lem in span}, when ordered in this manner, the bases $\{ \tau_a : a \in I \}$ and $\{ \bridge{\tau}_a : a \in I \}$ are related by an upper triangular matrix with $1$'s on the diagonal. So $\Delta_I(\tau) = \Delta_{I}(\bridge{\tau})$ as desired.

We have now established Lemma~\ref{lemma: main} at every face of $\bridge{G}$ except the boundary face between vertices $b$ and $b+1$. At this face, Lemma \ref{lemma: main} holds by Proposition~\ref{prop: lemmaboundary}. \end{proof}

\subsection{Conclusion of the proof} We may now complete the proof of the main theorem.

\begin{proof}[Proof of Lemma \ref{lemma: main}]
We have shown that Lemma~\ref{lemma: main} is true for reduced graphs with no internal faces (Proposition~\ref{prop: lemmaboundary}), that it remains true after 
adding a lollipop (Lemma~\ref{lemma: mutation invariance}) or a left bridge (Lemma~\ref{lemma add bridge}), and that it remains true after any mutation (Lemma~\ref{lemma: mutation invariance}). For any bounded affine permutation $\pi$, a reduced graph for $\pi$ can be built via repeatedly adding bridges and lollipops (Lemma~\ref{bridge decompositions exist}) and any two reduced graphs for $\pi$ are connected by a sequence of mutations (\cite[Theorem 13.4]{Pos}, see also \cite{OS14}). 
\end{proof}

\begin{proof}[Proof of Theorem \ref{main theorem}]
We prove the commutativity of the right-hand square in the Theorem; the other square will follow by a mirror argument. The commutative of the pairs of horizontal arrows is equivalent to Proposition \ref{prop: isotori} and Theorem \ref{thm: inversetwists}.

By Lemma \ref{lemma: main}, the composition $\sF\circ \rt \circ \tDD$ is regular and equal to $\rM$. This implies the commutativity of any pair of paths in the right-hand square which begin in the top row. 
In particular, it implies that the restriction of $\lpartial\circ \sF\circ \rt$ to the image of $\tDD$ is a (regular) right inverse to $\tDD$.

The positroid variety $\tPio(\cM)$ has dimension $k(n-k)- \ell(\pi)$ by \cite[Theorem 5.9]{KLS13}. By a combination of Theorem~12.7 and Proposition~17.10 in \cite{Pos}, this is equal to $|F|$, the number of faces of $G$. Hence, $\tDD$ is a regular map between integral varieties of the same dimension with a right inverse; hence it is an open inclusion. 

We compute
\[ \tDD\circ \lpartial \circ \sF \circ \rt \circ \tDD \stackrel{\text{Lemma~}\ref{lemma: main}}{=} \tDD\circ \lpartial \circ \rM \stackrel{\text{Prop.~}\ref{prop: isotori}}{=} \tDD \]
This implies that $\sF\circ \rt \circ \tDD\circ \rt$ is the identity on the image of $\tDD$. Since $\tDD$ is an open inclusion, this implies that $\sF\circ \rt \circ \tDD\circ \rt$ is equal to the identity as rational maps. This implies the commutativity of any pair of paths in the right-hand square which begin in the bottom row.
\end{proof}

\appendix

\section{Examples of the twist}\label{sec nice twist}

This appendix collects several examples in which the twist is simple or notable.

\subsection{Uniform positroid varieties}

\def\cMg{\cM_{uni}}

For fixed $k\leq n$, let $\cMg$ be the \newword{uniform positroid}, the matroid in which every $k$-element subset of $[n]$ is a basis.  The variety $\Mat^\circ(\cMg)$ parametrizes $k\times n$ complex matrices such that each cyclically consecutive minor
\[ \Delta_{12...k},\Delta_{23...(k+1)},...,\Delta_{(n-k+1)(n-k+2)...n},\Delta_{(n-k+2)(n-k+3)...n1},...,\Delta_{n1...(k-1)}\]
is non-zero.  We refer to $\Pio(\cMg)\subset Gr(k,n)$ as the \newword{(open) uniform positroid variety}; it is the open subvariety defined by the non-vanishing of the cyclically consecutive Pl\"ucker coordinates.\footnote{The open uniform positroid variety in $Gr(k,n)$ is the complement of a simple normal-crossing canonical divisor, making it an example of an \emph{affine log Calabi-Yau variety with maximal boundary}, in the sense of \cite{GHK15}.  This may be the source of the cluster structure, according to the perspective of \cite{GHKK}.}





\begin{myexample}
We consider the case $k=1$. The matrices in $\Mat(1,n)$ with uniform positroid envelope are those with no zero entries.
The twist acts on matrices by
\[ \rt\begin{bmatrix}a_1 & a_2 & \cdots & a_n \end{bmatrix} = \begin{bmatrix} a_1^{-1} & a_2^{-1} & \cdots & a_n^{-1}\end{bmatrix} \]
Here, $\Mato(\cMg)$ is an algebraic torus and the twist is inversion in the torus, so it has order $2$.

The Grassmannian $Gr(1,n)$ is projective space $\mathbb{P}^{n-1}$, and the open uniform positroid subvariety $\Pio(\cMg)$ is the subset on which no homogeneous coordinate vanishes.  The twist acts by simultaneously inverting each homogeneous coordinate.
\end{myexample}


\begin{myexample}\label{example: Gr2n}

We consider the case $k=2$. A matrix $A=\left( \begin{smallmatrix} a_1 & a_2 & \cdots & a_n \\ b_1 & b_2 & \cdots & b_n \end{smallmatrix} \right)$ has uniform positroid envelope if each $\Delta_{i(i+1)}:=a_ib_{i+1}-a_{i+1}b_i$ is non-zero. The twist acts by
\[\rt\begin{bmatrix} a_1 & a_2 & \cdots & a_n \\ b_1 & b_2 & \cdots & b_n \end{bmatrix}
=  \begin{bmatrix} \frac{b_2}{\Delta_{12}} & \frac{b_3}{\Delta_{23}} & \cdots & \frac{b_1}{\Delta_{n1}} \\ \frac{-a_2}{\Delta_{12}} & \frac{-a_3}{\Delta_{23}} & \cdots & \frac{-a_1}{\Delta_{n1}} \end{bmatrix}\]
So the $(i,j)$-th Pl\"ucker coordinate of the twist is $\frac{\Delta_{(i+1)(j+1)}}{\Delta_{i(i+1)}\Delta_{j(j+1)}}$.  In particular, up to an invertible monomial transformation, the $(i,j)$-th Pl\"ucker coordinate of the twist is the same as the $(i+1,j+1)$-st Pl\"ucker coordinate of the original matrix.  Our main result says that the $(i,j)$-th Pl\"ucker coordinate of the twist can be written in terms of the Pl\"ucker coordinates in any cluster as a sum over matchings in a planar graph; see \cite{CP03,Mus11} for examples of such formulas.

We can also observe that the square of the twist acts by
\[\rt^2\begin{bmatrix} a_1 & a_2 & \cdots & a_n \\ b_1 & b_2 & \cdots & b_n \end{bmatrix}
=  \begin{bmatrix} \frac{-\Delta_{12}}{\Delta_{23}}a_3 & \frac{-\Delta_{23}}{\Delta_{34}}a_4 & \cdots & \frac{-\Delta_{n1}}{\Delta_{12}}a_2 \\ \frac{-\Delta_{12}}{\Delta_{23}}b_3 & \frac{-\Delta_{23}}{\Delta_{34}}b_4 & \cdots & \frac{-\Delta_{n1}}{\Delta_{12}}b_2 \end{bmatrix}\]
This implies the order of $\rt$ depends on the parity of $n$.
If $n$ is odd, then $\rt^{2n}$ is the identity. 

If $n$ is even, then 
\[\rt^{n}\begin{bmatrix} a_1 & a_2 & \cdots & a_n \\ b_1 & b_2 & \cdots & b_n \end{bmatrix}
=  \begin{bmatrix} \alpha a_1 & \alpha^{-1}a_2 & \cdots & \alpha^{-1}a_n \\ \alpha b_1 & \alpha^{-1} b_2 & \cdots & \alpha^{-1}b_n \end{bmatrix}\ \mbox{where}\ \alpha := \frac{\Delta_{12} \Delta_{34}\cdots \Delta_{(n-1)n}}{\Delta_{23} \Delta_{45}\cdots \Delta_{n1}}\]
Since there are matrices on which $\alpha$ is not a root of unity, the twist $\rt$ has infinite order on $\Mato(2,n)$. Moreover, since the $GL_2$-invariant quantity $\Delta_{13}/\Delta_{12}$ scales by a factor of $\alpha^2$ each time $\rt^n$ is applied, and so $\rt$ is not periodic on $\Pio(\cMg)$ either. However, $\rt^n$ is trivial up to the action of $\GG_m^n$ on $\Pio(\cMg)$ by rescaling columns. 

For general $k$, the twist has order $2n/\gcd(k,n)$ on the quotient $\Pio(\cMg)/\GG_m^n\subset Gr(k,n)/\GG_m^n$ by rescaling columns.
\end{myexample}

\subsection{A twist of infinite order} \label{D4 example}

While Example \ref{example: Gr2n} provided a case where the twist has infinite order, that example was finite order modulo column rescaling. We provide a richer example of a twist with infinite order.

Consider the positroid variety $\Pi(\cM)$ in $Gr(4,8)$  cut out by the vanishing of Pl\"ucker coordinates
\[\Delta_{1234} = \Delta_{3456} = \Delta_{5678} = \Delta_{1268} = 0\]

\begin{figure}
\begin{tikzpicture}[scale=.8]
		\path[use as  bounding box] (-4.5,-4.5) rectangle (4.5,4.5);
		\draw[fill=black!10] (0,0) circle (4);
		\node[invisible] (6) at (157.5:4) {};
		\node[invisible] (7) at (112.5:4) {};
		\node[invisible] (8) at (67.5:4) {};
		\node[invisible] (1) at (22.5:4) {};
		\node[invisible] (2) at (-22.5:4) {};
		\node[invisible] (3) at (-67.5:4) {};
		\node[invisible] (4) at (-112:4) {};
		\node[invisible] (5) at (-157:4) {};
		
		\node[right] at (1) {$1$};
		\node[right] at (2) {$2$};
		\node[below] at (3) {$3$};
		\node[below] at (4) {$4$};
		\node[left] at (5) {$5$};
		\node[left] at (6) {$6$};
		\node[above] at (7) {$7$};
		\node[above] at (8) {$8$};
		
		\node[dot, fill=white] (a) at (-.75,2.75) {};
		\node[dot, fill=black!50] (b) at (.75, 2.75) {};
		\node[dot, fill=black!50] (c) at (-.75,1.5) {};
		\node[dot, fill=white] (d) at (.75,1.5) {};

		\node[dot, fill=black!50] (e) at (-2.75,.75) {};
		\node[dot, fill=white] (f) at (-1.5,.75) {};
		\node[dot, fill=white] (g) at (-2.75,-.75) {};
		\node[dot, fill=black!50] (h) at (-1.5,-.75) {};

		\node[dot, fill=black!50] (i) at (1.5,.75) {};
		\node[dot, fill=white] (j) at (2.75,.75) {};
		\node[dot, fill=white] (k) at (1.5,-.75) {};
		\node[dot, fill=black!50] (l) at (2.75,-.75) {};

		\node[dot, fill=white] (m) at (-.75,-1.5) {};
		\node[dot, fill=black!50] (n) at (.75, -1.5) {};
		\node[dot, fill=black!50] (o) at (-.75,-2.75) {};
		\node[dot, fill=white] (p) at (.75,-2.75) {};
		
		\draw (1) to (j);
		\draw (2) to (l);
		\draw (3) to (p);
		\draw (4) to (o);
		\draw (5) to (g);
		\draw (6) to (e);
		\draw (7) to (a);
		\draw (8) to (b);
		
		\draw (a) to (b);
		\draw (b) to (d);
		\draw (c) to (d);
		\draw (c) to (a);
		\draw (e) to (f);
		\draw (e) to (g);
		\draw (f) to (h);
		\draw (g) to (h);
		\draw (i) to (j);
		\draw (i) to (k);
		\draw (j) to (l);
		\draw (k) to (l);
		\draw (m) to (n);
		\draw (m) to (o);
		\draw (n) to (p);
		\draw (o) to (p);	
		\draw (d) to (i);
		\draw (k) to (n);
		\draw (m) to (h);
		\draw (f) to (c);

		\node[red!50!black] at (2,2) {$4678$};
		\node[red!50!black] at (3.35,0) {$1678$};
		\node[red!50!black] at (2,-2) {$1268$};
		\node[red!50!black] at (0,-3.35) {$1238$};
		\node[red!50!black] at (-2,-2) {$2348$};
		\node[red!50!black] at (-3.35,0) {$2345$};
		\node[red!50!black] at (-2,2) {$2456$};
		\node[red!50!black] at (0,3.35) {$4567$};
		
		\node[red!50!black] at (0,2.125) {$4568$};
		\node[red!50!black] at (2.125,0) {$2678$};
		\node[red!50!black] at (0,-2.125) {$1248$};
		\node[red!50!black] at (-2.125,0) {$2346$};
		\node[red!50!black] at (0,0) {$2468$};
\end{tikzpicture}
\caption{A graph with infinite order twist (source-labeled faces)} \label{D4Tilde}
\end{figure}
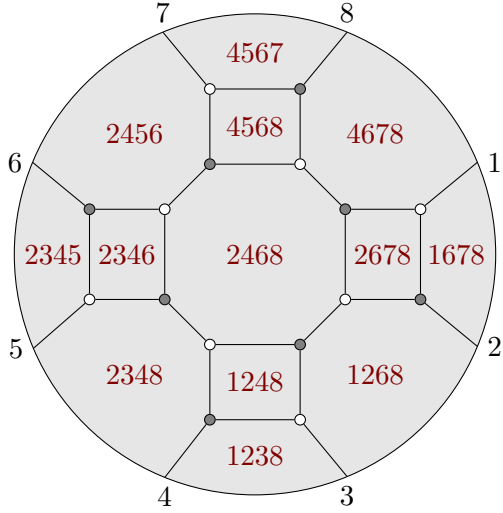

A reduced graph for $\cM$ is shown in Figure~\ref{D4Tilde}. 
Using Proposition \ref{dimer sum is twist}, the left twist of $\Delta_{4568}$ is given by a sum over the two matchings with boundary $4568$.
\begin{equation} \Delta_{4568}\circ \lt 
= \frac{1}{\Delta_{4568}} + \frac{\Delta_{4567}\Delta_{2468}}{\Delta_{4568}\Delta_{2456}\Delta_{4678}} \label{D4Recurrence1} 
\end{equation}

The left twists of the analogous coordinates $\Delta_{2678}$, $\Delta_{1248}$, and $\Delta_{2346}$ are given by similar binomials, obtained from this one by rotation of the graph by $\pi/2$. The right twist of the central coordinate $\Delta_{2468}$ is a sum over $17$ matchings with boundary $2468$.
\begin{equation}
\begin{array}{rl}
\Delta_{2468}\circ\lt 
&= \frac{1}{\Delta_{2468}} + \left[ \frac{\Delta_{2456}\Delta_{4678}\Delta_{1268}\Delta_{2348}}
{\Delta_{2468}\Delta_{4568}\Delta_{2678}\Delta_{1248}\Delta_{2346}}\right. \\
&\times \left.
\left(1+\frac{\Delta_{4567}\Delta_{2468}}{\Delta_{2456}\Delta_{4678}}\right)
\left(1+\frac{\Delta_{1678}\Delta_{2468}}{\Delta_{4678}\Delta_{1268}}\right)
\left(1+\frac{\Delta_{1238}\Delta_{2468}}{\Delta_{1268}\Delta_{2348}}\right)
\left(1+\frac{\Delta_{2345}\Delta_{2468}}{\Delta_{2348}\Delta_{2456}}\right)\right]
\end{array} \label{D4Recurrence2} \end{equation}

We will describe the $\lt$ orbit of the image under $\tDD$ of the identity element of $\mathbb{G}_m^E/\mathbb{G}_m^{V-1}$. This may be given as the row span of the following matrix.
\[\begin{bmatrix}
2 & 1 & 1 & 0 & -1 & 0 & 1 & 0 \\
-1 & 0 & 2 & 1 & 1 & 0 & -1 & 0 \\
1 & 0 & -1 & 0 & 2 & 1 & 1 & 0 \\
-1 & 0 & 1 & 0 & -1 & 0 & 2 & 1 \\
\end{bmatrix}\]
We list the values of the Pl\"ucker coordinates for the source-labelled faces under the first several twists, and then describe the general recursion.
\[ \begin{array}{|c|c|c|c|}
\hline
& 2468& 4568,\ 2678,\ 1248,\ 2346 & \mbox{boundary faces} \\
\hline
x & 1 & 1 & 1 \\
\lt(x) & 17 & 2 & 1 \\ 
\lt^2(x) & 386 & 9 & 1 \\ 
\lt^3(x) & 8857 & 43  & 1 \\ 
\lt^4(x) & 203321 & 206 & 1 \\
\hline
\end{array} \]
In general, if the $i$-th row is $(u_i, v_i, 1)$, Equations~\eqref{D4Recurrence1} and~\eqref{D4Recurrence2} give
\[ \left( u_{i+1}, v_{i+1}, 1 \right) = \left( \frac{v_{i+1}^4+1}{u_i},\ \frac{u_i+1}{v_i},\ 1\right) \]
An easy induction shows that $u_i$ and $v_i$ are given by the linear recursions.
\[ u_{i+1} - 23 u_i + u_{i-1} = -4 \qquad v_{i+1} - 5 v_i + v_{i-1} = 0. \]
It is easy to see from the linear recusrion that $v_i$ is increasing without bound, so the torus invariant quantity $(\Delta_{1248} \Delta_{2346} \Delta_{4568} \Delta_{2678})/(\Delta_{2348} \Delta_{2456} \Delta_{4678} \Delta_{1268}) = v_i^4$ is likewise increasing, and we have provided a direct computation that the twist is not periodic even up to column rescaling.

Note that the $u_i$ and $v_i$ had to be integers, because they are sums over matchings of Laurent monomials that evaluate to $1$.
This is a valuable check when performing computations by hand.

\begin{remark}
The mutable part of the quiver for this reduced graph is of type $\tilde{D}_4$, with edges oriented away from its central vertex.
From the above formulas, we may check that the twist is the same (up to torus action) as first mutating at all $4$ outer vertices of $\tilde{D}_4$, and then mutating at the center. 
This is the Coxeter transformation for this quiver, and (as $\tilde{D}_4$ is not of finite type) the Coxeter transformation is not of finite order, even up to torus symmetry. 
\end{remark}
%

\subsection{A reduced graph whose image is not given by nonvanishing of Pl\"ucker coordinates}

Consider the reduced graph in Figure~\ref{NonPluckerTwist}, with interior face labels ${124}$, ${346}$, ${256}$ and ${246}$. This graph is reduced, so $\DD: \GG_m^E/\GG_m^V \to Gr(3,6)$ is an open immersion. 
The complement of $\DD(\GG_m^E/\GG_m^V)$ is a degree $11$ hypersurface which factors as
\[ \Delta_{123} \Delta_{234} \Delta_{345} \Delta_{456} \Delta_{156} \Delta_{126} \Delta_{125} \Delta_{134} \Delta_{356}  X \]
where $X = \Delta_{124} \Delta_{356} - \Delta_{123} \Delta_{456}$. 
Up to column rescaling,  $X$ is the twist of $\Delta_{246}$. 
In particular, $X$ vanishes when the $2$-planes $\mathrm{Span}(v_1, v_2)$, $\mathrm{Span}(v_3, v_4)$ and $\mathrm{Span}(v_5, v_6)$ have a common intersection; 
this description of the non-Pl\"ucker cluster variable on $Gr(3,6)$ was observed by Scott \cite{Sco06}.

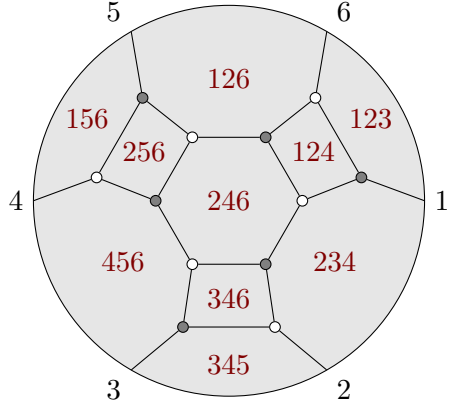
\begin{figure}
\begin{tikzpicture}[scale=.65]
		\path[use as  bounding box] (-4.5,-4.5) rectangle (4.5,4.5);
		\draw[fill=black!10] (0,0) circle (4);
		\node[invisible] (1) at (0:4) {};
		\node[invisible] (2) at (-60:4) {};
		\node[invisible] (3) at (-120:4) {};
		\node[invisible] (4) at (-180:4) {};
		\node[invisible] (5) at (-240:4) {};
		\node[invisible] (6) at (-300:4) {};
		
		\node[right] at (1) {$1$};
		\node[below right] at (2) {$2$};
		\node[below left] at (3) {$3$};
		\node[left] at (4) {$4$};
		\node[above left] at (5) {$5$};
		\node[above right] at (6) {$6$};
		
		\node[dot,fill=white] (a) at (0:1.5) {};
		\node[dot,fill=black!50] (b) at (-60:1.5) {};
		\node[dot,fill=white] (c) at (-120:1.5) {};
		\node[dot,fill=black!50] (d) at (-180:1.5) {};
		\node[dot,fill=white] (e) at (-240:1.5) {};
		\node[dot,fill=black!50] (f) at (-300:1.5) {};

		\node[dot,fill=black!50] (a') at (10:2.75) {};
		\node[dot,fill=white] (b') at (-70:2.75) {};
		\node[dot,fill=black!50] (c') at (-110:2.75) {};
		\node[dot,fill=white] (d') at (-190:2.75) {};
		\node[dot,fill=black!50] (e') at (-230:2.75) {};
		\node[dot,fill=white] (f') at (-310:2.75) {};
		
		\draw (f) to (a) to (a') to (1);
		\draw (a) to (b) to (b') to (2);
		\draw (b) to (c) to (c') to (3);
		\draw (c) to (d) to (d') to (4);
		\draw (d) to (e) to (e') to (5);
		\draw (e) to (f) to (f') to (6);
		
		\draw (a') to (f');
		\draw (c') to (b');
		\draw (e') to (d');
		
		\node[red!50!black] at (0,0) {$246$};
		
		\node[red!50!black] at (30:2) {$124$};
		\node[red!50!black] at (-90:2) {$346$};
		\node[red!50!black] at (-210:2) {$256$};
		
		\node[red!50!black] at (30:3.35) {$123$};
		\node[red!50!black] at (-90:3.35) {$345$};
		\node[red!50!black] at (-210:3.35) {$156$};

		\node[red!50!black] at (-30:2.5) {$234$};
		\node[red!50!black] at (-150:2.5) {$456$};
		\node[red!50!black] at (-270:2.5) {$126$};

\end{tikzpicture}
\caption{A graph for which the twist uses non-Pl\"ucker cluster variables} \label{NonPluckerTwist}
\end{figure}

\subsection{Double Bruhat cells and the Chamber Ansatz}\label{app: DBC}

Consider a reduced word $\mathbf{s}=s_{i_1}s_{i_2}...s_{i_\ell}$ for an element $w$ in the symmetric group $S_n$. Construct a reduced graph $G_\mathbf{s}$ as follows (an example is given in Figure \ref{fig: ansatzexample}).
\begin{itemize}
	\item Start with a rectangle. Add vertices numbered $1,2,...,n$ down the right side, and $n+1,n+2,...,2n$ up the left side.
	\item Connect each $i$ on the right to $2n-i+1$ on the left with a horizontal line.
	\item Reading left to right, for each $s_i$ in the reduced word $\mathbf{s}$, add a vertical edge between the line containing $i$ and the line containing $i+1$. Color the top vertex white and the bottom vertex black.
	\item Add $2$-valent white vertices to the edges so that the resulting graph is bipartite and every boundary vertex is adjacent to a white vertex.
\end{itemize}
\begin{figure}[h!t]
\begin{tikzpicture}[scale=.5]
	\draw[rounded corners,fill=black!10] (-5,-3) rectangle (5,3);
	\node[dot,fill=white] (a) at (-4,-2) {};
	\node[dot,fill=black!50] (b) at (-2,-2) {};
	\node[dot,fill=white] (c) at (0,-2) {};
	\node[dot,fill=black!50] (d) at (2,-2) {};
	\node[dot,fill=white] (e) at (4,-2) {};
	\node[dot,fill=white] (f) at (-2,0) {};
	\node[dot,fill=black!50] (g) at (0,0) {};
	\node[dot,fill=white] (h) at (2,0) {};
	\node[dot,fill=white] (i) at (0,2) {};
	\draw (-5,-2) to (a) to node[inner sep=0mm] (X') {} (b) to (c) to (d) to node[inner sep=0mm] (Y') {} (e) to (5,-2);
	\draw (-5,0) to (f) to (g) to (h) to (5,0);
	\draw (-5,2) to (i) to (5,2);
	\draw (b) to (f);
	\draw (d) to (h);
	\draw (g) to (i);
	\node[left] at (-5,-2) {\scriptsize $4$};
	\node[left] at (-5,0) {\scriptsize $5$};
	\node[left] at (-5,2) {\scriptsize $6$};
	\node[right] at (5,2) {\scriptsize $1$};
	\node[right] at (5,0) {\scriptsize $2$};
	\node[right] at (5,-2) {\scriptsize $3$};	
\end{tikzpicture}
\caption{The graph $G_\mathbf{s}$ associated to $\mathbf{s} =s_2s_1s_2$}
\label{fig: ansatzexample}
\end{figure}
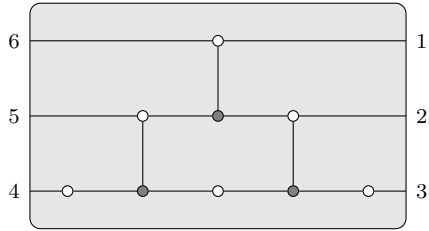
\begin{remark}
The reduced graph $G_\mathbf{s}$ is constructed so that the associated Postnikov diagram is the \emph{pseudoline arrangement} for $\mathbf{s}$ or, equivalently, the \emph{double wiring diagram} for $(\mathbf{s},\mathbf{e})$ \cite{BFZ96,FZ99}.
\end{remark}

Let $\overline{w_0}$ be the antidiagonal $n\times n$ matrix with $1$s in odd columns and $-1$s in even columns. Then the open inclusion
\[ Gl(n,\mathbb{C}) \hookrightarrow Gr(n,2n),\;\;\; A \mapsto \text{rowspan}\left(\begin{bmatrix} A & \overline{w_0} \end{bmatrix}\right)\]
induces an isomorphism from the double Bruhat cell $Gl^{w,e}:= B_+\cap (B_-wB_-)$ to the positroid variety $\Pio(\cM)$ associated to $G_\mathsf{s}$ \cite[Section 6]{KLS13}.

%

\begin{myexample}
Let $\mathbf{s}= s_2s_1s_2$. The associated positroid $\cM$ contains all $3$-element subsets of $\{1,2,3,4,5,6\}$ except those which contain $\{1,6\}$ and those contained in $\{1,2,5,6\}$. The open positroid variety $\Pio(\cM)$ in $Gr(3,6)$ can be parametrized as the row span of matrices of the form
\begin{equation}\label{eq: unipotent}
\begin{bmatrix}
	a & b & c & 0 & 0 & 1 \\
	0 & d & e & 0 & -1 & 0 \\
	0 & 0 & f & 1 & 0 & 0 
\end{bmatrix}
\end{equation}
such that $acdf(be-cd)\neq0$. This open condition is equivalent to requiring that the matrix
\[ \begin{bmatrix}
	a & b & c \\
	0 & d & e \\
	0 & 0 & f
\end{bmatrix}\in B_+\]
is an element of $B_-wB_-$.
\end{myexample}

Using this isomorphism, the boundary measurement map $\DD$
\[ \mathbb{G}_m^E/\mathbb{G}_m^V \stackrel{\DD}{\longrightarrow} \Pio(\cM)\]
is equivalent to an open inclusion of $\mathbb{G}_M^E/\mathbb{G}_m^V$ into the double Bruhat cell $Gl^{w,e}$.\footnote{This example is in the Grassmannian, not the Pl\"ucker cone, and so we use the quotient version of Theorem \ref{main theorem}.}

The domain of the boundary measurement map $\DD$ may also be simplified. Let $E'\subset E$ be the set of edges in $G$ which are either vertical or adjacent to the right boundary. It is a simple exercise to show that the action of the gauge group may be used to set the weight of every edge not in $E$ to $1$, yielding an isomorphism $\mathbb{G}_m^{E}/\mathbb{G}_m^V \garrow{\sim} \mathbb{G}_m^{E'}$.

The resulting incarnation of the boundary measurement map
\[ \mathbb{G}_m^{E'}\garrow{\DD} Gl^{w,e} \]
may be characterized in terms of matrix multiplication. Explicitly, let $d_1,d_2,...d_n$ be non-zero weights on the edges adjacent to the right boundary, and let $t_1,t_2,...,t_\ell$ be non-zero weights on the vertical edges in $E$ (all other weights are $1$). Then the image under $\DD$ is the product
\[ E_{i_1}(t_1)E_{i_2}(t_2)\cdots E_{i_\ell}(t_\ell)D(d_1,d_2,...,d_n) \]
where $D(d_1,d_2,...,d_n)$ is the diagonal matrix with the given entries, and $E_i(t)$ is the matrix with $1$s on the diagonal, $t$ in the $(i+1,i)$-entry, and $0$s elsewhere.

\begin{myexample}
Let $\mathbf{s}= s_2s_1s_2$.
Any set of non-zero edge weights on $G_\mathbf{s}$ is uniquely gauge equivalent to a set of edge weights of following form
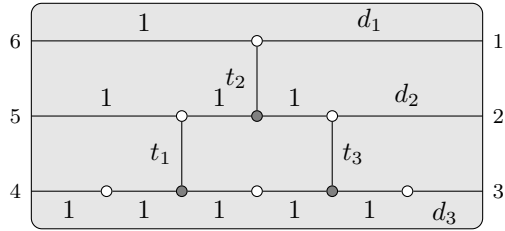
\[
\begin{tikzpicture}[scale=.5]
	\draw[rounded corners,fill=black!10] (-6,-3) rectangle (6,3);
	\node[dot,fill=white] (a) at (-4,-2) {};
	\node[dot,fill=black!50] (b) at (-2,-2) {};
	\node[dot,fill=white] (c) at (0,-2) {};
	\node[dot,fill=black!50] (d) at (2,-2) {};
	\node[dot,fill=white] (e) at (4,-2) {};
	\node[dot,fill=white] (f) at (-2,0) {};
	\node[dot,fill=black!50] (g) at (0,0) {};
	\node[dot,fill=white] (h) at (2,0) {};
	\node[dot,fill=white] (i) at (0,2) {};
	\draw (-6,-2) to (a) to node[inner sep=0mm] (X') {} (b) to (c) to (d) to node[inner sep=0mm] (Y') {} (e) to (6,-2);
	\draw (-6,0) to (f) to (g) to (h) to (6,0);
	\draw (-6,2) to (i) to (6,2);
	\draw (b) to (f);
	\draw (d) to (h);
	\draw (g) to (i);
	\node[left] at (-6,-2) {\scriptsize $4$};
	\node[left] at (-6,0) {\scriptsize $5$};
	\node[left] at (-6,2) {\scriptsize $6$};
	\node[right] at (6,2) {\scriptsize $1$};
	\node[right] at (6,0) {\scriptsize $2$};
	\node[right] at (6,-2) {\scriptsize $3$};	
	\node[above] at (-3,2) {\small $1$};
	\node[above] at (3,2) {\small $d_1$};
	\node[above] at (-4,0) {\small $1$};
	\node[above] at (-1,0) {\small $1$};
	\node[above] at (1,0) {\small $1$};
	\node[above] at (4,0) {\small $d_2$};
	\node[below] at (-5,-2) {\small $1$};
	\node[below] at (-3,-2) {\small $1$};
	\node[below] at (-1,-2) {\small $1$};
	\node[below] at (1,-2) {\small $1$};
	\node[below] at (3,-2) {\small $1$};
	\node[below] at (5,-2) {\small $d_3$};
	\node[left] at (0,1) {\small $t_2$};
	\node[left] at (-2,-1) {\small $t_1$};
	\node[right] at (2,-1) {\small $t_3$};
\end{tikzpicture}
\]

The boundary measurement map sends these edge weights to the row span of the matrix
\begin{equation*}
\begin{bmatrix}
	d_1 & d_2t_2 & d_3t_2t_3 & 0 & 0 & 1 \\
	0 & d_2 & d_3(t_1+t_3) & 0 & -1 & 0 \\
	0 & 0 & d_3 & 1 & 0 & 0 
\end{bmatrix}
\end{equation*}
The left half of this matrix arises as the product of elementary matrices below.
\begin{align*}
\begin{bmatrix}
	d_1 & d_2t_2 & d_3t_2t_3 \\
	0 & d_2 & d_3(t_1+t_3) \\
	0 & 0 & d_3
\end{bmatrix}
&= 
\begin{bmatrix}
	1 & 0 & 0 \\
	0 & 1 & t_1 \\
	0 & 0 & 1
\end{bmatrix}
\begin{bmatrix}
	1 & t_2 & 0 \\
	0 & 1 & 0 \\
	0 & 0 & 1
\end{bmatrix}
\begin{bmatrix}
	1 & 0 & 0 \\
	0 & 1 & t_3 \\
	0 & 0 & 1
\end{bmatrix}
\begin{bmatrix}
	d_1 & 0 & 0 \\
	0 & d_2 & 0 \\
	0 & 0 & d_3
\end{bmatrix} \\
&= E_2(t_1)E_1(t_2)E_2(t_3)D(d_1,d_2,d_3) \qedhere
\end{align*}
\end{myexample}

The problem of inverting the boundary measurement map $\DD$ for $G_\mathsf{s}$ is then equivalent to the problem of expressing a matrix in $Gl^{w,e}$ as a product of elementary matrices indexed by $\mathbf{s}$. This is a classical problem, whose solution in \cite{BFZ96,FZ99} (dubbed the \emph{Chamber Ansatz}) was an important precursor to both cluster algebras and Postnikov's diagrams. 

Proposition \ref{inversion} provides an explicit inverse to $\DD$, as the composition $\lpartial \circ \sF \circ \rt$. This composition directly generalizes the Chamber Ansatz, in that the computation exactly replicates the formulas given in \cite{BFZ96}. A key component in this assertion is that our right twist automorphism $\rt$ of $\Pio(\cM)$ induces the \emph{BFZ twist} automorphism of $GL^{w,e}$, as defined in \cite[Section 1.5]{FZ99}.

\begin{myexample}
We continue the running example of $\mathbf{s}=s_2s_1s_2$, and compute the action of $\lpartial \circ \sF \circ \rt$ on the matrix in \eqref{eq: unipotent}. This computation is given in Figure \ref{fig: ansatz}.

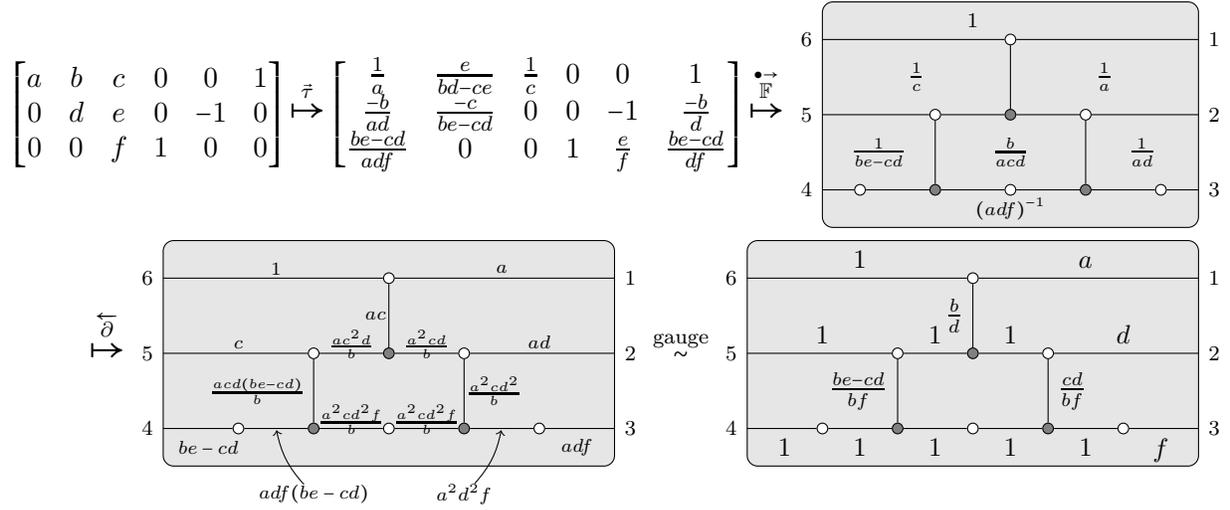
\begin{figure}[h!t]
\begin{align*}
\begin{bmatrix}
	a & b & c & 0 & 0 & 1 \\
	0 & d & e & 0 & -1 & 0 \\
	0 & 0 & f & 1 & 0 & 0 
\end{bmatrix}
\stackrel{\rt}{\scalebox{1.5}{$\mapsto$}} 
\begin{bmatrix}
	\frac{1}{a} & \frac{e}{bd-ce} & \frac{1}{c} & 0 & 0 & 1 \\
	\frac{-b}{ad} & \frac{-c}{be-cd} & 0 & 0 & -1 & \frac{-b}{d} \\
	\frac{be-cd}{adf} & 0 & 0 & 1 & \frac{e}{f} & \frac{be-cd}{df} 
\end{bmatrix}
\stackrel{\sF}{\scalebox{1.5}{$\mapsto$}} 
\begin{tikzpicture}[scale=.5,baseline=-.1cm]
	\draw[rounded corners,fill=black!10] (-5,-3) rectangle (5,3);
	\node[dot,fill=white] (a) at (-4,-2) {};
	\node[dot,fill=black!50] (b) at (-2,-2) {};
	\node[dot,fill=white] (c) at (0,-2) {};
	\node[dot,fill=black!50] (d) at (2,-2) {};
	\node[dot,fill=white] (e) at (4,-2) {};
	\node[dot,fill=white] (f) at (-2,0) {};
	\node[dot,fill=black!50] (g) at (0,0) {};
	\node[dot,fill=white] (h) at (2,0) {};
	\node[dot,fill=white] (i) at (0,2) {};
	\draw (-5,-2) to (a) to (b) to (c) to (d) to (e) to (5,-2);
	\draw (-5,0) to (f) to (g) to (h) to (5,0);
	\draw (-5,2) to (i) to (5,2);
	\draw (b) to (f);
	\draw (d) to (h);
	\draw (g) to (i);
	\node[left] at (-5,-2) {\scriptsize $4$};
	\node[left] at (-5,0) {\scriptsize $5$};
	\node[left] at (-5,2) {\scriptsize $6$};
	\node[right] at (5,2) {\scriptsize $1$};
	\node[right] at (5,0) {\scriptsize $2$};
	\node[right] at (5,-2) {\scriptsize $3$};	
	\node at (-1,2.5) {\scriptsize $1$};
	\node at (-2.5,1) {\scriptsize $\frac{1}{c}$};
	\node at (2.5,1) {\scriptsize $\frac{1}{a}$};
	\node at (-3.5,-1) {\scriptsize $\frac{1}{be-cd}$};
	\node at (0,-1) {\scriptsize $\frac{b}{acd}$};
	\node at (3.5,-1) {\scriptsize $\frac{1}{ad}$};
	\node at (0,-2.5) {\tiny $(adf)^{-1}$};
\end{tikzpicture}&
\\
\stackrel{\lpartial}{\scalebox{1.5}{$\mapsto$}} 
\begin{tikzpicture}[scale=.5,baseline=-.1cm]
	\draw[rounded corners,fill=black!10] (-6,-3) rectangle (6,3);
	\node[dot,fill=white] (a) at (-4,-2) {};
	\node[dot,fill=black!50] (b) at (-2,-2) {};
	\node[dot,fill=white] (c) at (0,-2) {};
	\node[dot,fill=black!50] (d) at (2,-2) {};
	\node[dot,fill=white] (e) at (4,-2) {};
	\node[dot,fill=white] (f) at (-2,0) {};
	\node[dot,fill=black!50] (g) at (0,0) {};
	\node[dot,fill=white] (h) at (2,0) {};
	\node[dot,fill=white] (i) at (0,2) {};
	\draw (-6,-2) to (a) to node[inner sep=0mm] (X') {} (b) to (c) to (d) to node[inner sep=0mm] (Y') {} (e) to (6,-2);
	\draw (-6,0) to (f) to (g) to (h) to (6,0);
	\draw (-6,2) to (i) to (6,2);
	\draw (b) to (f);
	\draw (d) to (h);
	\draw (g) to (i);
	\node[left] at (-6,-2) {\scriptsize $4$};
	\node[left] at (-6,0) {\scriptsize $5$};
	\node[left] at (-6,2) {\scriptsize $6$};
	\node[right] at (6,2) {\scriptsize $1$};
	\node[right] at (6,0) {\scriptsize $2$};
	\node[right] at (6,-2) {\scriptsize $3$};	
	\node[] at (-3,2.25) {\tiny $1$};
	\node[] at (3,2.25) {\tiny $a$};
	\node[] at (-4,.25) {\tiny $c$};
	\node[] at (-1,.25) {\tiny $\frac{ac^2d}{b}$};
	\node[] at (1,.25) {\tiny $\frac{a^2cd}{b}$};
	\node[] at (4,.3) {\tiny $ad$};
	\node[] at (-4.8,-2.5) {\tiny $be-cd$};
	\node[inner sep=0mm] (X) at (-2,-3.75) {\tiny $adf(be-cd)$};
	\node[] at (-1,-1.75) {\tiny $\frac{a^2cd^2f}{b}$};
	\node[] at (1,-1.75) {\tiny $\frac{a^2cd^2f}{b}$};
	\node[inner sep=0mm] (Y) at (2,-3.75) {\tiny $a^2d^2f$};
	\node[] at (5,-2.5) {\tiny $adf$};
	\node[] at (-.35,1) {\tiny $ac$};
	\node[] at (-3.5,-1) {\tiny $\frac{acd(be-cd)}{b}$};
	\node[] at (2.8,-1) {\tiny $\frac{a^2cd^2}{b}$};
	\draw[->,relative,out=15,in=165] (X) to (X');
	\draw[->,relative,out=-15,in=195] (Y) to (Y');
\end{tikzpicture}
\stackrel{\text{gauge}}{\sim} 
\begin{tikzpicture}[scale=.5,baseline=-.1cm]
	\draw[rounded corners,fill=black!10] (-6,-3) rectangle (6,3);
	\node[dot,fill=white] (a) at (-4,-2) {};
	\node[dot,fill=black!50] (b) at (-2,-2) {};
	\node[dot,fill=white] (c) at (0,-2) {};
	\node[dot,fill=black!50] (d) at (2,-2) {};
	\node[dot,fill=white] (e) at (4,-2) {};
	\node[dot,fill=white] (f) at (-2,0) {};
	\node[dot,fill=black!50] (g) at (0,0) {};
	\node[dot,fill=white] (h) at (2,0) {};
	\node[dot,fill=white] (i) at (0,2) {};
	\draw (-6,-2) to (a) to node[inner sep=0mm] (X') {} (b) to (c) to (d) to node[inner sep=0mm] (Y') {} (e) to (6,-2);
	\draw (-6,0) to (f) to (g) to (h) to (6,0);
	\draw (-6,2) to (i) to (6,2);
	\draw (b) to (f);
	\draw (d) to (h);
	\draw (g) to (i);
	\node[left] at (-6,-2) {\scriptsize $4$};
	\node[left] at (-6,0) {\scriptsize $5$};
	\node[left] at (-6,2) {\scriptsize $6$};
	\node[right] at (6,2) {\scriptsize $1$};
	\node[right] at (6,0) {\scriptsize $2$};
	\node[right] at (6,-2) {\scriptsize $3$};	
	\node[above] at (-3,2) {\small $1$};
	\node[above] at (3,2) {\small $a$};
	\node[above] at (-4,0) {\small $1$};
	\node[above] at (-1,0) {\small $1$};
	\node[above] at (1,0) {\small $1$};
	\node[above] at (4,0) {\small $d$};
	\node[below] at (-5,-2) {\small $1$};
	\node[below] at (-3,-2) {\small $1$};
	\node[below] at (-1,-2) {\small $1$};
	\node[below] at (1,-2) {\small $1$};
	\node[below] at (3,-2) {\small $1$};
	\node[below] at (5,-2) {\small $f$};
	\node[left] at (0,1) {\small $\frac{b}{d}$};
	\node[left] at (-2,-1) {\small $\frac{be-cd}{bf}$};
	\node[right] at (2,-1) {\small $\frac{cd}{bf}$};
\end{tikzpicture}&
\end{align*}
\caption{Explicitly inverting the boundary measurement map} \label{fig: ansatz}
\end{figure}

In the last step, gauge transformation has been used to normalize the weight of each edge not in $E'$ to $1$.
Since the result is the preimage of \eqref{eq: unipotent} under $\DD$, we have the following matrix identity.
\[ \footnotesize \begin{bmatrix}
a & b & c \\
0 & d & e \\
0 & 0 & f
\end{bmatrix}=
\begin{bmatrix}
1 & 0 & 0 \\
0 & 1 & \frac{be-cd}{bf} \\
0 & 0 & 1
\end{bmatrix}\begin{bmatrix}
1 & \frac{b}{d} & 0 \\
0 & 1 & 0 \\
0 & 0 & 1
\end{bmatrix}\begin{bmatrix}
1 & 0 & 0 \\
0 & 1 & \frac{cd}{bf} \\
0 & 0 & 1
\end{bmatrix}\begin{bmatrix}
a & 0 & 0 \\
0 & d & 0 \\
0 & 0 & f
\end{bmatrix}\qedhere \]
\end{myexample}

\begin{remark}
For a general double Bruhat cell $Gl^{w,v}$, one would choose a \emph{double reduced word} $\mathbf{s}$ for $(w,v)$ (see Section 1.2 in \cite{FZ99}). The construction of $G_{\mathbf{s}}$ is almost the same, except simple transpositions for $v$ determine vertical edges with black top vertex and white bottom vertex. The boundary measurement map is then equivalent to a product of $D$, $E_i$s and $F_i$s, where $F_i(t)$ is the elementary matrix with $t$ in the $(i,i+1)$-entry. Consequently, inverting the boundary measurement map recovers the \emph{formulas for factorization parameters} in \cite[Theorem 4.9]{FZ99}.
\end{remark}

\subsection{Three flags and plane partitions}

In this section, we will discuss a positroid of rank $m$ on $3m$ elements, which we will name $(u_1, u_2, \ldots, u_m, v_1, \ldots, v_m, w_1, \ldots, w_m)$.
The affine permutation on $\{ 1,2, \ldots, 3m \}$ is
\[ f(i) = \begin{cases} 2m+1-i & 1 \leq i \leq m \\ 4m+1-i & m+1 \leq i \leq 2m \\ 6m+1-i & 2m+1 \leq i \leq 3m \\ \end{cases}. \]
The defining rank conditions are
\[ \begin{array}{lcl}
  \mathrm{rank}( w_{m-r+1}, \ldots, w_{m-1}, w_{m} , u_1, u_2, \ldots, u_r) &=& r\\
 \mathrm{rank}( u_{m-r+1}, \ldots, u_{m-1}, u_{m} , v_1, v_2, \ldots, v_r) &=& r \\
 \mathrm{rank}( v_{m-r+1}, \ldots, v_{m-1}, v_{m} , w_1, w_2, \ldots, w_r) &=& r\\
 \end{array} \]
 and the consequences of these conditions.
 
 Let 
\[ \begin{array}{lclcl}
A_r &=&  \mathrm{Span}( w_{m-r+1}, \ldots, w_{m-1}, w_{m}) &=& \mathrm{Span}( u_1, u_2, \ldots, u_r)\\
B_r &=&  \mathrm{Span}( u_{m-r+1}, \ldots, u_{m-1}, u_{m}) &=& \mathrm{Span}( v_1, v_2, \ldots, v_r)\\
C_r &=&  \mathrm{Span}( v_{m-r+1}, \ldots, v_{m-1}, v_{m}) &=& \mathrm{Span}( w_1, w_2, \ldots, w_r) \\
 \end{array} \]
 So $A_{\bullet}$, $B_{\bullet}$ and $C_{\bullet}$ are three transverse complete flags in $m$-spaces. Conversely, any three transverse flags $A_{\bullet}$, $B_{\bullet}$ and $C_{\bullet}$ in $m$-space can be realized in this way, and uniquely so up to rescaling the $u_i$, $v_i$ and $w_i$. For example, $u_i$ can be recovered up to scaling by the formula $\mathrm{Span}(u_i) = A_i \cap B_{m-i+1}$.
So, for this positroid, $\Pio(\cM)/\GG_m^{3m}$ is the space of three transverse flags in $m$-space, up to symmetries of $m$-dimensional space.
This is the generalized Teichm\"uller space for $GL_n$ local systems on a disc with three marked boundary points~\cite{FG06}.

Let $u'_i$, $v'_i$ and $w'_i$ be the vectors of the twist and let $A'_{\bullet}$, $B'_{\bullet}$ and $C'_{\bullet}$ be the corresponding flags.
By the definition of the twist, $u'_i$ is perpendicular to $\mathrm{Span}(u_{i+1}, u_{i+2}, \ldots, u_m, v_{m-i+2}, \ldots, v_{m-1}, v_{m} ) = B_{m-i} + C_{i-1}$.
We compute that $A'_i = \mathrm{Span}(u_1, u_2, \ldots, u_i)$ is the orthogonal complement of $B_i$. So $A'_{\bullet}$ is the flag $B_{\bullet}^{\perp}$, whose $i$-th subspace is orthogonal to the $(m-i)$-th subspace of the flag $B_{\bullet}$. 
Continuing in this manner, $(A'_{\bullet}, B'_{\bullet}, C'_{\bullet}) = (B^{\perp}_{\bullet}, C^{\perp}_{\bullet}, A^{\perp}_{\bullet})$.

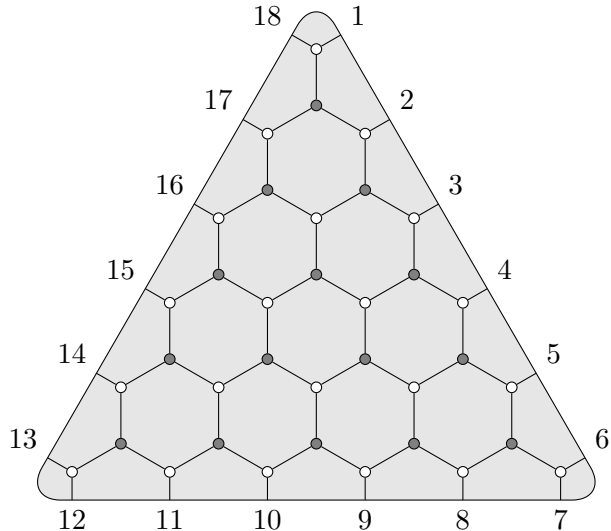
\begin{figure}[h!t]
\begin{tikzpicture}[scale=.75]
	\node[coordinate] (a) at \hexcoor{0}{3}{3} {};
	\node[coordinate] (b) at \hexcoor{0}{6}{0} {};
	\node[coordinate] (c) at \hexcoor{6}{0}{0} {};
	\node[coordinate] (d) at \hexcoor{0}{0}{6} {};
	\draw[rounded corners=5mm,fill=black!10] (a) to (b) to (c) to (d) to (a);
	
	\node[coordinate] (1) at \hexcoor{5.5}{.5}{0} {};
	\node[coordinate] (2) at \hexcoor{4.5}{1.5}{0} {};
	\node[coordinate] (3) at \hexcoor{3.5}{2.5}{0} {};
	\node[coordinate] (4) at \hexcoor{2.5}{3.5}{0} {};
	\node[coordinate] (5) at \hexcoor{1.5}{4.5}{0} {};
	\node[coordinate] (6) at \hexcoor{.5}{5.5}{0} {};
	\node[coordinate] (7) at \hexcoor{0}{5.5}{.5} {};
	\node[coordinate] (8) at \hexcoor{0}{4.5}{1.5} {};
	\node[coordinate] (9) at \hexcoor{0}{3.5}{2.5} {};
	\node[coordinate] (10) at \hexcoor{0}{2.5}{3.5} {};
	\node[coordinate] (11) at \hexcoor{0}{1.5}{4.5} {};
	\node[coordinate] (12) at \hexcoor{0}{0.5}{5.5} {};
	\node[coordinate] (13) at \hexcoor{0.5}{0}{5.5} {};
	\node[coordinate] (14) at \hexcoor{1.5}{0}{4.5} {};
	\node[coordinate] (15) at \hexcoor{2.5}{0}{3.5} {};
	\node[coordinate] (16) at \hexcoor{3.5}{0}{2.5} {};
	\node[coordinate] (17) at \hexcoor{4.5}{0}{1.5} {};
	\node[coordinate] (18) at \hexcoor{5.5}{0}{0.5} {};

	\node[dot,fill=white] (500) at \hexcoor{5}{0}{0} {};
	\node[dot,fill=white] (410) at \hexcoor{4}{1}{0} {};
	\node[dot,fill=white] (401) at \hexcoor{4}{0}{1} {};
	\node[dot,fill=white] (320) at \hexcoor{3}{2}{0} {};
	\node[dot,fill=white] (311) at \hexcoor{3}{1}{1} {};
	\node[dot,fill=white] (302) at \hexcoor{3}{0}{2} {};
	\node[dot,fill=white] (230) at \hexcoor{2}{3}{0} {};
	\node[dot,fill=white] (221) at \hexcoor{2}{2}{1} {};
	\node[dot,fill=white] (212) at \hexcoor{2}{1}{2} {};
	\node[dot,fill=white] (203) at \hexcoor{2}{0}{3} {};
	\node[dot,fill=white] (140) at \hexcoor{1}{4}{0} {};
	\node[dot,fill=white] (131) at \hexcoor{1}{3}{1} {};
	\node[dot,fill=white] (122) at \hexcoor{1}{2}{2} {};
	\node[dot,fill=white] (113) at \hexcoor{1}{1}{3} {};
	\node[dot,fill=white] (104) at \hexcoor{1}{0}{4} {};
	\node[dot,fill=white] (050) at \hexcoor{0}{5}{0} {};
	\node[dot,fill=white] (041) at \hexcoor{0}{4}{1} {};
	\node[dot,fill=white] (032) at \hexcoor{0}{3}{2} {};
	\node[dot,fill=white] (023) at \hexcoor{0}{2}{3} {};
	\node[dot,fill=white] (014) at \hexcoor{0}{1}{4} {};
	\node[dot,fill=white] (005) at \hexcoor{0}{0}{5} {};
	\node[dot,fill=black!50] (400) at \hexcoor{4}{0}{0} {};
	\node[dot,fill=black!50] (310) at \hexcoor{3}{1}{0} {};
	\node[dot,fill=black!50] (301) at \hexcoor{3}{0}{1} {};
	\node[dot,fill=black!50] (220) at \hexcoor{2}{2}{0} {};
	\node[dot,fill=black!50] (211) at \hexcoor{2}{1}{1} {};
	\node[dot,fill=black!50] (202) at \hexcoor{2}{0}{2} {};
	\node[dot,fill=black!50] (130) at \hexcoor{1}{3}{0} {};
	\node[dot,fill=black!50] (121) at \hexcoor{1}{2}{1} {};
	\node[dot,fill=black!50] (112) at \hexcoor{1}{1}{2} {};
	\node[dot,fill=black!50] (103) at \hexcoor{1}{0}{3} {};
	\node[dot,fill=black!50] (040) at \hexcoor{0}{4}{0} {};
	\node[dot,fill=black!50] (031) at \hexcoor{0}{3}{1} {};
	\node[dot,fill=black!50] (022) at \hexcoor{0}{2}{2} {};
	\node[dot,fill=black!50] (013) at \hexcoor{0}{1}{3} {};
	\node[dot,fill=black!50] (004) at \hexcoor{0}{0}{4} {};
		
	\draw (1) to (500) to (400) to (410);
	\draw (2) to (410) to (310) to (320);
	\draw (3) to (320) to (220) to (230);
	\draw (4) to (230) to (130) to (140);
	\draw (5) to (140) to (040) to (050) to (6);
	\draw (400) to (401) to (301) to (311);
	\draw (310) to (311) to (211) to (221);
	\draw (220) to (221) to (121) to (131);
	\draw (130) to (131) to (031) to (041) to (040);
	\draw (301) to (302) to (202) to (212);
	\draw (211) to (212) to (112) to (122);
	\draw (121) to (122) to (022) to (032) to (031);
	\draw (202) to (203) to (103) to (113);
	\draw (112) to (113) to (013) to (023) to (022);
	\draw (103) to (104) to (004) to (014) to (013);
	\draw (004) to (005) to (13);
	\draw (7) to (050);
	\draw (8) to (041);
	\draw (9) to (032);
	\draw (10) to (023);
	\draw (11) to (014);
	\draw (12) to (005);
	\draw (14) to (104);
	\draw (15) to (203);
	\draw (16) to (302);
	\draw (17) to (401);
	\draw (18) to (500);
	
	\node[above right] at (1) {$1$};
	\node[above right] at (2) {$2$};
	\node[above right] at (3) {$3$};
	\node[above right] at (4) {$4$};
	\node[above right] at (5) {$5$};
	\node[above right] at (6) {$6$};
	\node[below] at (7) {$7$};
	\node[below] at (8) {$8$};
	\node[below] at (9) {$9$};
	\node[below] at (10) {$10$};
	\node[below] at (11) {$11$};
	\node[below] at (12) {$12$};
	\node[above left] at (13) {$13$};
	\node[above left] at (14) {$14$};
	\node[above left] at (15) {$15$};
	\node[above left] at (16) {$16$};
	\node[above left] at (17) {$17$};
	\node[above left] at (18) {$18$};
\end{tikzpicture}
\caption{The reduced graph for three transverse $GL_6$ flags}
\label{TriangleGraph}
\end{figure}

There is a unique\footnote{By unique here, we allow isotopies and the first two types of moves, but not the third (urban renewal).} reduced graph for this positroid, shown in Figure~\ref{TriangleGraph}.
The face labels are indexed by $(a,b,c) \in \ZZ_{\geq 0}^3$ with $a+b+c=m$, and are 
\[ q_{abc} := \Delta_{12\cdots a\ (m+1) (m+2) \cdots (m+b) \ (2m+1) (2m+2) \cdots (2m+c)} \]
Monomials in the $q_{abc}$ which are invariant under rescaling the vectors $u$, $v$ and $w$ form coordinates on the moduli space of triples of transverse flags.
Let $q'_{abc}$ be the corresponding functions for the twisted vectors.
So Proposition~\ref{dimer sum is twist} writes the $q'_{abc}$ as Laurent polynomials in the $q_{abc}$, where we sum over dimer configurations on the graph in Figure~\ref{TriangleGraph}.
Once we eliminate forced edges from these graphs, we see that matchings with the given boundary are in bijection with rhombus tilings of a hexagon with side length $(a,b,c,a,b,c)$ (see Figure~\ref{Hexagon}), which are in turn in bijection to plane partitions in an $a \times b \times c$ box. Hence, the twisted coordinate $q'_{abc}$ is given by a sum over plane partitions of a box, one of the most classically studied questions in enumerative combinatorics, beginning with Major MacMahon in 1916.


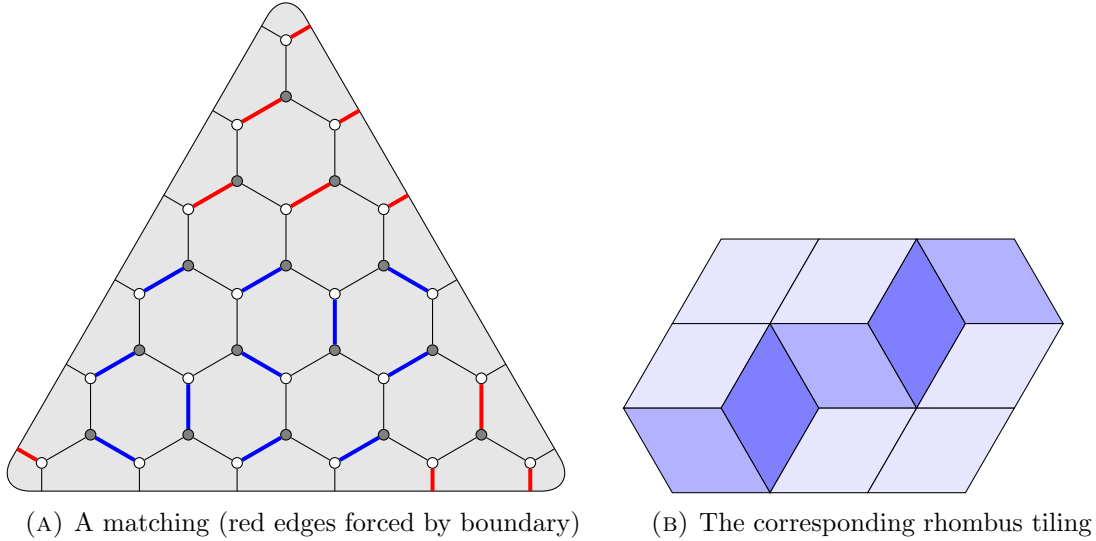
\begin{figure}[h!t]
\begin{subfigure}[b]{.5\textwidth}
\begin{tikzpicture}[scale=.75,baseline=(current bounding box.center)]
	\node[coordinate] (a) at \hexcoor{0}{3}{3} {};
	\node[coordinate] (b) at \hexcoor{0}{6}{0} {};
	\node[coordinate] (c) at \hexcoor{6}{0}{0} {};
	\node[coordinate] (d) at \hexcoor{0}{0}{6} {};
	\draw[rounded corners=5mm,fill=black!10] (a) to (b) to (c) to (d) to (a);

	\node[coordinate] (1) at \hexcoor{5.5}{.5}{0} {};
	\node[coordinate] (2) at \hexcoor{4.5}{1.5}{0} {};
	\node[coordinate] (3) at \hexcoor{3.5}{2.5}{0} {};
	\node[coordinate] (4) at \hexcoor{2.5}{3.5}{0} {};
	\node[coordinate] (5) at \hexcoor{1.5}{4.5}{0} {};
	\node[coordinate] (6) at \hexcoor{.5}{5.5}{0} {};
	\node[coordinate] (7) at \hexcoor{0}{5.5}{.5} {};
	\node[coordinate] (8) at \hexcoor{0}{4.5}{1.5} {};
	\node[coordinate] (9) at \hexcoor{0}{3.5}{2.5} {};
	\node[coordinate] (10) at \hexcoor{0}{2.5}{3.5} {};
	\node[coordinate] (11) at \hexcoor{0}{1.5}{4.5} {};
	\node[coordinate] (12) at \hexcoor{0}{0.5}{5.5} {};
	\node[coordinate] (13) at \hexcoor{0.5}{0}{5.5} {};
	\node[coordinate] (14) at \hexcoor{1.5}{0}{4.5} {};
	\node[coordinate] (15) at \hexcoor{2.5}{0}{3.5} {};
	\node[coordinate] (16) at \hexcoor{3.5}{0}{2.5} {};
	\node[coordinate] (17) at \hexcoor{4.5}{0}{1.5} {};
	\node[coordinate] (18) at \hexcoor{5.5}{0}{0.5} {};

	\node[dot,fill=white] (500) at \hexcoor{5}{0}{0} {};
	\node[dot,fill=white] (410) at \hexcoor{4}{1}{0} {};
	\node[dot,fill=white] (401) at \hexcoor{4}{0}{1} {};
	\node[dot,fill=white] (320) at \hexcoor{3}{2}{0} {};
	\node[dot,fill=white] (311) at \hexcoor{3}{1}{1} {};
	\node[dot,fill=white] (302) at \hexcoor{3}{0}{2} {};
	\node[dot,fill=white] (230) at \hexcoor{2}{3}{0} {};
	\node[dot,fill=white] (221) at \hexcoor{2}{2}{1} {};
	\node[dot,fill=white] (212) at \hexcoor{2}{1}{2} {};
	\node[dot,fill=white] (203) at \hexcoor{2}{0}{3} {};
	\node[dot,fill=white] (140) at \hexcoor{1}{4}{0} {};
	\node[dot,fill=white] (131) at \hexcoor{1}{3}{1} {};
	\node[dot,fill=white] (122) at \hexcoor{1}{2}{2} {};
	\node[dot,fill=white] (113) at \hexcoor{1}{1}{3} {};
	\node[dot,fill=white] (104) at \hexcoor{1}{0}{4} {};
	\node[dot,fill=white] (050) at \hexcoor{0}{5}{0} {};
	\node[dot,fill=white] (041) at \hexcoor{0}{4}{1} {};
	\node[dot,fill=white] (032) at \hexcoor{0}{3}{2} {};
	\node[dot,fill=white] (023) at \hexcoor{0}{2}{3} {};
	\node[dot,fill=white] (014) at \hexcoor{0}{1}{4} {};
	\node[dot,fill=white] (005) at \hexcoor{0}{0}{5} {};
	\node[dot,fill=black!50] (400) at \hexcoor{4}{0}{0} {};
	\node[dot,fill=black!50] (310) at \hexcoor{3}{1}{0} {};
	\node[dot,fill=black!50] (301) at \hexcoor{3}{0}{1} {};
	\node[dot,fill=black!50] (220) at \hexcoor{2}{2}{0} {};
	\node[dot,fill=black!50] (211) at \hexcoor{2}{1}{1} {};
	\node[dot,fill=black!50] (202) at \hexcoor{2}{0}{2} {};
	\node[dot,fill=black!50] (130) at \hexcoor{1}{3}{0} {};
	\node[dot,fill=black!50] (121) at \hexcoor{1}{2}{1} {};
	\node[dot,fill=black!50] (112) at \hexcoor{1}{1}{2} {};
	\node[dot,fill=black!50] (103) at \hexcoor{1}{0}{3} {};
	\node[dot,fill=black!50] (040) at \hexcoor{0}{4}{0} {};
	\node[dot,fill=black!50] (031) at \hexcoor{0}{3}{1} {};
	\node[dot,fill=black!50] (022) at \hexcoor{0}{2}{2} {};
	\node[dot,fill=black!50] (013) at \hexcoor{0}{1}{3} {};
	\node[dot,fill=black!50] (004) at \hexcoor{0}{0}{4} {};
		
	\draw (1) to (500) to (400) to (410);
	\draw (2) to (410) to (310) to (320);
	\draw (3) to (320) to (220) to (230);
	\draw (4) to (230) to (130) to (140);
	\draw (5) to (140) to (040) to (050) to (6);
	\draw (400) to (401) to (301) to (311);
	\draw (310) to (311) to (211) to (221);
	\draw (220) to (221) to (121) to (131);
	\draw (130) to (131) to (031) to (041) to (040);
	\draw (301) to (302) to (202) to (212);
	\draw (211) to (212) to (112) to (122);
	\draw (121) to (122) to (022) to (032) to (031);
	\draw (202) to (203) to (103) to (113);
	\draw (112) to (113) to (013) to (023) to (022);
	\draw (103) to (104) to (004) to (014) to (013);
	\draw (004) to (005) to (13);
	\draw (7) to (050);
	\draw (8) to (041);
	\draw (9) to (032);
	\draw (10) to (023);
	\draw (11) to (014);
	\draw (12) to (005);
	\draw (14) to (104);
	\draw (15) to (203);
	\draw (16) to (302);
	\draw (17) to (401);
	\draw (18) to (500);
	
	\draw[matching,red] (1) to (500);
	\draw[matching,red] (2) to (410);
	\draw[matching,red] (3) to (320);
	\draw[matching,red] (400) to (401);
	\draw[matching,red] (310) to (311);
	\draw[matching,red] (301) to (302);
	\draw[matching,red] (7) to (050);
	\draw[matching,red] (8) to (041);
	\draw[matching,red] (040) to (140);
	\draw[matching,red] (13) to (005);
	
	\draw[matching] (203) to (202);
	\draw[matching] (212) to (211);
	\draw[matching] (221) to (121);
	\draw[matching] (230) to (220);
	\draw[matching] (104) to (103);
	\draw[matching] (113) to (013);
	\draw[matching] (122) to (112);
	\draw[matching] (131) to (130);
	\draw[matching] (014) to (004);
	\draw[matching] (023) to (022);
	\draw[matching] (032) to (031);
\end{tikzpicture}
\caption{A matching (red edges forced by boundary)} 
\end{subfigure}
\begin{subfigure}[b]{.4\textwidth}
\begin{tikzpicture}[scale=.75,baseline=(current bounding box.center)]
	\node[coordinate] (303) at \hexcoor{3}{0}{3} {};	
	\node[coordinate] (312) at \hexcoor{3}{1}{2} {};	
	\node[coordinate] (321) at \hexcoor{3}{2}{1} {};	
	\node[coordinate] (330) at \hexcoor{3}{3}{0} {};	
	\node[coordinate] (204) at \hexcoor{2}{0}{4} {};	
	\node[coordinate] (213) at \hexcoor{2}{1}{3} {};	
	\node[coordinate] (222) at \hexcoor{2}{2}{2} {};	
	\node[coordinate] (231) at \hexcoor{2}{3}{1} {};	
	\node[coordinate] (240) at \hexcoor{2}{4}{0} {};	
	\node[coordinate] (105) at \hexcoor{1}{0}{5} {};	
	\node[coordinate] (114) at \hexcoor{1}{1}{4} {};	
	\node[coordinate] (123) at \hexcoor{1}{2}{3} {};	
	\node[coordinate] (132) at \hexcoor{1}{3}{2} {};	
	\node[coordinate] (141) at \hexcoor{1}{4}{1} {};	
	\node[coordinate] (015) at \hexcoor{0}{1}{5} {};	
	\node[coordinate] (024) at \hexcoor{0}{2}{4} {};	
	\node[coordinate] (033) at \hexcoor{0}{3}{3} {};	
	\node[coordinate] (042) at \hexcoor{0}{4}{2} {};	
	
	\draw[fill=blue!10] (303) to (312) to (213) to (204) to (303);
	\draw[fill=blue!10] (312) to (321) to (222) to (213) to (312);
	\draw[fill=blue!50] (321) to (231) to (132) to (222) to (321);
	\draw[fill=blue!30] (321) to (330) to (240) to (231) to (321);
	\draw[fill=blue!10] (204) to (213) to (114) to (105) to (204);
	\draw[fill=blue!50] (213) to (123) to (024) to (114) to (213);
	\draw[fill=blue!30] (213) to (222) to (132) to (123) to (213);
	\draw[fill=blue!10] (231) to (240) to (141) to (132) to (231);
	\draw[fill=blue!30] (105) to (114) to (024) to (015) to (105);
	\draw[fill=blue!10] (123) to (132) to (033) to (024) to (123);
	\draw[fill=blue!10] (132) to (141) to (042) to (033) to (132);

\end{tikzpicture}
\caption{The corresponding rhombus tiling}
\end{subfigure}
\caption{An example of the correspondence between matchings with boundary $(1,2,3,7,8,13)$ and rhombus tilings of the $(3,2,1,3,2,1)$ hexagon}
\label{Hexagon}
\end{figure}

\section{The lattice structure on matchings} \label{sec Propp}

The set of matchings of $G$ has a natural partial ordering, which makes the set of matchings with a fixed boundary into a combinatorial lattice.  As a consequence, these sets have unique minimal and maximal elements.  In this appendix, we demonstrate that the matchings $\vecM(f)$ and $\cevM(f)$ can be described in terms of this partial order, without reference to strands. Specifically, $\vecM(f)$ is the unique minimal matching with boundary $\sI(f)$, and $\cevM(f)$ is the unique maximal matching with boundary $\tI(f)$.\footnote{The boundaries of $\vecM(f)$ and $\cevM(f)$ are distinct, and so these are extremal elements in \emph{different} posets.}

\subsection{Lattice structure on matchings} \label{sec lattice}


Let $G$ be a reduced graph and let $M$ be a matching of $G$.
Let $f$ be an internal face of $G$ such that $M$ contains exactly half the edges in the boundary of $f$, the most possible.  The \newword{swivel} of $M$ at $f$ is the matching $M'$ which contains the other half of the edges in the boundary of $f$ and is otherwise the same as $M$.\footnote{Propp uses the word ``twist" rather than ``swivel", but that word has another meaning for us.}  The new matching $M'$ also has boundary $\partial M'=\partial M=I$.\footnote{Note that, by Lemma~\ref{clean dual}, there are no topological subtleties in defining the boundary of an internal face.}

\begin{figure}[h!t]
\centering
\begin{tikzpicture}
	\node (left) at (0,0) {
\begin{tikzpicture}[scale=.55]
		\path[use as  bounding box] (-4.5,-4.5) rectangle (4.5,4.5);
		\draw[fill=black!10] (0,0) circle (4);
		\node[invisible] (1) at (180:4) {};
		\node[invisible] (2) at (120:4) {};
		\node[invisible] (3) at (60:4) {};
		\node[invisible] (4) at (0:4) {};
		\node[invisible] (5) at (-60:4) {};
		\node[invisible] (6) at (-120:4) {};
		
		\node[left] at (1) {$5$};
		\node[above left] at (2) {$6$};
		\node[above right] at (3) {$1$};
		\node[right] at (4) {$2$};
		\node[below right] at (5) {$3$};
		\node[below left] at (6) {$4$};
		\node at (0,-.65) {$F$};
		
		\node[dot, fill=white] (a) at (-2.75,.25) {};
		\node[dot, fill=black!50] (b) at (-1.5,2) {};
		\node[dot, fill=white] (c) at (1.5,2) {};
		\node[dot, fill=black!50] (d) at (2.75,.25) {};
		\node[dot, fill=black!50] (e) at (-1.5,-.8) {};
		\node[dot, fill=white] (f) at (-.45,.75) {};
		\node[dot, fill=black!50] (g) at (.45,.75) {};
		\node[dot, fill=white] (h) at (1.5,-.8) {};
		\node[dot, fill=black!50] (i) at (.5,-2) {};
		\node[dot, fill=white] (j) at (-.5,-2) {};
		
		\node[dot,fill=white] (b') at (-1.75,2.75) {};
		\node[dot,fill=white] (d') at (3.4,.125) {};
		\node[dot,fill=white] (i') at (1.25,-2.75) {};
		
		\draw (1) to (a);
		\draw[matching] (a) to (b);
		\draw (b) to (b');
		\draw[matching] (b') to (2);
		\draw (b) to (f);
		\draw[matching] (f) to (g);
		\draw (g) to (c) to (3);
		\draw[matching] (c) to (d);
		\draw (d) to (d');
		\draw[matching] (d') to (4);
		\draw (d) to (h);
		\draw[matching] (h) to (i);
		\draw (i) to (i');
		\draw[matching] (i') to (5);
		\draw (i) to (j);
		\draw (j) to (6);
		\draw[matching] (j) to (e);
		\draw (e) to (a);
		\draw (e) to (f);
		\draw (g) to (h);	
\end{tikzpicture}};
	\node (right) at (3in,0) {
\begin{tikzpicture}[xshift=3in,scale=.55]
		\path[use as  bounding box] (-4.5,-4.5) rectangle (4.5,4.5);
		\draw[fill=black!10] (0,0) circle (4);
		\node[invisible] (1) at (180:4) {};
		\node[invisible] (2) at (120:4) {};
		\node[invisible] (3) at (60:4) {};
		\node[invisible] (4) at (0:4) {};
		\node[invisible] (5) at (-60:4) {};
		\node[invisible] (6) at (-120:4) {};
		
		\node[left] at (1) {$5$};
		\node[above left] at (2) {$6$};
		\node[above right] at (3) {$1$};
		\node[right] at (4) {$2$};
		\node[below right] at (5) {$3$};
		\node[below left] at (6) {$4$};
		\node at (0,-.65) {$F$};
		
		\node[dot, fill=white] (a) at (-2.75,.25) {};
		\node[dot, fill=black!50] (b) at (-1.5,2) {};
		\node[dot, fill=white] (c) at (1.5,2) {};
		\node[dot, fill=black!50] (d) at (2.75,.25) {};
		\node[dot, fill=black!50] (e) at (-1.5,-.8) {};
		\node[dot, fill=white] (f) at (-.45,.75) {};
		\node[dot, fill=black!50] (g) at (.45,.75) {};
		\node[dot, fill=white] (h) at (1.5,-.8) {};
		\node[dot, fill=black!50] (i) at (.5,-2) {};
		\node[dot, fill=white] (j) at (-.5,-2) {};
		
		\node[dot,fill=white] (b') at (-1.75,2.75) {};
		\node[dot,fill=white] (d') at (3.4,.125) {};
		\node[dot,fill=white] (i') at (1.25,-2.75) {};
		
		\draw (1) to (a);
		\draw[matching] (a) to (b);
		\draw (b) to (b');
		\draw[matching] (b') to (2);
		\draw (b) to (f);
		\draw[] (f) to (g);
		\draw (g) to (c) to (3);
		\draw[matching] (c) to (d);
		\draw (d) to (d');
		\draw[matching] (d') to (4);
		\draw (d) to (h);
		\draw[] (h) to (i);
		\draw (i) to (i');
		\draw[matching] (i') to (5);
		\draw[matching] (i) to (j);
		\draw (j) to (6);
		\draw[] (j) to (e);
		\draw (e) to (a);
		\draw[matching] (e) to (f);
		\draw[matching] (g) to (h);	
\end{tikzpicture}};
	\draw[out=15,in=165,-angle 90] (left) to node[above] {Swiveling up} (right);
	\draw[out=-165,in=-15,-angle 90] (right) to node[below] {Swiveling down} (left);
\end{tikzpicture}
\caption{Swiveling up and down at the face $F$}
\label{fig: swiveling}
\end{figure}
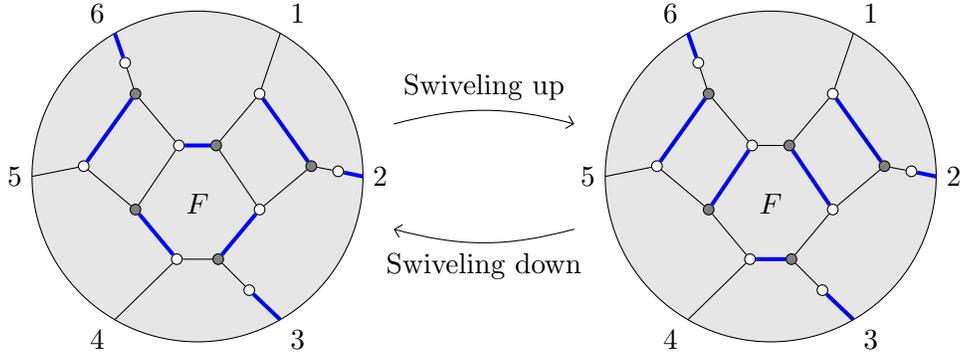

Swiveling twice at the same internal face returns to the original matching, but we may use the orientation of the face $f$ and the coloring of the vertices to distinguish between \newword{swiveling up} and \newword{swiveling down}, as in Figure \ref{fig: swiveling}.  We may extend this to a partial ordering $\preceq$ on the set of matchings with boundary $I$, where $M_1\preceq M_2$ means that $M_2$ can be obtained from $M_1$ by repeatedly swiveling up. An example is given in Figure \ref{fig: poset}. (It is true, though not obvious, that it is impossible to swivel up repeatedly and return to the original matching.)

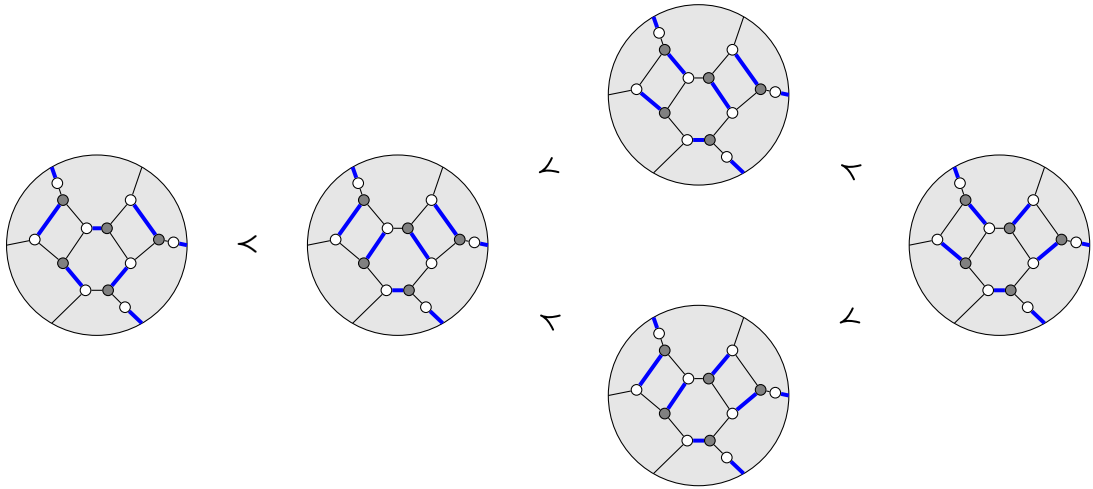
\begin{figure}[h!t]
\begin{tikzpicture}
	\node (a1) at (0,0) {
\begin{tikzpicture}[scale=.3]
		\path[use as  bounding box] (-4.5,-4.5) rectangle (4.5,4.5);
		\draw[fill=black!10] (0,0) circle (4);
		\node[invisible] (1) at (180:4) {};
		\node[invisible] (2) at (120:4) {};
		\node[invisible] (3) at (60:4) {};
		\node[invisible] (4) at (0:4) {};
		\node[invisible] (5) at (-60:4) {};
		\node[invisible] (6) at (-120:4) {};
		
		\node[dot, fill=white] (a) at (-2.75,.25) {};
		\node[dot, fill=black!50] (b) at (-1.5,2) {};
		\node[dot, fill=white] (c) at (1.5,2) {};
		\node[dot, fill=black!50] (d) at (2.75,.25) {};
		\node[dot, fill=black!50] (e) at (-1.5,-.8) {};
		\node[dot, fill=white] (f) at (-.45,.75) {};
		\node[dot, fill=black!50] (g) at (.45,.75) {};
		\node[dot, fill=white] (h) at (1.5,-.8) {};
		\node[dot, fill=black!50] (i) at (.5,-2) {};
		\node[dot, fill=white] (j) at (-.5,-2) {};
		
		\node[dot,fill=white] (b') at (-1.75,2.75) {};
		\node[dot,fill=white] (d') at (3.4,.125) {};
		\node[dot,fill=white] (i') at (1.25,-2.75) {};
		
		\draw (1) to (a);
		\draw[matching] (a) to (b);
		\draw (b) to (b');
		\draw[matching] (b') to (2);
		\draw (b) to (f);
		\draw[matching] (f) to (g);
		\draw (g) to (c) to (3);
		\draw[matching] (c) to (d);
		\draw (d) to (d');
		\draw[matching] (d') to (4);
		\draw (d) to (h);
		\draw[matching] (h) to (i);
		\draw (i) to (i');
		\draw[matching] (i') to (5);
		\draw (i) to (j);
		\draw (j) to (6);
		\draw[matching] (j) to (e);
		\draw (e) to (a);
		\draw (e) to (f);
		\draw (g) to (h);	
\end{tikzpicture}};
	\node (a2) at (4,0) {
\begin{tikzpicture}[scale=.3]
		\path[use as  bounding box] (-4.5,-4.5) rectangle (4.5,4.5);
		\draw[fill=black!10] (0,0) circle (4);
		\node[invisible] (1) at (180:4) {};
		\node[invisible] (2) at (120:4) {};
		\node[invisible] (3) at (60:4) {};
		\node[invisible] (4) at (0:4) {};
		\node[invisible] (5) at (-60:4) {};
		\node[invisible] (6) at (-120:4) {};
		
		\node[dot, fill=white] (a) at (-2.75,.25) {};
		\node[dot, fill=black!50] (b) at (-1.5,2) {};
		\node[dot, fill=white] (c) at (1.5,2) {};
		\node[dot, fill=black!50] (d) at (2.75,.25) {};
		\node[dot, fill=black!50] (e) at (-1.5,-.8) {};
		\node[dot, fill=white] (f) at (-.45,.75) {};
		\node[dot, fill=black!50] (g) at (.45,.75) {};
		\node[dot, fill=white] (h) at (1.5,-.8) {};
		\node[dot, fill=black!50] (i) at (.5,-2) {};
		\node[dot, fill=white] (j) at (-.5,-2) {};
		
		\node[dot,fill=white] (b') at (-1.75,2.75) {};
		\node[dot,fill=white] (d') at (3.4,.125) {};
		\node[dot,fill=white] (i') at (1.25,-2.75) {};
		
		\draw (1) to (a);
		\draw[matching] (a) to (b);
		\draw (b) to (b');
		\draw[matching] (b') to (2);
		\draw (b) to (f);
		\draw[] (f) to (g);
		\draw (g) to (c) to (3);
		\draw[matching] (c) to (d);
		\draw (d) to (d');
		\draw[matching] (d') to (4);
		\draw (d) to (h);
		\draw[] (h) to (i);
		\draw (i) to (i');
		\draw[matching] (i') to (5);
		\draw[matching] (i) to (j);
		\draw (j) to (6);
		\draw[] (j) to (e);
		\draw (e) to (a);
		\draw[matching] (e) to (f);
		\draw[matching] (g) to (h);	
\end{tikzpicture}};
	\node (u3) at (8,2) {
\begin{tikzpicture}[scale=.3]
		\path[use as  bounding box] (-4.5,-4.5) rectangle (4.5,4.5);
		\draw[fill=black!10] (0,0) circle (4);
		\node[invisible] (1) at (180:4) {};
		\node[invisible] (2) at (120:4) {};
		\node[invisible] (3) at (60:4) {};
		\node[invisible] (4) at (0:4) {};
		\node[invisible] (5) at (-60:4) {};
		\node[invisible] (6) at (-120:4) {};
		
		\node[dot, fill=white] (a) at (-2.75,.25) {};
		\node[dot, fill=black!50] (b) at (-1.5,2) {};
		\node[dot, fill=white] (c) at (1.5,2) {};
		\node[dot, fill=black!50] (d) at (2.75,.25) {};
		\node[dot, fill=black!50] (e) at (-1.5,-.8) {};
		\node[dot, fill=white] (f) at (-.45,.75) {};
		\node[dot, fill=black!50] (g) at (.45,.75) {};
		\node[dot, fill=white] (h) at (1.5,-.8) {};
		\node[dot, fill=black!50] (i) at (.5,-2) {};
		\node[dot, fill=white] (j) at (-.5,-2) {};
		
		\node[dot,fill=white] (b') at (-1.75,2.75) {};
		\node[dot,fill=white] (d') at (3.4,.125) {};
		\node[dot,fill=white] (i') at (1.25,-2.75) {};
		
		\draw (1) to (a);
		\draw[] (a) to (b);
		\draw (b) to (b');
		\draw[matching] (b') to (2);
		\draw[matching] (b) to (f);
		\draw[] (f) to (g);
		\draw (g) to (c) to (3);
		\draw[matching] (c) to (d);
		\draw (d) to (d');
		\draw[matching] (d') to (4);
		\draw (d) to (h);
		\draw[] (h) to (i);
		\draw (i) to (i');
		\draw[matching] (i') to (5);
		\draw[matching] (i) to (j);
		\draw (j) to (6);
		\draw[] (j) to (e);
		\draw[matching] (e) to (a);
		\draw[] (e) to (f);
		\draw[matching] (g) to (h);	
\end{tikzpicture}};
	\node (d3) at (8,-2) {
\begin{tikzpicture}[scale=.3]
		\path[use as  bounding box] (-4.5,-4.5) rectangle (4.5,4.5);
		\draw[fill=black!10] (0,0) circle (4);
		\node[invisible] (1) at (180:4) {};
		\node[invisible] (2) at (120:4) {};
		\node[invisible] (3) at (60:4) {};
		\node[invisible] (4) at (0:4) {};
		\node[invisible] (5) at (-60:4) {};
		\node[invisible] (6) at (-120:4) {};
		
		\node[dot, fill=white] (a) at (-2.75,.25) {};
		\node[dot, fill=black!50] (b) at (-1.5,2) {};
		\node[dot, fill=white] (c) at (1.5,2) {};
		\node[dot, fill=black!50] (d) at (2.75,.25) {};
		\node[dot, fill=black!50] (e) at (-1.5,-.8) {};
		\node[dot, fill=white] (f) at (-.45,.75) {};
		\node[dot, fill=black!50] (g) at (.45,.75) {};
		\node[dot, fill=white] (h) at (1.5,-.8) {};
		\node[dot, fill=black!50] (i) at (.5,-2) {};
		\node[dot, fill=white] (j) at (-.5,-2) {};
		
		\node[dot,fill=white] (b') at (-1.75,2.75) {};
		\node[dot,fill=white] (d') at (3.4,.125) {};
		\node[dot,fill=white] (i') at (1.25,-2.75) {};
		
		\draw (1) to (a);
		\draw[matching] (a) to (b);
		\draw (b) to (b');
		\draw[matching] (b') to (2);
		\draw (b) to (f);
		\draw[] (f) to (g);
		\draw[matching] (g) to (c);
		\draw (c) to (3);
		\draw[] (c) to (d);
		\draw (d) to (d');
		\draw[matching] (d') to (4);
		\draw[matching] (d) to (h);
		\draw[] (h) to (i);
		\draw (i) to (i');
		\draw[matching] (i') to (5);
		\draw[matching] (i) to (j);
		\draw (j) to (6);
		\draw[] (j) to (e);
		\draw (e) to (a);
		\draw[matching] (e) to (f);
		\draw[] (g) to (h);	
\end{tikzpicture}};
	\node (a4) at (12,0) {
\begin{tikzpicture}[scale=.3]
		\path[use as  bounding box] (-4.5,-4.5) rectangle (4.5,4.5);
		\draw[fill=black!10] (0,0) circle (4);
		\node[invisible] (1) at (180:4) {};
		\node[invisible] (2) at (120:4) {};
		\node[invisible] (3) at (60:4) {};
		\node[invisible] (4) at (0:4) {};
		\node[invisible] (5) at (-60:4) {};
		\node[invisible] (6) at (-120:4) {};
		
		\node[dot, fill=white] (a) at (-2.75,.25) {};
		\node[dot, fill=black!50] (b) at (-1.5,2) {};
		\node[dot, fill=white] (c) at (1.5,2) {};
		\node[dot, fill=black!50] (d) at (2.75,.25) {};
		\node[dot, fill=black!50] (e) at (-1.5,-.8) {};
		\node[dot, fill=white] (f) at (-.45,.75) {};
		\node[dot, fill=black!50] (g) at (.45,.75) {};
		\node[dot, fill=white] (h) at (1.5,-.8) {};
		\node[dot, fill=black!50] (i) at (.5,-2) {};
		\node[dot, fill=white] (j) at (-.5,-2) {};
		
		\node[dot,fill=white] (b') at (-1.75,2.75) {};
		\node[dot,fill=white] (d') at (3.4,.125) {};
		\node[dot,fill=white] (i') at (1.25,-2.75) {};
		
		\draw (1) to (a);
		\draw[] (a) to (b);
		\draw (b) to (b');
		\draw[matching] (b') to (2);
		\draw[matching] (b) to (f);
		\draw[] (f) to (g);
		\draw[matching] (g) to (c);
		\draw (c) to (3);
		\draw[] (c) to (d);
		\draw (d) to (d');
		\draw[matching] (d') to (4);
		\draw[matching] (d) to (h);
		\draw[] (h) to (i);
		\draw (i) to (i');
		\draw[matching] (i') to (5);
		\draw[matching] (i) to (j);
		\draw (j) to (6);
		\draw[] (j) to (e);
		\draw[matching] (e) to (a);
		\draw[] (e) to (f);
		\draw[] (g) to (h);	
\end{tikzpicture}};
	\path (a1) to node {\scalebox{1.5}{$\prec$}} (a2);
	\path (a2) to node[rotate=25] {\scalebox{1.5}{$\prec$}} (u3);
	\path (a2) to node[rotate=-25] {\scalebox{1.5}{$\prec$}} (d3);
	\path (u3) to node[rotate=-25] {\scalebox{1.5}{$\prec$}} (a4);
	\path (d3) to node[rotate=25] {\scalebox{1.5}{$\prec$}} (a4);
\end{tikzpicture}
\caption{The poset of matchings with boundary $236$}
\label{fig: poset}
\end{figure}

\begin{thm} \label{lattice}
Let $G$ be a reduced graph, and let $I$ be a matchable subset of the $[n]$
Then the partial ordering $\preceq$ makes the set of matchings on $G$ with boundary $I$ into a finite distributive lattice.
\end{thm}
\noindent We will deduce this result from a similar result of Propp, which we now describe.

Let $\Gamma$ be a planar graph embedded in the two-sphere $S^2$, so that all the faces of $S^2 \setminus \Gamma$ are discs and no edge separates a face from itself. We designate one face $F_{\infty}$ to play a special role.
Let $d$ be a function from the vertices of $\Gamma$ to the positive integers. A \newword{$d$-factor} of $\Gamma$ is a set $M$ of edges such that, for each vertex $v$ of $\Gamma$, there are precisely $d(v)$ edges of $M$ containing $v$.
So, if $d$ is identically one, then a $d$-factor is a perfect matching.
As with matchings, we can define upward and downward swivels taking $d$-factors to other $d$-factors; we do not permit swivels around $F_{\infty}$. 
Again, we define $M_1 \preceq M_2$ if we can obtain $M_2$ from $M_1$ by repeated upward swivels. 

\begin{Theorem}[{\cite[Theorem 2]{Pro02}}] \label{Propp lattice}
Let $\Gamma$, $d$ and $F_{\infty}$ be as above. Assume the following condition:
\begin{quote}
\textbf{Condition $(\ast)$:} For every edge $e$ of $\Gamma$, there is some $d$-factor containing $e$ and some other $d$-factor omitting $e$. 
\end{quote}
Then the partial order $\preceq$ is a finite distributive lattice. 
\end{Theorem}

One might hope to prove Theorem~\ref{lattice} from Theorem~\ref{Propp lattice} by deleting certain boundary vertices from $G$ in order to make a graph $\Gamma$ whose matchings correspond to the matchings of $G$ with boundary $\Gamma$. Unfortunately, if we do this in the obvious  way, condition~$(\ast)$ fails. 
We therefore take a different route.

\begin{proof}[Proof of Theorem~\ref{lattice}]
We may assume that every boundary vertex $i$ of $G$ is used in some matching and not used in some other matching. Otherwise, the vertex $i$ lies in a component disconnected from the rest of $G$ and we can delete that component and study the remaining graph. We may also delete any lollipops in the graph, as the corresponding edge is either in every matching or no matching with boundary $I$.

Applying the move from Figure~\ref{fig: boundaryvert} repeatedly, we may assume that all the boundary vertices of $G$ border white vertices.
Now remove the boundary vertices and replace them by one black vertex $v_{\infty}$, which we connect to all of the white vertices which used to border boundary vertices.
Call the resulting graph $\Gamma$; we embed it in $S^2$ in the obvious manner. We choose $F_{\infty}$ to be the face which contains the vertices $1$, $n$, and $v_{\infty}$.
Lemma~\ref{clean dual} implies that all faces of $S^2 \setminus \Gamma$ are discs and no edge separates a face from itself.

Let $d$ be the function which is $1$ on every vertex of $\Gamma$ other than $v_{\infty}$, and $k$ at $v_{\infty}$.
It is straightforward to see that $d$-factors of $\Gamma$ correspond to matchings of $G$.
Also, we claim that every edge $e$ of $\Gamma$ is in some $d$-factor but not in some other $d$-factor. For the edges from $v_{\infty}$, this follows from the reduction in the first paragraph.
For an edge $e$ not adjacent to the boundary of $G$, if $e$ is not used in any matching, then we can delete $G$ from $e$ and obtain a graph with the same boundaries of matchings; by Lemma~\ref{clean dual}, this will merge two faces of $G$, contradicting that $G$ is reduced. If $e$ is used in every matching, then we can likewise delete $e$ and all edges with an endpoint in common with $e$. So the hypotheses of Propp's result apply, and we obtain a lattice structure on the set of $d$-factors of $\Gamma$.

Let $\Lambda$ be this lattice with meet and join operations $\vee$ and $\wedge$. Let $\partial: \Lambda \to \binom{[n]}{k}$ send a $d$-factor of $\Gamma$ to the boundary of the corresponding matching of $G$. Here is our key claim: If $\partial(M_1) = \partial(M_2)=I$, then $\partial(M_1 \vee M_2) = \partial(M_1 \wedge M_2) = I$. 

To prove this, we have to enter the proof of Propp's Theorem~\ref{Propp lattice}. 
Propp defines a correspondence between $d$-factors $M$ of $\Gamma$ and certain real valued height functions $h_M$ on the faces of $\Gamma$. Let $e$ be an edge of $\Gamma$ incident to $v_{\infty}$ and let $F$ and $F'$ be the faces separated by $e$. Then there is some number $0 < \delta < 1$ such that $h_M(F) - h_M(F') = \delta$ if $e \in M$ and $=\delta-1$ otherwise.  We have $h_{M_1 \vee M_2}(F) = \max(h_{M_1}(F), h_{M_2}(F))$ and $h_{M_1 \wedge M_2}(F) = \min(h_{M_1}(F), h_{M_2}(F))$. Moreover, $h_{M_1}(F)-h_{M_2}(F)$ and $h_{M_1}(F') - h_{M_2}(F')$ are integers. It follows from these formulas that that, if $e$ is in both $M_1$ and $M_2$, then it is in $M_1 \vee M_2$ and $M_1 \wedge M_2$ and, if $e$ is in neither $M_1$ nor $M_2$, then it is also not in $M_1 \vee M_2$ or $M_1 \wedge M_2$. In particular, our key claim holds.

So the subset of $\Lambda$ with boundary $I$ is closed under $\vee$ and $\wedge$. Restricting the operations of $\Lambda$ to this subset, we have a finite distributive lattice as claimed.
\end{proof}

\begin{cor} \label{max and min}
The set of matchings of $G$ with boundary $I$ has a unique $\preceq$-minimal element and unique $\preceq$-maximal element, assuming the set is non-empty.
\end{cor}

\begin{proof}
A lattice has a unique minimal element and unique maximal element.
\end{proof}

\begin{cor}
Any two matchings  of $G$ with boundary $I$ are related by a sequence of swivels.
\end{cor}

\begin{proof}
They may both be swiveled up to the maximal matching.
\end{proof}

\begin{remark}
This proof leads to some results about positroids which appear to be new.
Place a partial order $\preceq$ on $\binom{[n]}{k}$ by $(i_1, i_2, \ldots, i_k) \leq (j_1, \ldots, j_k)$ if and only if $i_a \leq j_a$ for all $a$.
It is well known that $\preceq$ is a distributive lattice, with $\wedge$ and $\vee$ given by termwise $\min$ and $\max$. 
One can show that $\partial : \Lambda \to \binom{[n]}{k}$ obeys $\partial(M_1 \vee M_2) = \partial(M_1) \vee \partial(M_2)$ and  $\partial(M_1 \wedge M_2) = \partial(M_1) \wedge \partial(M_2)$. We therefore obtain the following corollary: The set of bases of the positroid $\cM$ is closed under termwise $\min$ and $\max$. Also, assume that $G$ is connected, which is the same as assuming that the positroid is connected as a matroid. Then upward swivels around the faces of $\Gamma$ incident to $v_{\infty}$ change the boundary by turning $i$ into $i+1$. We deduce that it is possible to turn any basis of $\cM$ into the maximal basis by repeatedly replacing $i$ by $i+1$.
\end{remark}

\subsection{Extremal matchings}

We connect this lattice structure to the matchings $\vecM(f)$ and $\cevM(f)$.

\begin{prop}
For any face $f\in F$, the matching $\vecM(f)$ is the minimal matching of $G$ with boundary $\sI(f)$.  
\end{prop}
\begin{proof}
If $f'\in F$ is an internal face with $f'\neq f$, then $\vecM(f)$ contains one fewer than half the edges in the boundary of $f'$; hence, $\vecM(f)$ cannot be swiveled at $f'$.  If $f$ is internal, then $\vecM(f)$ contains those edges $e$ in the boundary of $f$ such that $f$ is directly downstream from $e$.  Consulting to Figure \ref{fig: swiveling}, we see that swiveling $\vecM(f)$ at $f$ is always increasing for $\preceq$.  Hence, $\vecM(f)$ is minimal for $\preceq$.
\end{proof}


A corollary of this result is an alternate proof of Proposition \ref{prop: boundaryunique}.

\begin{corollary} \label{boundary unique}
For a boundary face $f$, $\vecM(f)$ is the unique matching of $G$ with boundary $\sI(f)$.
\end{corollary}
\begin{proof}
The matching $\vecM(f)$ does not contain enough edges around any internal face to swivel.  Since any matching with boundary $\sI(f)$ is obtained from a sequences of swivels of $\vecM(f)$, it is unique.
\end{proof}

%
%
%

\bibliography{MyNewBib}{}
\bibliographystyle{halpha}

\end{document}